\documentclass[12pt]{article}
\usepackage[dvipsnames,svgnames,table]{xcolor} 

\usepackage{enumitem} 
\usepackage{rotfloat} 
\usepackage{graphicx}
\usepackage{epsfig}
\usepackage{amssymb, amsmath, mathtools}

\usepackage{verbatim}
\usepackage{natbib}
\usepackage{authblk}
\usepackage{kotex}
\usepackage{enumitem}
\usepackage{multirow}
\usepackage{booktabs} 
\usepackage{adjustbox}
\usepackage{subcaption} 
\usepackage{algorithm}
\usepackage{algpseudocode}
\usepackage{bm}
\usepackage{multirow}
\usepackage{caption}
\usepackage{url}
\usepackage{lineno}
\usepackage{import}
\usepackage{amsthm}
\usepackage{float}
\usepackage{xcolor}
\usepackage[colorlinks=true,linkcolor=blue,citecolor=blue,urlcolor=blue]{hyperref}%

\newtheorem{theorem}{Theorem}[section]
\newtheorem{lemma}[theorem]{Lemma}
\newtheorem{definition}[theorem]{Definition}
\newtheorem{proposition}[theorem]{Proposition}
\newtheorem{corollary}[theorem]{Corollary}

\newtheorem{remark}[theorem]{Remark}

\newtheorem{assumption}{Assumption}

\everymath{\displaystyle}

\newcommand{\lra}{\longrightarrow}

\newcommand{\be}{\begin{equation}}
\newcommand{\ee}{\end{equation}}
\newcommand{\bea}{\begin{eqnarray*}}
\newcommand{\eea}{\end{eqnarray*}}
\newcommand{\bean}{\begin{eqnarray}}
\newcommand{\eean}{\end{eqnarray}}
\newcommand{\ben}{\begin{enumerate}}
\newcommand{\een}{\end{enumerate}}
\newcommand{\bi}{\begin{itemize}}
\newcommand{\ei}{\end{itemize}}
\newcommand{\brem}{\begin{remark}}
\newcommand{\erem}{\end{remark}}
\newcommand{\bcen}{\begin{center}}
\newcommand{\ecen}{\end{center}}
\newcommand{\bsv}{\begin{semiverbatim}}
\newcommand{\esv}{\end{semiverbatim}}

\newcommand{\bt}{\begin{theorem}}
\newcommand{\et}{\end{theorem}}
\newcommand{\bl}{\begin{lemma}}
\newcommand{\el}{\end{lemma}}
\newcommand{\bd}{\begin{definition}}
\newcommand{\ed}{\end{definition}}
\newcommand{\bc}{\begin{corollary}}
\newcommand{\ec}{\end{corollary}}
\newcommand{\bp}{\begin{proposition}}
\newcommand{\ep}{\end{proposition}}

\newcommand{\bbC}{{ \mathbb{C}}}

\newcommand{\bbR}{ \mathbb{R}}
\newcommand{\bbX}{ \mathbb{X}}
 
\newcommand{\bbZ}{ \mathbb{Z}}

\newcommand{\calC}{\mathcal{C}}

\newcommand{\pkg}[1]{\texttt{#1}}

\begin{document}

\title{Bayesian Analysis of Spiked Covariance Models: Correcting Eigenvalue Bias and Determining the Number of Spikes}
\author[1]{Kwangmin Lee}
\author[2]{Sewon Park}
\author[3]{Seongmin Kim}
\author[3]{Jaeyong Lee}
\affil[1]{Department of Big Data Convergence, Chonnam National University}
\affil[2]{Department of Statistics, Sookmyung Women's University}
\affil[3]{Department of Statistics, Seoul National University}

\maketitle

\begin{abstract}
We study Bayesian inference in the spiked covariance model, where a small number of spiked eigenvalues dominate the spectrum. Our goal is to infer the spiked eigenvalues, their corresponding eigenvectors, and the number of spikes, providing a Bayesian solution to principal component analysis with uncertainty quantification.
We place an inverse-Wishart prior on the covariance matrix to derive posterior distributions for the spiked eigenvalues and eigenvectors. Although posterior sampling is computationally efficient due to conjugacy, a bias may exist in the posterior eigenvalue estimates under high-dimensional settings. To address this, we propose two bias correction strategies: (i) a hyperparameter adjustment method, and (ii) a post-hoc multiplicative correction.
For inferring the number of spikes, we develop a BIC-type approximation to the marginal likelihood and prove posterior consistency in the high-dimensional regime $p>n$. Furthermore, we establish concentration inequalities and posterior contraction rates for the leading eigenstructure, demonstrating minimax optimality for the spiked eigenvector in the single-spike case.
Simulation studies and a real data application show that our method performs better than existing approaches in providing accurate quantification of uncertainty  for both eigenstructure estimation and estimation of the number of spikes.

\end{abstract}


\section{Introduction}\label{sec-intro}

Covariance matrix estimation is a crucial component of multivariate analysis, as the covariance matrix encodes the dependencies between variables. The sample covariance matrix is a widely used estimator; however, it becomes singular when the number of variables exceeds the number of observations. Moreover, \citet{yin1988limit} and \citet{bai2007asymptotics} showed that the eigenvalues and eigenvectors of the sample covariance matrix may fail to converge to their population counterparts in high-dimensional settings. To address this issue, structural assumptions are often imposed on the covariance matrix. For instance, \citet{cai2010optimal} and \citet{lee2023post} explored banded or bandable covariance matrices, while \citet{cai2013sparse} and \citet{lee2023postecon} studied sparse structures.

In this paper, we consider the spiked covariance model, which assumes that a small number of eigenvalues are significantly larger than the rest. Under this model, most variation in the data is captured along directions associated with these spiked eigenvectors. Our goal is to estimate both the spiked eigenvalues and their corresponding eigenvectors, as well as to determine the number of spikes. These parameters are particularly relevant in principal component analysis (PCA), where the spiked components correspond to dominant directions of variation.

\citet{johnstone2009consistency} investigated the asymptotic behavior of the sample eigenvector under a single-spiked covariance model of the form $\bm{\Sigma} = \nu_p \bm{\xi}_p \bm{\xi}_p^\top + \bm{I}_p$,
where $\nu_p > 0$ and $\bm{\xi}_p \in \mathbb{S}^{p-1}$ with $\mathbb{S}^{p-1} = \{ \bm{x} \in \mathbb{R}^p : \|\bm{x}\|_2 = 1 \}$. In this model, the largest eigenvalue of $\bm{\Sigma}$ is $\nu_p + 1$, with corresponding eigenvector $\bm{\xi}_p$. \citet{johnstone2009consistency} showed that the first sample eigenvector is consistent if and only if $p/n \to 0$, assuming $\nu_p$ is bounded above. Furthermore, the minimax lower bound for estimating $\bm{\xi}_p$ is given by $\min \left\{ (1 + \nu_p)/\nu_p^2 \cdot {p}/{n},\, 1 \right\}$
(see Example 15.19 in \citet{wainwright2019high}). These results imply that in high-dimensional regimes, $p> n$, consistent estimation of the spiked eigenvector requires either a diverging spike, $\nu_p \to \infty$, or additional structural assumptions.

In high-dimensional spiked models, sparsity assumptions on the eigenvectors have been studied by \citet{johnstone2009consistency} 
and \citet{ma2013sparse}, among others, to attain the consistency. In contrast, an alternative line of research including \citet{fan2013large}, \citet{wang2017asymptotics}, and \citet{cai2020limiting} has focused on divergence conditions on spiked eigenvalues without assuming sparsity on eigenvectors. \citet{fan2013large} introduced the pervasiveness condition in statistical factor models to characterize when such divergence conditions are practically relevant. Under the divergent conditions, \citet{wang2017asymptotics} and \citet{cai2020limiting} analyzed the asymptotic behavior of the sample eigenstructure in high dimensions.

Bayesian methods have also been developed to infer the spiked structure of covariance matrices. \citet{bishop1998bayesian} proposed Bayesian PCA using a factor model with isotropic Gaussian noise, but this method lacks consistency guarantees in high dimensions. 
\citet{berger2020bayesian} introduced shrinkage inverse Wishart (SIW) priors to address the eigenvalue separation issue of inverse-Wishart and Jeffreys priors, but their high-dimensional convergence properties remain unexplored.  
\citet{ma2022posterior} considered a sparse Bayesian factor model, where sparsity in the loading matrix implies that principal directions involve only a small subset of variables. While the model achieves consistency, the sparsity assumption limits its applicability.


There has been significant interest in estimating the number of spikes in spiked covariance or factor models, as this parameter is essential for determining the effective dimensionality of the data. Specifically, this problem has been studied in the context of selecting the number of principal components in PCA.
For example, \citet{bai2002determining} and \citet{ahn2013eigenvalue} proposed methods for determining the number of factors in approximate factor models by penalizing cross-sectional and time-series dimensions, and by identifying sharp declines in the eigenvalue spectrum, respectively.
 \citet{ke2023estimation} introduced a two-step method leveraging bulk eigenvalues under a gamma-distributed residual covariance model to robustly estimate $K$. Bayesian approaches, such as those proposed by \citet{bishop1998bayesian} and \citet{lopes2004bayesian} 
  attempt to infer the number of latent components by placing priors on the model dimension or rank. However, these methods lack asymptotic consistency guarantees.  
\citet{minka2000automatic} proposed an approximate marginal likelihood approach to develop a Bayesian procedure for selecting the number of principal components, and \citet{hoyle2008automatic} showed that this approach is consistent when \( p < n \). To the best of our knowledge, no theoretical result has established the asymptotic consistency of any Bayesian PCA dimension selection method when \( p > n \).



In this work, we study Bayesian inference for both the number of spikes and the associated spiked eigenstructure. Specifically, we formulate the problem through the joint posterior distribution \( \pi(\lambda_{1:K}, \bm{\xi}_{1:K}, K \mid \bbX_n) \), where \( \lambda_{1:K} \) and \( \bm{\xi}_{1:K} \) denote the top \( K \) spiked eigenvalues and their corresponding eigenvectors, and \( K \) is the unknown number of spikes. 
Here, $\bbX_n$ denotes the set of $n$ observations $\bm{X}_1,\ldots, \bm{X}_n$.
This formulation highlights a key advantage of the fully Bayesian approach: it simultaneously estimates the principal components and their dimensionality, while providing coherent uncertainty quantification for both. We decompose the posterior as $\pi(\lambda_{1:K}, \bm{\xi}_{1:K}, K \mid \bbX_n) = \pi(\lambda_{1:K}, \bm{\xi}_{1:K} \mid K, \bbX_n)\, \pi(K \mid \bbX_n)$,
and propose Bayesian methods for estimating \( \pi(\lambda_{1:K}, \bm{\xi}_{1:K} \mid K, \bbX_n) \) and \( \pi(K \mid \bbX_n) \), respectively.

For the estimation of $\pi(\lambda_{1:K}, \bm\xi_{1:K} \mid K, \bbX_n)$, 
we proceed by first obtaining the posterior distribution of the covariance matrix, from which the posterior distributions of the eigenvalues and eigenvectors are derived. 
Suppose an inverse-Wishart prior is imposed on the covariance matrix. 
The inverse-Wishart prior enables efficient posterior sampling by allowing direct draws from the closed-form posterior distribution owing to conjugacy.
However, in high-dimensional settings, we observe that the posterior estimates of $\lambda_{1:K}$ from the inverse-Wishart posterior may exhibit systematic bias. This bias motivates us to propose two bias-correction strategies: 
(i) a hyperparameter calibration approach that adjusts the hyperparmeter of inverse-Wishart prior, and 
(ii) a post-processing correction method. 
The proposed methods are theoretically justified through an eigenvalue perturbation framework developed in this study. 
Moreover, owing to the conjugacy of the inverse-Wishart prior, posterior samples by the bias-correction strategies can be generated independently, without requiring Markov Chain Monte Carlo (MCMC) convergence diagnostics. 
As a result, accurate inference can be obtained with a relatively small number of posterior samples. 
This independence also makes the posterior sampling procedure trivially parallelizable across multiple cores, further accelerating computation in proportion to the available computing resources and offering a computational advantage.

For estimating $\pi(K \mid \bbX_n)$, we adopt a Bayesian model selection approach by placing a prior on $K$ and approximating the marginal likelihood using a BIC-type criterion. This allows us to efficiently evaluate the posterior distribution over different values of $K$. Furthermore, we establish posterior consistency for the number of spikes, demonstrating that our method can correctly recover the true number of spikes in high-dimensional regimes.

We also study the asymptotic properties of the inverse-Wishart posterior for estimating spiked eigenvalues and eigenvectors. Previous works have primarily focused on the asymptotic behavior of sample eigenstructures \citep{johnstone2009consistency, wang2017asymptotics, cai2020limiting}. These results crucially rely on the rank-deficiency of the sample covariance matrix when \( p > n \). In contrast, posterior samples drawn from the inverse-Wishart prior are full-rank, which complicates asymptotic analysis. We develop novel concentration inequalities for spiked eigenstructures of full-rank random matrices and apply them to the inverse-Wishart setting to tackle this challenge. Under the single-spiked model, we further show that the posterior achieves the minimax optimal rate for estimating the leading eigenvector.

The rest of the paper is organized as follows. In Section \ref{sec:spiked}, we introduce the spiked covariance model. Section \ref{ssec:bayeseigen} presents the Bayesian method for estimating the spiked eigenstructure with the posterior contraction rate analysis given the number of spikes. 
Section \ref{sec:numpc} presents the Bayesian method for estimating the number of spikes with the analysis of posterior consistency.
In Section \ref{sec:numerical}, we illustrate the proposed method through simulation studies and real data analysis. The concluding remarks are given in Section \ref{sec:concludingremark}. We also provide the additional simulation studies and the proofs of theorems in the Appendix.

\section{Spiked Covariance Model and Factor Representation}\label{sec:spiked}

Suppose $\bm{X}_1, \ldots, \bm{X}_n$ are random samples from a $p$-dimensional multivariate normal distribution $N_p(\bm{0}_p, \bm{\Sigma})$, where $\bm{\Sigma} \in \calC_p$ and $\calC_p$ is the set of all $p \times p$ positive definite matrices. Let $\lambda_1 \geq \ldots \geq \lambda_p > 0$ and $\bm{\xi}_1, \ldots, \bm{\xi}_p$ denote the eigenvalues and corresponding eigenvectors of $\bm{\Sigma}$, respectively.  The covariance matrix $\bm{\Sigma}$ is referred to as a spiked covariance when the top $K$ eigenvalues are much larger than the remaining ones; that is, $\lambda_1 \geq \ldots\ge  \lambda_K >> \lambda_{K+1} \geq \ldots \geq \lambda_p$.

In particular, we consider the following model: 
\bean \bm{X}_1,\ldots, \bm{X}_n &\sim& N_p(\bm{0}, \bm\Sigma),\label{eq:model}\\
\frac{\lambda_{K+1}p}{\lambda_K n} &\lra& 0,\label{eq:spikecondition}\\
\frac{\lambda_{k}}{\lambda_{k+1}} &>& C , ~k=1,\ldots, K \label{eq:seperatedeigens} \eean for some positive constant $C > 1$. Condition \eqref{eq:seperatedeigens} assumes that the top $K$ eigenvalues are well-separated. Condition \eqref{eq:spikecondition} ensures that the top eigenvalues dominate relative to $\lambda_{K+1}p/n$. In other words, if $\lambda_{K+1}p/n$ is small, the top eigenvalues need not be large; however, if $\lambda_{K+1}p/n$ is large, they must exceed $\lambda_{K+1}p/n$ significantly. 
The condition \eqref{eq:spikecondition} requires the spiked eigenvalue to diverge whenever $\lambda_{K+1}p\gtrsim n$. 

\cite{fan2013large} introduced the pervasiveness condition to demonstrate the practicality of the divergent condition in the context of the statistical factor model. We describe the factor model and the pervasiveness condition, demonstrating that the pervasiveness condition ensures condition \eqref{eq:spikecondition}.
Suppose a $p$-dimensional observation $\bm{X}$ is explained by $K$ unobserved factors $\bm{f} = (f_{1},\ldots, f_{K})^T \in \bbR^K$ and can be represented by
\begin{equation}
    \begin{aligned}
\bm{X}\mid \bm{f} =  \sum_{k=1}^K\bm{b}_k f_{k} + \bm\epsilon,\quad\bm\epsilon \sim N_p(\bm{0}_p,   \bm\Sigma_u), \quad
\bm{f} \sim N_K(\bm{0}_K, \bm{I}_K), 
 \end{aligned}
\label{eq:fm}
\end{equation}
where $\bm{b}_k\in \bbR^p $ quantifies the effect of the $k$-th factor $f_k$ on the observation $\bm{X}$, and $ \bm\epsilon$ represents the error of $\bm{X}$ that is not explained by the factors.
By integrating out $f_{1},\ldots, f_{K}$, the observation $\bm{X}$ follows a multivariate normal distribution $N_p (\bm{0}_p,\bm{B} \bm{B}^T + \bm\Sigma_u) $, where $\bm{B} = (\bm{b}_1,\ldots,\bm{b}_K)\in \bbR^{p\times K}$.
Additionally, we suppose the columns of $\bm{B}$ are orthogonal, which is the canonical condition for the identifiability (see Proposition 2.1 in \cite{fan2013large}).

If $||\bm\Sigma_u||_2$ is bounded and $||\bm{b}_k||_2^2$ diverges at a rate of at least $p$ as $p\lra \infty$, i.e., $\liminf_{p\lra \infty}||\bm{b}_k||_2^2/p > 0$, $k=1,\ldots, K$, we define the factor model \eqref{eq:fm} satisfies the pervasiveness condition (see Assumption 2.1 in \cite{fan2021robust}), which means that the factor $f_k$ affects a substantial proportion of the variation in the observations. Without loss of generality, we assume $||\bm{b}_1||_2 \ge \ldots \ge ||\bm{b}_K||_2$. Since $\lambda_k(\bm\Sigma) \ge \lambda_k(\bm{B} \bm{B}^T) = ||\bm{b}_k||_2^2$, $k=1,\ldots, K$, and $\lambda_{K+1}(\bm\Sigma) \le ||\bm\Sigma_u||_2$, the pervasiveness condition yields $(\lambda_{K+1}p)/(\lambda_{K}n) \le (p ||\bm\Sigma_u||_2 )/(n ||\bm{b}_K||_2^2  ) \lesssim 1/n$, which implies \eqref{eq:spikecondition}.

The spiked covariance model provides a natural framework for capturing low-rank signal structures in high-dimensional data, with the top \(K\) eigencomponents representing the dominant variation. We therefore focus on Bayesian inference for the spiked covariance model, particularly on its posterior distribution, which factorizes as
\[
\pi(\lambda_{1:K}, \bm\xi_{1:K}, K \mid \bbX_n) 
= \pi(\lambda_{1:K}, \bm\xi_{1:K} \mid K, \bbX_n) \, \pi(K \mid \bbX_n).
\]
The conditional posterior \(\pi(\lambda_{1:K}, \bm\xi_{1:K} \mid K, \bbX_n)\) is discussed in Section~\ref{ssec:bayeseigen}, while the marginal posterior \(\pi(K \mid \bbX_n)\) is examined in Section~\ref{sec:numpc}. Details on Bayesian inference for the non-spiked component of the spiked covariance matrix are provided in Section~3 of \citet{lee2023postecon}.

\section{Bayesian Inference of Spiked Eigenstructure}\label{ssec:bayeseigen}

\subsection{Bias Correction in Posterior Eigenvalues}\label{ssec:biascorrect}

We consider the Bayesian inference of the spiked structure given the number of spikes, i.e., the estimation of $\pi(\lambda_{1:K}, \bm\xi_{1:K} \mid K, \bbX_n)$. 
As \cite{wang2017asymptotics} and \cite{cai2020limiting} inferred spiked eigenvalues and eigenvectors based on the sample covariance matrix, we derive the posterior distribution of spiked eigenvalues and eigenvectors from that of the covariance matrix.

Suppose that we place an inverse-Wishart (IW) prior on the population covariance matrix, $\bm\Sigma \sim IW_p(\bm{A}_n, \nu_n)$, with density 
\(
\pi(\bm\Sigma) \propto |\bm\Sigma|^{-\nu_n/2} \exp\left\{- \textup{tr}(\bm\Sigma^{-1} \bm{A}_n)/2\right\},
\)
where $\nu_n > 2p$ is the degrees of freedom and $\bm{A}_n$ is a $p \times p$ positive definite scale matrix.
Due to conjugacy, the posterior distribution of $\bm\Sigma$ given the data $\bbX_n$ is also inverse-Wishart:
\begin{equation}\label{eq:IWposterior}
\bm\Sigma \mid \bbX_n \sim IW_p\left(\bm{A}_n + n \bm{S}_n, \nu_n + n\right), 
\end{equation}
where $\bm{S}_n = \sum_{i=1}^n \bm{X}_i \bm{X}_i^\top / n$. While this conjugate framework is computationally attractive, we have observed that the posterior of eigenvalues derived from the inverse-Wishart posterior may be inflated. For example, when the degrees of freedom are set to \( \nu_n = 2p + 2 \) and the scale matrix is set to \( \bm{A}_n = \bm{O} \), the resulting posterior distribution has been empirically observed to inflate the eigenvalues, as illustrated in Figure \ref{fig:inflation_evals}. This inflation becomes more pronounced in high-dimensional settings and leads to overestimation of the eigenvalues. 
\begin{figure}[ht!]
\centering
\includegraphics[width=\textwidth]{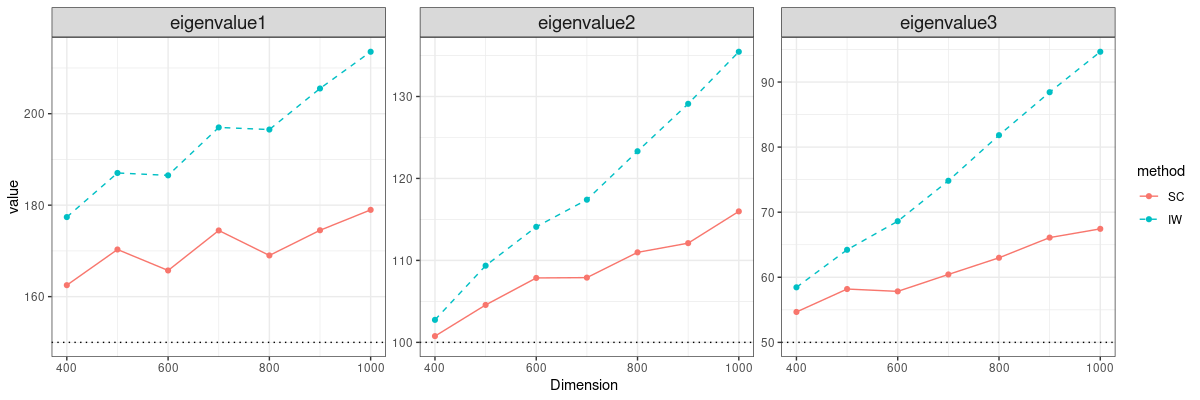}
\caption{\label{fig:inflation_evals} Bias inflation in leading eigenvalue estimates from sample covariance (SC) and inverse-Wishart posterior (IW) over increasing dimension \(p\). Each panel shows the average of the top three eigenvalue estimates across 10 replicate data sets. The solid line represents the eigenvalues of the sample covariance matrix, and the dashed line corresponds to the posterior means under the inverse-Wishart prior. The dotted line indicates the true eigenvalues, given by  \(\lambda_1 = 150\), \(\lambda_2 = 100\), \(\lambda_3 = 50\), with all remaining eigenvalues set to 1.}
\end{figure}

We propose two strategies to mitigate the inflation of posterior eigenvalues: adjusting the prior’s degrees-of-freedom parameter (prior calibration) and applying a post-hoc transformation to the posterior eigenvalue samples. The first strategy is to increase the degree-of-freedom hyperparameter \(\nu_n\) of the IW prior.  The degrees of freedom determines the level of shrinkage of the posterior distribution toward the scale matrix, with larger values of \(\nu_n\) leading to stronger shrinkage, thereby alleviating the overestimation (inflation) of posterior eigenvalues.  
Specifically, when we are interested in the $k$th eigenvalue $(k\le K)$, we propose
\bean\label{eq:biasnu}
\nu_n 
=
\frac{
n \lambda_k^+ + \sqrt{(n \lambda_k^+)^2 + 4(n \lambda_k(\bm S_n) - \hat{c} p)\Lambda^+}
}{
2(n \lambda_k(\bm S_n) - \hat{c} p)/n
}
- n + 2p + 2,
\eean
where
\(
\lambda_k^+ := \lambda_k\left(\bm{S}_n + \bm{A}_n / n \right), \quad
\Lambda^+ := \sum_{l=K+1}^p \lambda_l\left( \bm{S}_n + \bm{A}_n / n \right), \quad
\hat{c} =  {\sum_{j=K+1}^p \lambda_j(\bm{S}_n)} / (p-K-pK/n),
\)
and \(\lambda_k(\cdot)\) denotes the \(k\)th largest eigenvalue of a matrix.
We also set the scale matrix to satisfy \(\| \bm{A}_n \| = O(1)\). 
We refer to this approach as the \emph{prior calibration strategy}. 
This choice of $\nu_n$ is justified by the theoretical analysis given in Section~\ref{ssec:eigenbias}, which ensures that the posterior eigenvalue \( \lambda_k(\bm{\Sigma}) \) is centered around its true population counterpart \( \lambda_k(\bm{\Sigma}_0) \). This method is also asymptotically justified as given in Section \ref{eq:asymptotic}. However, since the suggested value of \( \nu_n \) depends on the index \( k \), the prior calibration strategy cannot correct the  $K$ leading eigenvalues at once using a single degree-of-freedom value. This structural limitation becomes restrictive when bias correction is needed for a large number of leading spiked eigenvalues.

To overcome this drawback, we propose the \emph{post-hoc correction strategy}. 
We generate posterior samples \( \lambda_k(\bm{\Sigma}) \) from \eqref{eq:IWposterior} with \(\nu_n\) satisfying \( \nu_n - 2p = o(n) \) and a scale matrix satisfying \( \| \bm{A}_n \| = O(1) \). 
We then define the post-processed posterior eigenvalue as
\bean\label{eq:biasadj}
\lambda_k^{\mathrm{adj}}(\bm\Sigma) :=
\frac{\gamma_2(\lambda_k(\bm{S}_n), \hat{c})}
     {\tilde{\gamma}_1(\nu_n,\bm A_n,\bm{S}_n, k)}
     \, \lambda_k(\bm{\Sigma}),
\eean
where the functions \(\tilde{\gamma}_1(\cdot)\) and \(\gamma_2(\cdot)\) are given by
{\footnotesize
\bean
\tilde{\gamma}_1(\nu_n, \bm A_n,\bm{S}_n, k)
&:=& \frac{n}{n + \nu_n - 2p - 2} 
\frac{\lambda_k(\bm S_n+\bm A_n/n)}{\lambda_k(\bm S_n)}
\left[1 + \frac{\sum_{l=K+1}^p \lambda_l(\bm S_n+\bm A_n/n)}{(n + \nu_n - 2p - 2)\lambda_k(\bm S_n+\bm A_n/n)}  \right], \label{eq:gamma1tilde}\\
\gamma_2(\lambda_k(\bm{S}_n), \hat{c}) 
&:=& 1 - \frac{\hat{c} p}{n \lambda_k(\bm{S}_n)}.\label{eq:gamma2}
\eean}

The theoretical justification for the correction factor 
\(
\gamma_2 / \tilde{\gamma}_1
\)
is provided in Section~\ref{ssec:eigenbias}, and also asymptotically justified as given in Section \ref{eq:asymptotic}.

Next, for the Bayesian inference of the spiked eigenvectors $\bm{\xi}_{1:K}$, we use posterior samples from the inverse-Wishart distribution \eqref{eq:IWposterior}. 
Each posterior draw of $\bm{\Sigma}$ from \eqref{eq:IWposterior} is decomposed into its eigenvectors, and the inference is based on the posterior distribution
\(
[\bm{\xi}_k(\bm{\Sigma}) \mid \bbX_n],
\)
where $\bm{\xi}_k(\cdot)$ denotes the operator that extracts the $k$th eigenvector of a positive definite matrix. While we apply bias correction to the spiked eigenvalues obtained from the inverse-Wishart posterior, we use the eigenvectors from the same posterior without any modification.  
This approach  is consistent with \citet{wang2017asymptotics},  which debiased the sample eigenvalues while retaining the eigenvectors of the sample covariance matrix.  
The validity of this procedure is guaranteed by posterior consistency, to be established in Section~\ref{eq:asymptotic}.

In summary, to approximate the posterior distribution of $\lambda_k$ and $\bm\xi_k$ with $k \le K$, we have suggested two methods:  
(i) by the prior calibration method, we impose the inverse-Wishart with the degree of freedom in \eqref{eq:biasnu}, and draw $N$ independent samples $\bm\Sigma_1, \dots, \bm\Sigma_N$ from the inverse-Wishart posterior, yielding posterior samples $(\lambda_k(\bm\Sigma_j), \bm\xi_k(\bm\Sigma_j))_{j=1}^N$;  
(ii) by the post-hoc correction, we draw $N$ independent samples $\bm\Sigma_1, \dots, \bm\Sigma_N$ from the inverse-Wishart posterior with arbitrary $\nu_n$ satisfying $\nu_n-2p=o(n)$, and then compute adjusted posterior samples $(\lambda_k^{\text{adj}}(\bm\Sigma_j), \bm\xi_k(\bm\Sigma_j))_{j=1}^N$.  
Both methods yield independent posterior samples for eigenvalues and eigenvectors, which serve as a Monte Carlo approximation to $[\lambda_k(\bm\Sigma), \bm\xi_k(\bm\Sigma) \mid \bbX_n]$.  
This procedure avoids convergence issues commonly associated with MCMC methods and remains computationally efficient due to conjugacy.  
Furthermore, because all posterior draws are independent, the computation can be trivially parallelized across multiple cores, resulting in additional speed-ups proportional to the available computing resources.

\subsection{Theoretical Analysis of Eigenvalue Bias}\label{ssec:eigenbias}

We provide a theoretical justification for the bias correction methods introduced in \eqref{eq:biasnu} and \eqref{eq:biasadj}.
To explain the phenomenon of eigenvalue inflation illustrated in Figure \ref{fig:inflation_evals}, we analyze the inflation of eigenvalues from the IW posterior and, based on this analysis, derive the degree-of-freedom parameter in \eqref{eq:biasnu} and the adjustment factors in \eqref{eq:biasadj} required to mitigate this bias.
Our theoretical analysis characterizes the range of high-dimensional regimes and the conditions under which these bias correction strategies remain valid.

Throughout this section, we use the following notation: for an integer $K$, we write $[K] = \{1, \ldots, K\}$. For sequences $\{a_n\}$ and $\{b_n\}$, we write $a_n = o(b_n)$ if $a_n / b_n \to 0$, and $a_n \lesssim b_n$ if there exists a constant $C > 0$ such that $a_n \le C b_n$ for all sufficiently large $n$. 
Similarly, $a_n = O(b_n)$ means there exists a constant $C > 0$ such that $|a_n| \le C |b_n|$ for all sufficiently large $n$. 
For random variables $X_n$ and positive sequences $a_n$, we write $X_n = O_p(a_n)$ if, for any $\epsilon > 0$, there exists a constant $M > 0$ such that $\mathbb{P}(|X_n| > M a_n) < \epsilon$ for all sufficiently large $n$. 
The notation \( X_n=  o_p(a_n) \) denotes a term that converges to zero in probability when divided by \( a_n \).

To understand the structure of eigenvalue bias, we introduce an eigenvalue perturbation framework formalized in Theorem~\ref{thm:eigenbias}, inspired by Rayleigh--Schrödinger perturbation theory \citep{schrodinger1926quantisierung}.
Theorem~\ref{thm:eigenbias} provides a multiplicative approximation for the \( k \)-th eigenvalue of a covariance matrix \( \bm{\Sigma} \) for \( k \le K \), in terms of the leading principal submatrix \( \bm{\Omega}_{11} \in \mathcal{C}_K \) defined in \eqref{eq:Sigmadecomp} for an arbitrary orthogonal matrix $\bm\Gamma$.
Specifically, as shown in \eqref{eq:SigmakRS}, the eigenvalue \( \lambda_k(\bm{\Sigma}) \) is approximated by \( \lambda_k(\bm{\Omega}_{11}) \) multiplied by a correction factor containing a higher-order residual term \( R \).

\begin{theorem}\label{thm:eigenbias}
Let \( \bm{\Sigma} \in \mathcal{C}_p \) and let \( \bm{\Gamma} \) be a \( p \times p \) orthogonal matrix. Define
\bean\label{eq:Sigmadecomp}
 \begin{pmatrix}
    \bm{\Omega}_{11} & \bm{\Omega}_{12} \\
    \bm{\Omega}_{21} & \bm{\Omega}_{22}
\end{pmatrix} = \bm{\Gamma}^T \bm{\Sigma} \bm{\Gamma},
\eean
where \( \bm{\Omega}_{11} \in \mathcal{C}_K \), \( \bm{\Omega}_{22} \in \mathcal{C}_{p-K} \). Suppose, for arbitrary $d_1,d_2>0$ and $C>1$,
\[
\|\bm{\Omega}_{22}\|_2 \le x, \quad 
\left( 1 \vee \sum_{l\le K, l\neq k} \left|\frac{\lambda_k}{C(\lambda_l-\lambda_k)} \right|^{d_1/2} \vee \sum_{l\le K, l\neq k} \left|\frac{\lambda_k}{C(\lambda_l-\lambda_k)} \right|^{d_{2}/2} \right) \|\bm{\Omega}_{21}\bm{\xi}_l\|_2 \le x,
\]
and \( (4eCx/\lambda_k) < 1 \), where 
$\lambda_l = \lambda_l(\bm\Omega_{11})$ and \( \bm{\xi}_l  = \bm\xi_l(\bm\Omega_{11})\). If $\lambda_k(\bm\Omega_{11})>1$, then
\begin{align}
\lambda_k(\bm{\Sigma}) &= \lambda_k(\bm{\Omega}_{11}) \left[1 + \frac{\|\bm{\Omega}_{21}\bm{\xi}_k\|^2}{\lambda_k(\bm{\Omega}_{11})^2} + R \right], \label{eq:SigmakRS}\\
R &\le \left( \frac{4eCx }{\lambda_k(\bm{\Omega}_{11})} \right)^3 \left(1 - \frac{4eCx}{\lambda_k(\bm{\Omega}_{11})} \right)^{-1}. \nonumber
\end{align}
\end{theorem}

Applying equation \eqref{eq:SigmakRS} in Theorem~\ref{thm:eigenbias} to the inverse-Wishart posterior, we obtain Theorem~\ref{thm:IWeigenbias}, which provides an approximation for the $k$th eigenvalue of inverse-Wishart posterior sample. Theorem~\ref{thm:IWeigenbias} relies on Assumptions~\ref{assump:spike}--\ref{assump:scalegap}, which describe the spiked structure of the population covariance matrix along with distributional conditions on the data. We first state the assumptions required for Theorem~\ref{thm:IWeigenbias}.
\begin{assumption}[Spike eigenvalue condition]\label{assump:spike}
Let \( d_i := p / (n \lambda_{0,i}) \). The spike eigenvalues satisfy
\(
d_i (\log n)^3 \to 0 \quad \text{for all } i = 1, \ldots, K,
\)
and the non-spike true eigenvalues are assumed to be bounded. Additionally, we assume \( K / n^{1/6} \to 0 \) and \( K^2 d_K \to 0 \).
\end{assumption}

\begin{assumption}[Moment bound]\label{assump:moment}
For all \( d \in \mathbb{N}_+ \), there exist constants \( c_d > 0 \) such that \( \mathbb{E}|X_{ij}|^d \le c_d \) for all \( i = 1, \ldots, n \) and \( j = 1, \ldots, p \), where $\bm X_i = (X_{i1},\ldots, X_{ip})^T$.
\end{assumption}

\begin{assumption}[Bulk eigenvalue separation]\label{assump:bulk}
Let \( \bm\Sigma_1 = \bm{U}_{0,2} \bm\Lambda_{0,2} \bm{U}_{0,2}^\top \), where \( \bm\Lambda_{0,2} = \mathrm{diag}(\lambda_{0,K+1}, \ldots, \lambda_{0,p}) \) and \( \bm{U}_{0,2} \in \mathbb{R}^{p \times (p-K)} \) are the matrices of the $K+1$th to $p$th eigenvalues and eigenvectors. Let \( m_1(z) \) be the solution to
\(
m_1(z) = -{1}/{\Big(z -  \operatorname{tr}\left( \left( \bm{I} + m_1(z) \bm\Sigma_1 \right)^{-1} \bm\Sigma_1 \right)/n\Big)}, \quad z \in \mathbb{C}^+,
\)
and define
\(
\gamma_+ = \inf \left\{ x \in \mathbb{R} : F_0(x) = 1 \right\},
\)
where \( F_0(x) \) is the cumulative distribution function determined by \( m_1(z) \) (see the third display on page 4 in \cite{bao2013local}). Suppose that
\[
\limsup_{n \to \infty} \lambda_{0,K+1} d < 1,
\quad \text{where} \quad
d = -\lim_{z \to \gamma_+,\, z \in \mathbb{C}^+} m_1(z).
\]
\end{assumption}

\begin{assumption}[Relative scale of target eigenvalue]\label{assump:scalegap}
Let \( k \in \{1, \ldots, K\} \) denote the index of target eigenvalue. Assume
\(
{K \lambda_{0,1} (\log n)^2}/{(n \lambda_{0,k})} = o(1).
\)

\end{assumption}

Assumptions~\ref{assump:spike}--\ref{assump:bulk} are adapted from Theorem 2.5 of \citet{cai2020limiting} (specifically, Assumptions 2, 7, and 8 therein). Assumption~\ref{assump:moment} is satisfied under Gaussianity. 
Assumptions~\ref{assump:spike} and \ref{assump:bulk} impose eigenvalue separation conditions ensuring that the top \( K \) eigenvalues are spiked. These conditions hold, for example, under the flat bulk scenario \( \lambda_{0,K+1} = \cdots = \lambda_{0,p} = \sigma^2 \) with bounded \( \sigma^2 \), and when \( p / (n \lambda_{0,k}) \to 0 \) for all \( k = 1, \ldots, K \), where \( \lambda_{0,k} \) denotes the \( k \)th eigenvalue of \( \bm\Sigma_0 \); see Remark 1.9 of \citet{bao2015universality} and Remark 7 of \citet{cai2020limiting}. Assumption~\ref{assump:scalegap} places a lower bound on the scale gap between the $k$th and first eigenvalues. Specifically, it requires
\(
{\lambda_{0,k}}/{\lambda_{0,1}} \gg {K (\log n)^2}/{n},
\)
which does not contradict the ordering $\lambda_{0,k} < \lambda_{0,1}$. 
This condition holds provided that $\lambda_{0,k}$ remains sufficiently large relative to $\lambda_{0,1}$—that is, $\lambda_{0,k}$ must be at least of order $\lambda_{0,1} {K (\log n)^2}/{n}$.

\begin{theorem}\label{thm:IWeigenbias}
Suppose \( \bm{X}_1, \ldots, \bm{X}_n \) are independent samples with \( \mathbb{E}(\bm{X}_i) = \bm{0} \) and \( \mathbb{E}(\bm{X}_i \bm{X}_i^\top) = \bm{\Sigma}_0 \), where \( p > n > K \).  
Assume that Assumptions~\ref{assump:spike}--\ref{assump:scalegap} hold, and suppose $\bm\Sigma$ follows the posterior distribution \eqref{eq:IWposterior} with hyperparameters satisfying \( \nu_n - 2p = o(n) \) and \( \|\bm{A}_n\| = O(1) \).  

Let  
\(
\bm{\Omega}_{11} := \hat{\bm{\Gamma}}_1^\top \bm{\Sigma} \hat{\bm{\Gamma}}_1,
\quad
\bm{\Omega}_{21} := \hat{\bm{\Gamma}}_2^\top \bm{\Sigma} \hat{\bm{\Gamma}}_1,
\)
where \( \hat{\bm{\Gamma}}_1 \in \mathbb{R}^{p \times K} \) is the matrix of the top \( K \) eigenvectors of  
\(
\hat{\bm{\Sigma}} = {(n \bm{S}_n + \bm{A}_n)}/{(n + \nu_n - 2p - 2)},
\)
and \( \hat{\bm{\Gamma}}_2 \in \mathbb{R}^{p \times (p-K)} \) consists of the remaining eigenvectors of $\hat{\bm\Sigma}$.  

Then, for \( k \in [K] \),
\bean\label{eq:IWeigenapprox}
\lambda_k(\bm{\Sigma}) 
= \lambda_k(\bm{\Omega}_{11}) \left( 1 + \frac{\|[\bm{\Omega}_{21}]_k\|^2}{\lambda_k(\bm{\Omega}_{11})^2} + o_p\left( \frac{p}{n \lambda_{0,k}} \right) \right),
\eean
where \( [\bm{\Omega}_{21}]_k \) denotes the \( k \)th column of \( \bm{\Omega}_{21} \).  

Furthermore, the posterior distribution of \( \bm{\Omega}_{11} \) is
\bean\label{eq:Omega11IW}
\bm{\Omega}_{11} \mid \mathbb{X}_n \sim IW_K\left( (n + \nu_n - 2p - 2) \hat{\bm{\Lambda}}_1, \, n + \nu_n - 2p + 2K \right),
\eean
where \( \hat{\bm{\Lambda}}_1 = \mathrm{diag}(\hat{\lambda}_1, \ldots, \hat{\lambda}_K) \), and \( \hat{\lambda}_k \) is the \( k \)th eigenvalue of \( \hat{\bm{\Sigma}} \).  

The posterior expectation of \( \|[\bm{\Omega}_{21}]_k\|^2 \) is given by
\bean\label{eq:postexpectation}
\mathbb{E}\left( \|[\bm{\Omega}_{21}]_k\|^2 \mid \mathbb{X}_n \right)
= \frac{(n - \nu_n - 2p - 2) \, \hat{\lambda}_k \, \sum_{l=K+1}^p \hat{\lambda}_l}
{(n + \nu_n - 2p - 1)(n + \nu_n - 2p - 4)}.
\eean 
\end{theorem}
The proof of Theorem~\ref{thm:IWeigenbias} is given in the Appendix \ref{sec:pf3.2}. Leveraging Theorem~\ref{thm:IWeigenbias}, we obtain the following approximation for the $k$-th posterior eigenvalue:
\bean\label{eq:IWapprox1}  
\lambda_k(\bm{\Sigma}) \approx \tilde{\gamma}_1(\nu_n, \bm A_n, \bm{S}_n,k) \, \lambda_k(\bm{S}_n),
\eean
where $\tilde{\gamma}_1(\nu_n,\bm A_n,  \bm{S}_n, k)$ is defined in \eqref{eq:gamma1tilde} and explicitly characterizes the bias of the posterior eigenvalue relative to the sample eigenvalue. To derive equation \eqref{eq:IWapprox1}, note that \( K \ll n \), equation~\eqref{eq:Omega11IW} implies that the eigenvalue \( \lambda_k(\bm{\Omega}_{11}) \) concentrates around a rescaled sample eigenvalue:
\(
\lambda_k(\bm{\Omega}_{11}) \approx {n}/{(n + \nu_n - 2p - 2)} \lambda_k(\bm{S}_n+ \bm A_n/n) .
\)
Replacing \( \lambda_k(\bm{\Omega}_{11}) \) in equation~\eqref{eq:IWeigenapprox} with this approximation yields
\(
\lambda_k(\bm{\Sigma}) \approx \gamma_1(\nu_n, \bm A_n,\bm{S}_n, k, \bm{\Omega}_{21}) \, \lambda_k(\bm{S}_n),
\)
where
{\footnotesize
\begin{equation*}
\gamma_1(\nu_n,  \bm A_n,\bm{S}_n, k, \bm{\Omega}_{21}) 
= \frac{n}{n + \nu_n - 2p - 2}
\frac{\lambda_k(\bm{S}_n+ \frac{\bm A_n}{n}) }{\lambda_k(\bm S_n)}
\left[ 1 + \frac{\| [\bm{\Omega}_{21}]_k \|^2}{ \left\{ \frac{n}{n + \nu_n - 2p - 2} \lambda_k(\bm{S}_n+\frac{\bm A_n}{n}) \right\}^2 }
+ o_p\!\left( \frac{p}{n \lambda_{0,k}} \right) \right].
\end{equation*}
}

To make the bias structure explicit, we approximate \( \|[\bm{\Omega}_{21}]_k\|^2 \) in \( \gamma_1 \) by its posterior expectation given in \eqref{eq:postexpectation}, leading to
\(
\gamma_1(\nu_n,  \bm A_n,\bm{S}_n, k, \bm{\Omega}_{21})
\approx \tilde{\gamma}_1(\nu_n,  \bm A_n,\bm{S}_n, k),
\)
where  the approximations $(n - \nu_n - 2p - 2) / (n - \nu_n - 2p - 1) \approx 1,\quad (n - \nu_n - 2p - 2) / (n - \nu_n - 2p - 4) \approx 1$ are applied. Moreover, the remainder term \( o_p\left( {p}/{(n \lambda_{0,k})} \right) \) is dominated by 
\(
\left[ (n + \nu_n - 2p - 2) \, \lambda_k\!\left(\bm S_n + \bm A_n / n\right) \right]^{-1}  \sum_{l=K+1}^p \lambda_l\!\left(\bm S_n + \bm A_n / n\right) 
\)
when $||\bm S_n||  \gg ||\bm A_n||$.
Combining these approximations leads to \eqref{eq:IWapprox1}, which expresses the posterior eigenvalue bias relative to the corresponding sample eigenvalue.

To further connect the posterior eigenvalue with the population eigenvalue, we analyze the bias of the sample eigenvalue relative to its population counterpart. Using the result of \citet{wang2017asymptotics}, we have  
\bean\label{eq:sampleapprox}  
\lambda_k(\bm{\Sigma}_0) \approx \gamma_2(\lambda_k(\bm{S}_n), \hat{c}) \, \lambda_k(\bm{S}_n),
\eean  
where $\gamma_2(\lambda_k(\bm{S}_n), \hat{c})$ is defined in \eqref{eq:gamma2}.  
This expression is a reformulation of  
\(
\lambda_k(\bm{S}_n) \approx \lambda_k(\bm{\Sigma}_0) + {\hat{c} p}/{n},
\)
as derived in \citet{wang2017asymptotics}. By combining \eqref{eq:IWapprox1} and \eqref{eq:sampleapprox}, we determine the degree-of-freedom parameter \(\nu_n\) by solving
\(
\tilde{\gamma}_1(\nu_n, \bm A_n,\bm{S}_n, k) = \gamma_2(\lambda_k(\bm{S}_n), \hat{c}),
\)
which yields the prior calibration rule in \eqref{eq:biasnu}.  For the post-hoc correction in \eqref{eq:biasadj}, each posterior eigenvalue \(\lambda_k(\bm{\Sigma})\) is multiplied by
\(
{\gamma_2(\lambda_k(\bm{S}_n), \hat{c})}/{\tilde{\gamma}_1(\nu_n, \bm A_n,\bm{S}_n, k)},
\)
thereby adjusting for both bias components in \eqref{eq:IWapprox1} and \eqref{eq:sampleapprox} and aligning the posterior eigenvalues with the population eigenvalues.

\subsection{Posterior asymptotic analysis}\label{eq:asymptotic}

We develop a general framework for posterior contraction of spiked eigenvalues and eigenvectors. The framework can be applied to arbitrary random positive definite matrices and is used to analyze the convergence of the posterior distribution of the spiked eigenstructure in Section~\ref{ssec:biascorrect}. Formally, a sequence $\epsilon_n \to 0$ is called a posterior contraction rate at $\theta_0$ with respect to a loss function $d$ if, for any $M_n \to \infty$,
\(
\pi\{ \theta : d(\theta,\theta_0) \ge M_n \epsilon_n \mid \mathbb{X}_n \} \to 0
\)
in $\mathbb{P}(\mathbb{X}_n;\theta_0)$-probability; see \citet{ghosal2017fundamentals}.

We begin by establishing the framework for spiked eigenvalues. 
The following result provides a concentration bound for the leading eigenvalues of a random positive definite matrix $\bm{\Sigma}$ relative to a fixed reference matrix $\bm{\Sigma}_0 \in \mathcal{C}_p$. 
Specifically, we assess the probability that
\(
P\left( \sup_{l=1,\ldots,k} \left| {\lambda_l(\bm{\Sigma})}/{\lambda_l(\bm{\Sigma}_0)} - 1 \right| > C t \right).
\)

\begin{theorem}\label{thm:higheigen}
    Suppose $\bm{\Sigma}$ is a positive definite random matrix, and $\bm{\Sigma}_0 \in \mathcal{C}_p$ is fixed. Let $k, K \in [p]$ with $k \le K$, and let $\bm{u}_{0,1},\ldots, \bm{u}_{0,p}$ be the eigenvectors of $\bm{\Sigma}_0$.
    Define $\bm{\Gamma} = [\bm{u}_{0,1}, \ldots, \bm{u}_{0,K}]$ and $\bm{\Gamma}_\perp = [\bm{u}_{0,K+1}, \ldots, \bm{u}_{0,p}]$. 
    Suppose the eigengap condition $\min_{l=1,\ldots,K-1} \{\lambda_{0,l}/\lambda_{0,l+1}\} > c$ holds for some $c > 1$. Then, for all $t \le \delta$,
    \begin{align*}
        P\left( \sup_{l=1,\ldots,k} \left| \frac{\lambda_l(\bm{\Sigma})}{\lambda_{0,l}} - 1 \right| > C t \right) 
        &\le P\left( \left\| \bm{\Gamma}_\perp^T \bm{\Sigma}_0^{-1/2} \bm{\Sigma} \bm{\Sigma}_0^{-1/2} \bm{\Gamma}_\perp \right\|_2^{1/2} \frac{\sqrt{\lambda_{0,K+1}}}{\sqrt{\lambda_{0,k}}} > t \right) \\
        &\quad + P\left( K \left\| \bm{\Gamma}^T \bm{\Sigma}_0^{-1/2} \bm{\Sigma} \bm{\Sigma}_0^{-1/2} \bm{\Gamma} - \bm{I}_K \right\|_2 > t \right),
    \end{align*}
    where $\delta$ and $C$ are positive constants depending on $c$.
\end{theorem}
The proof is provided in the Appendix \ref{sec:pf3.3}. Theorem \ref{thm:higheigen} shows that the concentration of the leading eigenvalues of $\bm{\Sigma}$ is governed by the deviation of the scaled covariance matrix $\bm{\Sigma}_0^{-1/2} \bm{\Sigma} \bm{\Sigma}_0^{-1/2}$ from the identity on both the spiked and non-spiked subspaces.
Building on this framework, Theorem \ref{corr:IWeigenvalue} examines the contraction behavior of the posterior distribution of the spiked eigenvalues introduced in Section~\ref{ssec:biascorrect}.
Since $\nu_n$ from \eqref{eq:biasnu} satisfies $\nu_n - 2p = o(n)$, equation \eqref{eq:IWeigen1} provides the posterior contraction rate of the posterior eigenvalues under the posterior calibration rule \eqref{eq:biasnu}, and equation \eqref{eq:IWeigen2} gives the posterior contraction rate for the post-hoc correction method in \eqref{eq:biasadj}.

\begin{theorem}\label{corr:IWeigenvalue}
    Suppose $\bm{X}_1,\ldots, \bm{X}_n \sim N_p(\bm{0}_p, \bm{\Sigma}_0)$, and that $\bm{\Sigma}_0$ satisfies conditions \eqref{eq:spikecondition}--\eqref{eq:seperatedeigens}, with $K^3/n = o(1)$ and ${p}/{n^2} = o(1)$. 
    Consider the IW prior $\bm{\Sigma} \sim IW_p(\bm{A}_n, \nu_n)$.
    Let $k \in [K]$ and define
    \(
    \epsilon_n^2 = {K^3}/{n} + {\lambda_{0,K+1}}/{\lambda_{0,k}} \left( {p}/{n} \vee 1 \right),
    \)
    where $\lambda_{0,1} \ge \cdots \ge \lambda_{0,p} > 0$ are the eigenvalues of $\bm{\Sigma}_0$. 
    If $\nu_n- 2p = o(n)$ and $\|\bm{A}_n\| = O(1)$, then
    \bean\label{eq:IWeigen1}
    \pi\left( \left| \sup_{l=1,\ldots, k}\frac{\lambda_l(\bm{\Sigma})}{\lambda_{0,l}} - 1 \right| > M_n \epsilon_n \,\Big|\, \mathbb{X}_n \right) \to 0
    \eean
    in probability for any $M_n \to \infty$.

    For the post-hoc correction method \eqref{eq:biasadj}, suppose $(p-K)^{-1}\sum_{j=K+1}^p \lambda_{0,j} = \bar c + o_p(n^{-1/2})$ for some positive constant $\bar c$ and define 
    \(
    \epsilon_n^{(2)} = \epsilon_n + (\nu_n - 2p - 2) / (n + \nu_n - 2p - 2) + p \bar{c} / (n \lambda_{0,k}).
    \)
    If $\nu_n - 2p = o(n)$ and $\|\bm{A}_n\| = O(1)$, then
    \bean\label{eq:IWeigen2}
    \pi\left( \sup_{l=1,\ldots,k} \left| \frac{\lambda_l^{\mathrm{adj}}(\bm{\Sigma})}{\lambda_{0,l}} - 1 \right| > M_n \epsilon_n^{(2)} \,\Big|\, \mathbb{X}_n \right) \to 0
    \eean
    in probability for any $M_n \to \infty$.
\end{theorem}

The proof is provided in the Appendix \ref{sec:pf3.4}.  
Theorem \ref{corr:IWeigenvalue} requires the additional condition \( {p}/{n^2} = o(1) \), which is not restrictive since it is satisfied whenever \( {p}/{n} \to c \) for any constant \( c \in \mathbb{R} \).  Theorem \ref{corr:IWeigenvalue} shows that the posterior eigenvalue obtained via the prior calibration strategy has the contraction rate as
\(
\epsilon_n = \sqrt{{K^3}/{n}} + \sqrt{{\lambda_{0,K+1}}/{\lambda_{0,k}}}\left( \sqrt{{p}/{n}} \vee 1 \right),
\)
while the posterior contraction rate of the post-hoc correction method \eqref{eq:biasadj} is  
\(
    \epsilon_n^{(2)} = \epsilon_n + (\nu_n - 2p - 2) / (n + \nu_n - 2p - 2) + p \bar{c} / (n \lambda_{0,k}).
\)
For the analysis of \eqref{eq:biasadj}, Theorem \ref{corr:IWeigenvalue} additionally assumes that the non-spiked true eigenvalues satisfy  
\(
(p-K)^{-1}\sum_{j=K+1}^p \lambda_{0,j} = \bar c + o_p(n^{-1/2})
\)
for some positive constant \(\bar c\).  
This condition holds, for example, when the non-spiked eigenvalues are bounded.  
This assumption follows Assumption 2.2 in \citet{wang2017asymptotics}.


Next, we establish a theoretical framework for analyzing the contraction of eigenvectors $\bm{\xi}_k(\bm{\Sigma})$ toward their true counterparts $\bm{\xi}_{0,k}$. 
We measure the discrepancy between two unit vectors $\bm{u}$ and $\bm{u}_0$ by 
\(
d(\bm{u}, \bm{u}_0) = 1 - (\bm{u}^\top \bm{u}_0)^2,
\)
which is invariant under sign changes. Theorem \ref{thm:highevector} provides a general concentration inequality for eigenvectors.

\begin{theorem}\label{thm:highevector}
Suppose $\bm\Sigma$ is a random positive definite matrix, and $\bm\Sigma_0$ is a fixed positive definite matrix. Let $K \in [p]$, and let $\lambda_{0,1} \ge \cdots \ge \lambda_{0,p} > 0$ be the eigenvalues of $\bm\Sigma_0$, with corresponding eigenvectors $\bm{u}_{0,1}, \ldots, \bm{u}_{0,p}$. Likewise, let $\lambda_1 \ge \cdots \ge \lambda_p > 0$ be the eigenvalues of $\bm\Sigma$, with corresponding eigenvectors $\bm{u}_1, \ldots, \bm{u}_p$.
Define $\bm\Gamma = [\bm{u}_{0,1}, \ldots, \bm{u}_{0,K}]$ and $\bm\Gamma_\perp = [\bm{u}_{0,K+1}, \ldots, \bm{u}_{0,p}]$, and set
\(
B_k = \sup_{l=1,\ldots,k} \left\{ \left( \sup_{i \le l-1} \lambda_{0,l} / \lambda_{0,i} \right) \vee \left( \sup_{i \ge l+1} \lambda_{0,i} / \lambda_{0,l} \right) \right\}, \quad k \in [K].
\)

Assume $\min_{l=1,\ldots,K-1} (\lambda_{0,l}/\lambda_{0,l+1}) > c$ for some $c > 1$, and $(\lambda_{0,K+1}/\lambda_k) \le d$ for some $d < 1$.

Then, for all $t \le \delta$, we have:
{\footnotesize
\begin{align}
P\left( \sup_{l=1,\ldots, k} \left\{1-(\bm{u}_l^\top \bm{u}_{0,l})^2\right\} > Ct \right)
&\le
P\left( \left\| \bm\Gamma_\perp^\top \bm\Sigma_0^{-1/2} \bm\Sigma \bm\Sigma_0^{-1/2} \bm\Gamma_\perp \right\|_2^{1/2} \frac{\sqrt{\lambda_{0,K+1}}}{\sqrt{\lambda_{0,k}}} > \sqrt{t} \right) \nonumber \\
&\quad +
P\left( \left\| \bm\Gamma^\top \bm\Sigma_0^{-1/2} \bm\Sigma \bm\Sigma_0^{-1/2} \bm\Gamma - \bm{I}_K \right\|_2 > \frac{\sqrt{t}}{\sqrt{B_k}} \wedge \delta_2 \right) \nonumber \\
&\quad +
P\left( \sup_{l=1,\ldots,k} \left| \frac{\lambda_l}{\lambda_{0,l}} - 1 \right| > \delta_1 \right),
\label{eq:eigenupper1}
\end{align}}
and
{\footnotesize
\begin{align}
P\left( \sup_{l=1,\ldots, k} \left\{1-(\bm{u}_l^\top \bm{u}_{0,l})^2\right\} > Ct \right)
&\le
P\left( \left\| \bm\Sigma_0^{-1/2} \bm\Sigma \bm\Sigma_0^{-1/2} - \bm{I}_p \right\|_2 \frac{\sqrt{\lambda_{0,K+1}}}{\sqrt{\lambda_{0,k}}} > \sqrt{t} \right) \nonumber \\
&\quad +
P\left( \left\| \bm\Gamma^\top \bm\Sigma_0^{-1/2} \bm\Sigma \bm\Sigma_0^{-1/2} \bm\Gamma - \bm{I}_K \right\|_2 > \frac{\sqrt{t}}{\sqrt{B_k}} \wedge \delta_2 \right) \nonumber \\
&\quad +
P\left( \sup_{l=1,\ldots,k} \left| \frac{\lambda_l}{\lambda_{0,l}} - 1 \right| > \delta_1 \right),
\label{eq:eigenupper2}
\end{align}}
where $C$, $\delta_1$, $\delta_2$, and $\delta$ are positive constants depending only on $c$ and $d$.
\end{theorem}

The proof of Theorem \ref{thm:highevector} is given in the Appendix \ref{sec:pf3.5}.  
Note that Theorem \ref{thm:highevector} provides two types of concentration inequalities: equation \eqref{eq:eigenupper1} is particularly useful when $p > n$, while equation \eqref{eq:eigenupper2} is suited for $p \le n$. 

By Theorem~\ref{thm:highevector}, we obtain Theorem~\ref{corr:IWeigenvector}, which establishes the posterior contraction rate for the top \(k\) eigenvectors as
$
\epsilon_n = B_k K / n + p \lambda_{0,K+1} / (n \lambda_{0,k}),
$
where \(B_k\) is defined in Theorem~\ref{thm:highevector}.
In contrast to Theorem~3.2 of \citet{wang2017asymptotics} and Theorem~4.1 of \citet{cai2020limiting}, this result ensures uniform convergence over the top \(k\) eigenvectors.

\begin{theorem}\label{corr:IWeigenvector}
Suppose the same setting as in Theorem \ref{corr:IWeigenvalue}. \\ 
Let 
\(
B_k = \sup_{l=1,\ldots,k} \left\{ \left( \sup_{i \le l-1} {\lambda_{0,l}}/{\lambda_{0,i}} \right) \vee \left( \sup_{i \ge l+1} {\lambda_{0,i}}/{\lambda_{0,l}} \right) \right\}
\)
and define 
\(
\epsilon_n = {B_k K}/{n} + {p \lambda_{0,K+1}}/{(n \lambda_{0,k})}.
\)
Then
\(
\pi\left( \sup_{l=1,\ldots, k} \left\{ 1 - (\bm{\xi}_l(\bm{\Sigma})^\top \bm{\xi}_{l}(\bm{\Sigma}_0))^2 \right\} > M_n \epsilon_n \mid \mathbb{X}_n \right) \to 0
\)
in probability for any positive sequence $M_n \to \infty$.
\end{theorem}

The proof is given in the Appendix \ref{sec:pf3.6}. 

The posterior contraction rate of the eigenvectors achieves minimax optimality under the single spiked covariance model 
\(
\bm{\Sigma}_0  = \nu_p \bm{\xi}_p \bm{\xi}_p^\top + \bm{I}_p,
\)
where $\bm{\xi}_p  \in \mathbb{S}^{p-1}$ and $\nu_p>0$. The eigenvalues of $\bm{\Sigma}_0$ are 
\(
\lambda_{0,1} = \nu_p + 1
\)
and 
\(
\lambda_{0,2} = \cdots = \lambda_{0,p} = 1,
\)
with the first eigenvector being $\bm{\xi}_p$.
Theorem \ref{corr:IWeigenvector} with $K=1$ and $k=1$ gives the posterior contraction rate for the first eigenvector as 
\(
{(1+p)}/{(n\nu_p + n)},
\)
which is asymptotically equivalent to the minimax lower bound given in Proposition \ref{prop:minimaxlower}.

\begin{proposition}\label{prop:minimaxlower}
Suppose $\bm{X}_1,\ldots, \bm{X}_n$ are independent samples from 
\(
N_p(\bm{0}_p, \nu_p \bm{\xi}_p\bm{\xi}_p^\top + \bm{I}_p).
\)
Let $\hat{\bm{\xi}}_p$ denote an eigenvector estimator. Then, the minimax lower bound is 
    \[
     \inf_{\hat{\bm{\xi}}_p} \sup_{\bm{\xi}_p\in \mathbb{S}^{p-1}} 
     \mathbb{E} \big[ 1-(\hat{\bm{\xi}}_p^\top \bm{\xi}_p)^2 \big] 
     \gtrsim 
     \min \left\{ \frac{1+\nu_p}{\nu_p^2} \frac{p}{n}, \, 1 \right\}.
    \]
\end{proposition}
The proof is given in the Appendix \ref{sec:pf3.7}.

\section{Bayesian inference of the Number of Spikes}\label{sec:numpc}

We consider the problem of estimating the number of spikes \(K\) in a spiked covariance model using a Bayesian approach. Given observed data \(\bbX_n = (X_1, \ldots, X_n)\), the posterior distribution of \(K\) is given by $\pi(K \mid \bbX_n) \propto \pi(K) p(\bbX_n \mid K)$,
where \(\pi(K)\) denotes the prior on \(K\), and \(p(\bbX_n \mid K)\) is the marginal likelihood under the model with \(K\) spikes. Since the marginal likelihood is unavailable in closed form, we approximate it using the Bayesian Information Criterion (BIC), following \cite{kass1995bayes}:
\bean\label{eq:postK}
\pi(K \mid \bbX_n) \approx \frac{\exp(-\mathrm{BIC}_K / 2) \, \pi(K)}{\sum_{k=1}^p \exp(-\mathrm{BIC}_k / 2) \, \pi(k)}.
\eean

In particular, Section \ref{ssec:BICcomp} details the computation of $\mathrm{BIC}_k$, and Section~\ref{sec:posterior-consistency} establishes the posterior contraction rate for \eqref{eq:postK}.


\subsection{Computation of the Bayesian Information Criterion}\label{ssec:BICcomp}

We now describe the procedure for computing the BIC to approximate the marginal likelihood. For a model with \( K \) spikes, the BIC is given by
\(
\mathrm{BIC}_K = -2 \hat{L}_K + d_K \log n,
\)
where \( \hat{L}_K \) denotes the maximized log-likelihood under the model with \( K \) spikes, and \( d_K \) is the number of free parameters. Let \( \bm{S}_n \) be the sample covariance matrix with ordered eigenvalues \( \hat{\lambda}_1 \ge \cdots \ge \hat{\lambda}_p \) and corresponding eigenvectors \( \hat{\bm{U}} \in \bbR^{p\times p}\). We retain the top \( K \) eigenvalues and approximate the remaining \( p - K \) eigenvalues with their average:
\(
\hat{c}_K = \sum_{k = K+1}^p \hat{\lambda}_k/(p-K).
\)
The covariance estimator becomes
\(
\hat{\bm{\Sigma}}_K = \hat{\bm{U}} \, \mathrm{diag}(\hat{\lambda}_1, \ldots, \hat{\lambda}_K, \hat{c}_K, \ldots, \hat{c}_K) \, \hat{\bm{U}}^\top,
\)
and the corresponding log-likelihood is given by
\[
\hat{L}_K = -\frac{n}{2} \log |\hat{\bm{\Sigma}}_K| - \frac{n}{2} \mathrm{tr}(\hat{\bm{\Sigma}}_K^{-1} \bm{S}_n) +\textup{constant}.
\]

The number of free parameters \( d_K \) includes three components: (i) \( pK - K(K+1)/2 \) degrees of freedom for the orthonormal matrix \( \bm{\Gamma} \in \mathbb{R}^{p \times K} \), accounting for orthogonality constraints; (ii) \( K \) parameters for the distinct spiked eigenvalues; and (iii) one parameter for the common remaining eigenvalues. Thus, the total parameter count is
\(
d_K = pK - K(K+1)/2 + K + 1.
\)
Omitting constants independent of \( K \), the BIC simplifies to
\[
\mathrm{BIC}_K = C + n \sum_{k=1}^K \log \hat{\lambda}_k + n(p - K) \log \hat{c}_K + d_K \log n,
\]
where \( C \) is a constant not depending on \( K \). 

\subsection{Posterior Consistency} \label{sec:posterior-consistency}

We study the asymptotic behavior of the BIC to establish the posterior consistency of \eqref{eq:postK}. Using this asymptotic result, we show that the posterior distribution \( \pi(K \mid \bbX_n) \) concentrates on the true number of spikes \( K_0 \) as \( n \to \infty \).

We first present the following asymptotic result for the BIC:
\begin{theorem}\label{thm:BIC}
Let \( \bm{X}_1, \ldots, \bm{X}_n \) be independent samples with \( \mathbb{E}(\bm{X}_i) = \bm{0} \) and \( \mathbb{E}(\bm{X}_i \bm{X}_i^\top) = \bm{\Sigma}_0 \), and suppose that Assumptions~\ref{assump:spike}--\ref{assump:scalegap} hold.  
Let \( K_0 \) denote the true number of spikes, and let \( \lambda_{0,k} \) denote the \( k \)-th eigenvalue of \( \bm{\Sigma}_0 \).  
Then there exists a constant \( C > 0 \) such that, with high probability and for all sufficiently large \( n \),
\[
\mathrm{BIC}_K - \mathrm{BIC}_{K_0} \ge
\begin{cases}
C n \sum_{k=K+1}^{K_0} \hat\lambda_k, & \text{if } K < K_0, \\
C (K - K_0) \dfrac{p \log n}{n}, & \text{if } K > K_0 \text{ and } K - K_0 = o(n).
\end{cases}
\]
\end{theorem}

The proof is given in the Appendix \ref{sec:pf4.1}. This result implies that \(\mathrm{BIC}_{K_0}\), the BIC for the true model, is asymptotically minimal among all \(K = o(n)\). From a Bayesian perspective, it is natural to place negligible prior mass on large values of \(K\), reflecting the assumption that the true number of spikes is much smaller than the sample size. As a result, models with \(K \gtrsim n\) contribute negligibly to the posterior distribution, and it is unnecessary to analyze their marginal likelihoods in detail. This justifies restricting asymptotic analysis to the sublinear regime \(K = o(n)\) when performing Bayesian inference on the number of spikes.

From the BIC-based posterior approximation
\(
\pi(K \mid \bbX_n) \propto \exp(-\mathrm{BIC}_K / 2) \, \pi(K),
\)
it follows that
\begin{equation*}
    \pi(K_0 \mid \bbX_n) = \left(1 + \sum_{K \ne K_0} \exp\left( -\tfrac{1}{2} (\mathrm{BIC}_K - \mathrm{BIC}_{K_0}) \right) \cdot \frac{\pi(K)}{\pi(K_0)} \right)^{-1}.
\end{equation*}

Assume the prior \(\pi(K)\) is supported on \(\{1, \dots, K_n\}\) with \(K_n = o(n)\), i.e., \(\pi(K) = 0\) for all \(K > K_n\). Then Theorem~\ref{thm:BIC} implies that each exponential term in the sum vanishes as \(n \to \infty\), provided the prior ratio \(\pi(K)/\pi(K_0)\) is not too large. This condition is satisfied, for instance, by a uniform prior over \(\{1, \dots, K_n\}\), or by an exponential prior \(\pi(K) \propto \exp(-\alpha K)\) for some fixed \(\alpha > 0\). Hence, even when \(p > n\), the posterior is consistent:
\[
\pi(K = K_0 \mid \bbX_n) \to 1 \quad \text{as } n \to \infty.
\]

\begin{remark}
\cite{bai2018consistency} study the consistency of BIC-type estimators for the number of significant components in high-dimensional PCA. Theorem 3.2 in \cite{bai2018consistency} establishes consistency under the classical BIC only when \( p/n < 1 \). To address the high-dimensional case \( p > n \), they propose a modified criterion called quasi-BIC (qBIC).  While qBIC achieves consistency in the high-dimensional regime, it does not directly approximate the marginal likelihood, which limits its interpretability as Bayesian model evidence. 
In particular, approximating the posterior distribution of \(K\), $\pi(K\mid \bbX_n)$, requires a marginal likelihood approximation such as BIC, rather than qBIC. 
Our result therefore provides the necessary theoretical foundation for Bayesian inference on the spike number in spiked covariance models under high-dimensional asymptotics, and constitutes a contribution toward fully Bayesian model selection in the regime of $p>n$.
\end{remark}

\section{Numerical studies}\label{sec:numerical}


We examine the performance of the proposed Bayesian methods through simulation studies and a real-data analysis of the S\&P~500 dataset.
The simulation studies investigate two main tasks: (i) estimation of spiked eigenvalues, evaluating both accuracy and computational efficiency; and (ii) estimation of the number of spikes, comparing the proposed approach with several existing methods. In particular, we illustrate the benefits of the Bayesian approach for spiked covariance estimation using the S\&P 500 data analysis. In this data analysis, our interest lies in estimating functionals of eigenvalues, such as the absorption ratio, which is a systemic risk measure that depends on multiple spiked eigenvalues as well as the number of spikes. Since uncertainty arises from both the eigenvalues and the number of spikes, it is necessary to account for them jointly. The Bayesian framework is well suited for this purpose, as it provides a coherent quantification of uncertainty for the eigenvalues as well as for functionals derived from them. Although the proposed model is capable of estimating eigenvectors, this is not the primary focus of this paper. Therefore, we omit detailed discussion in this section and present the corresponding results in the Appendix \ref{sec:addsims}. 


\subsection{Estimation of eigenvalues}\label{subsec:est_eigenval}

We conduct a simulation study to evaluate the bias correction methods for posterior eigenvalues proposed in Section \ref{ssec:biascorrect}.
Subsequently, we evaluate the accuracy and uncertainty quantification of these eigenvalues using 100 replicated data sets. Specifically, we compute the relative errors for the leading $k$-th eigenvalue, denoted as $\text{err}_\lambda := |\lambda_k(\bm{\Sigma})-\lambda_k(\bm\Sigma_0)|/\lambda_k(\bm\Sigma_0)$, where $\lambda_k(\bm{\Sigma})$ represents the estimated leading $k$-th eigenvalue of the covariance. 
The point estimates of the Bayesian method are given as the average of the $500$ posterior samples. 
As the interval estimation, we use the 95\% credible interval of posteriors for uncertainty estimation, since the frequentist methods are challenging to apply in this setting. The coverage probability (CP) is measured by determining how often the true parameters of interest fall within credible intervals (or confidence intervals)  across 100 replicates.  

We consider two high-dimensional settings for evaluating the proposed methods. In both settings, the number of spikes is fixed at \(K = 3\), and we examine combinations of the sample size \(n \in \{100, 500\}\) and the dimension \(p \in \{500, 1000\}\). In the first setting, synthetic datasets are generated from a multivariate normal distribution \(N_p(\mathbf{0}, \bm{\Sigma}_0)\),  where the true spiked covariance matrix is defined as $\bm{\Sigma}_0 = \mathrm{diag}(150, 100, 50, 1, \ldots, 1)$. Here, the spike strengths are determined so that the spiked eigenstructure of the covariance satisfies the conditions~\eqref{eq:spikecondition} and~\eqref{eq:seperatedeigens}. In particular, when \( n = 100 \) and \( p = 1000 \), the values of $d_1, d_2$, and $d_3$ are \( 0.0667 \), \( 0.1 \), and \( 0.2 \) respectively, satisfying the spiked eigenvalue condition.

The second setting follows \citet{wang2017asymptotics} and is based on the factor model \eqref{eq:fm}, where \(\bm{\Sigma}_0 = \mathbf{B}\mathbf{B}^\top + \bm{\Sigma}_u\). The loading matrix \(\mathbf{B}\) has rows sampled from a standard multivariate normal distribution, with the \( k \)-th column normalized such that  \( \lambda_k \) for \( k = 1, 2, 3 \). We set \( \lambda_1 = 50 \), \( \lambda_2 = 20 \), and \( \lambda_3 = 10 \). The idiosyncratic covariance matrix is diagonal, \(\bm{\Sigma}_u = \mathrm{diag}(\sigma_1^2, \ldots, \sigma_p^2)\), where each \(\sigma_i\) is independently drawn from a \(\mathrm{Gamma}(a, b)\) distribution with \(a = 150\) and \(b = 100\). In both settings, the number of factors \(K\) is assumed to be known, and this assumption is imposed across all competing methods. 
The inverse-Wishart prior is specified with hyperparameters \(\bm{A}_n = 0.1 \times I_p\) and \(\nu_n = 2p + 2\). 

We compared our proposed estimators using prior calibration (IW-PC) and post-hoc correction (IW-PHC) against the sample covariance (SC) as a reference estimator, as well as three additional approaches: the inverse-Wishart posterior without bias correction, having degrees of freedom $ n + 2p + 2$ (IW), shrinkage principal orthogonal complement
thresholding (SPOET) introduced by \cite{wang2017asymptotics}  and the posterior using the shrinkage inverse-Wishart (SIW) prior proposed by \cite{berger2020bayesian}. For SPOET, we utilized the implementations provided in the \pkg{POET} R package.

\begin{table}[t!]
\begin{center}
\caption{Average relative errors and coverage probabilities (CP) of the estimated eigenvalues over 100 replications under the first setting. NA indicates that the value is not available.}
{\scriptsize
\begin{tabular}{cccccccccccccccc}
\toprule\noalign{}
 & & & \multicolumn{2}{c}{SC} & \multicolumn{2}{c}{IW} &  \multicolumn{2}{c}{SPOET} & \multicolumn{2}{c}{SIW} & \multicolumn{2}{c}{IW-PHC} & \multicolumn{2}{c}{IW-PC} \\
\cmidrule(lr){4-5} \cmidrule(lr){6-7} \cmidrule(lr){8-9} \cmidrule(lr){10-11} \cmidrule(lr){12-13}  \cmidrule(lr){14-15} 
 & $n$ & $p$ & $\text{Err}_\lambda$ & CP & $\text{Err}_\lambda$ & CP & $\text{Err}_\lambda$ & CP  & $\text{Err}_\lambda$ & CP & $\text{Err}_\lambda$ & CP  & $\text{Err}_\lambda$ & CP\\
\midrule\noalign{}
\multirow{4}{*}{$\lambda_1$} & 100 & 500 &    0.1212 & NA & 0.1432 & 0.89 & 0.1122 & 0.92 & 0.1115 & 0.92 & 0.1120 & 0.90 & 0.1120 & 0.90 \\ 
         & 100 & 1000   & 0.1237 & NA & 0.1926 & 0.80 & 0.0947 & 0.98 & 0.0922 & 0.94 & 0.0972 & 0.96 & 0.0963 & 0.96 \\ 
       
 & 500 & 500 &0.0512 & NA & 0.0511 & 0.93 & 0.0512 & 0.96 & 0.0560 & 0.93 & 0.0509 & 0.93 & 0.0509 & 0.93 \\ 
         & 500 & 1000   &  0.0478 & NA & 0.0526 & 0.94 & 0.0473 & 0.98 & 0.0507 & 0.95 & 0.0482 & 0.94 & 0.0477 & 0.94 \\ 
          
\addlinespace
\multirow{4}{*}{$\lambda_2$} & 100 & 500 &0.1156 & NA & 0.1046 & 0.95 & 0.1071 & 0.94 & 0.1112 & 0.91 & 0.1006 & 0.88 & 0.0993 & 0.91 \\ 
         & 100 & 1000   &0.1113 & NA & 0.1580 & 0.84 & 0.0934 & 1.00 & 0.0974 & 0.95 & 0.0865 & 0.96 & 0.0877 & 0.95 \\ 
 & 500 & 500 &  0.0491 & NA & 0.0481 & 0.94 & 0.0493 & 0.96 & 0.0514 & 0.93 & 0.0487 & 0.93 & 0.0492 & 0.91 \\ 
         & 500 & 1000   &0.0462 & NA & 0.0463 & 0.96 & 0.0492 & 0.99 & 0.0526 & 0.96 & 0.0497 & 0.95 & 0.0502 & 0.95 \\ 
\addlinespace
\multirow{4}{*}{$\lambda_3$} & 100 & 500 & 0.1211 & NA & 0.1507 & 0.88 & 0.1064 & 0.97 & 0.1115 & 0.95 & 0.1039 & 0.95 & 0.1068 & 0.90 \\ 
          & 100  & 1000   & 0.1699 & NA & 0.3792 & 0.06 & 0.1038 & 1.00 & 0.1178 & 0.92 & 0.0853 & 0.95 & 0.0958 & 0.88 \\ 
 & 500 & 500 & 0.0548 & NA & 0.0570 & 0.90 & 0.0527 & 0.94 & 0.0532 & 0.93 & 0.0519 & 0.94 & 0.0520 & 0.91 \\ 
         & 500 & 1000   & 0.0565 & NA & 0.0728 & 0.81 & 0.0518 & 0.94 & 0.0539 & 0.91 & 0.0522 & 0.91 & 0.0514 & 0.90 \\ 
\bottomrule\noalign{}
\end{tabular}
}
\label{tab:eigval}
\end{center}
\end{table}

 Table \ref{tab:eigval} presents the results for the first setting (the second setting is reported in the Appendix \ref{sec:addsims}). For $n = 100$,  the proposed IW-PC and IW-PHC estimators offer superior estimation accuracy and reliable uncertainty quantification compared to the standard approaches SC and IW. These methods significantly reduce the bias inherent in the inverse-Wishart posterior, particularly under high-dimensional settings, as shown in Figure \ref{fig:bias_evals}. IW-PHC achieves the lowest or near-lowest estimation errors across all eigenvalues, outperforming other methods for $\lambda_2$ and $\lambda_3$. In contrast, the standard IW estimator exhibits substantial bias and poor coverage performance as the dimension increases. The SIW method shows overall comparable performance in estimation accuracy and coverage. For $n = 100$, it provides the most accurate estimate of the leading eigenvalue $\lambda_1$, but performs slightly worse than IW-PHC and IW-PC for $\lambda_2$ and $\lambda_3$. Although SPOET demonstrates strong performance in point estimation, its uncertainty quantification tends to be less reliable when  $p = 1000$. 
This could be attributed to its reliance on an asymptotic normal approximation, which may result in conservative or miscalibrated confidence intervals when the sample size is finite. Specifically, the confidence intervals for SPOET are derived from the following asymptotic distribution:
\begin{equation*}
    \sqrt{n}\left(\frac{\hat{\lambda}_i^{S}}{\lambda_{0,i}} - 1\right) \overset{d}{\longrightarrow} N(0, 2), \quad \text{provided that } \sqrt{p} = o(\lambda_{0,i}),
\end{equation*}
where $\hat{\lambda}_i^S$ denotes the shrinkage eigenvalue. For further details, see \cite{wang2017asymptotics}.

\begin{figure}
\centering
\includegraphics[width=\textwidth]{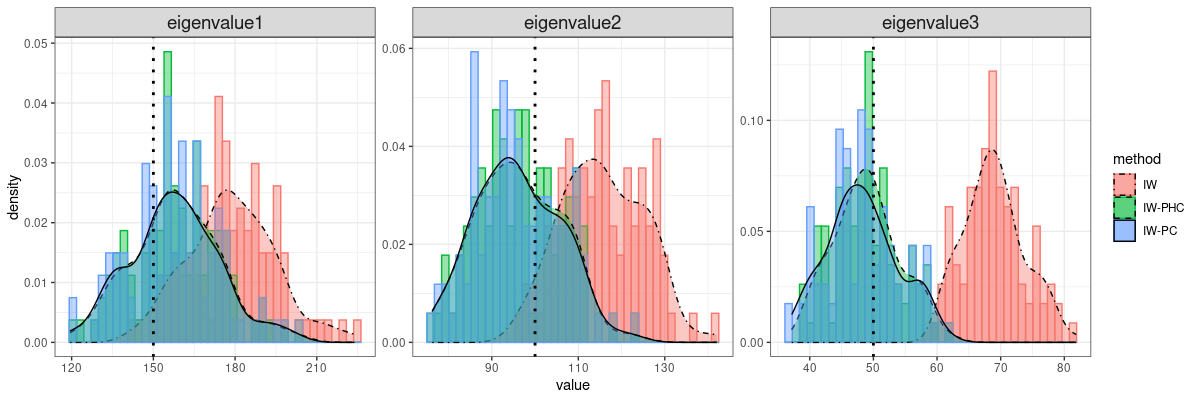}
\caption{\label{fig:bias_evals} Histograms and density plots of posterior means for the top three eigenvalues (\(\lambda_1 = 150\), \(\lambda_2 = 100\), \(\lambda_3 = 50\)) across 100 simulated datasets with \(p = 1000\). Vertical dotted lines indicate the true values. The density curves are shown with dot-dashed (IW), dashed (IW-PHC), and solid (IW-PC) lines.}
\end{figure}

When the sample size increases to $n = 500$, all methods demonstrate comparable performance in both estimation accuracy and uncertainty quantification across all eigenvalues. The relative errors are significantly reduced, and the coverage probabilities close to the nominal level, even for the IW estimator in some cases, which previously underperformed when $n = 100$. This indicates that the advantages of the proposed IW-PHC and IW-PC estimators are less noticeable as the sample size grows, likely due to the mitigation of high-dimensionality effects with more observations. Nevertheless, IW-PHC and IW-PC continue to exhibit competitive accuracy and coverage across all eigenvalues and dimensions.

Furthermore, we compare the average computation times of SIW and IW-PHC over 100 repetitions to assess the computational efficiency of the two Bayesian methods. For comparison, we ensure that the effective sample size of the leading eigenvalue is approximately 500. While the IW-PHC method produces 500 independent posterior samples, the SIW method relies on MCMC sampling to achieve a comparable number of effective samples. Specifically, we set the number of iterations to 60{,}000 for both \(p = 500\) and  \(p = 1000\) based on the convergence diagnostic. Table~\ref{tab:ess_time} shows that IW-PHC achieves both high efficiency (high ESS) and short computation time, being up to 3--7 times faster than SIW in estimating the eigenvalues and eigenvectors, whereas SIW attains high ESS for some leading eigenvalues (e.g. $\lambda_1$) but its efficiency deteriorates substantially as the dimension increases and for lower-order eigenvalues.
 These results suggest that the proposed method is more computationally efficient than SIW in high-dimensional settings and is readily applicable to real-world data. The computational advantage arises from the fact that the proposed method generates posterior samples independently, whereas SIW relies on MCMC sampling, which requires a large number of iterations to obtain the desired number of effective samples in high dimensions. In particular, since the proposed method generates posterior samples independently, it is amenable to parallel sampling. This feature is advantageous in high-dimensional settings, where substantial improvements in computational speed can be expected.

\begin{table}[t!]
\centering
\caption{Average effective sample size (ESS) and computation time (TIME, in seconds) for the Bayesian methods under the setting \(n = 100\). Each value is averaged over 100 repetitions.}
{\footnotesize
\begin{tabular}{cccccccc}
\toprule
 & & \multicolumn{3}{c}{\(p = 500\)} & \multicolumn{3}{c}{\(p = 1000\)} \\
\cmidrule(lr){3-5} \cmidrule(lr){6-8}
& & \(\lambda_1\) & \(\lambda_2\) & \(\lambda_3\) & \(\lambda_1\) & \(\lambda_2\) & \(\lambda_3\) \\
\midrule
\multirow{2}{*}{SIW} & ESS  & 555&  208 & 242 & 237 &113  &  98\\
                     & TIME &     & 803 &     &     &  3542&     \\
\midrule
\multirow{2}{*}{IW-PHC} & ESS  & 500 & 500 & 500 & 500 & 500 & 500 \\
                        & TIME &     & 121  &     &     & 1002  &     \\
\bottomrule
\end{tabular}
\label{tab:ess_time}
}
\end{table}

\subsection{Estimation of the number of spikes}
This subsection investigates how well the number of spikes can be estimated in a high-dimensional setting. We conduct experiments for selecting $K$ under the two spiked covariance structure settings introduced in the previous subsection.
Although the proposed method provides uncertainty quantification for \( K \), we focus on point estimation in this simulation study to enable comparison with existing methods, which do not quantify uncertainty. The practical utility of uncertainty quantification for \( K \) will be illustrated in the real data application in the next subsection.

We compare four existing methods for estimating the number of spikes $K$: the approach of \cite{ke2023estimation}, referred to as BEMA0; the information criterion proposed by \cite{bai2002determining}, denoted as $\text{IC}_{p2}$; the Bayesian model selection introduced by \cite{minka2000automatic}, termed ACPCA; and the eigenvalue ratio test (growth ratio) of \cite{ahn2013eigenvalue}, denoted as GR. Similar to the previous experiment, each setting is repeated 100 times. We report two evaluation metrics: the accuracy (ACC), defined as the proportion of replications where the number of spikes is correctly estimated, and average (AVG), which denotes the mean of the estimated number of spikes across replications.

\begin{table}[t!]
\begin{center}
\caption{Average number of estimated spikes (AVG) and accuracy (ACC) of spike number estimation across 100 replications for each method under the first and second spiked covariance structure settings.}
{\scriptsize
\begin{tabular}{cccccccccccccc}
\toprule\noalign{}
 & & & \multicolumn{2}{c}{BEMA0} & \multicolumn{2}{c}{$\text{IC}_{p2}$} & \multicolumn{2}{c}{GR} & \multicolumn{2}{c}{ACPCA} & \multicolumn{2}{c}{Proposed Method}  \\
\cmidrule(lr){4-5} \cmidrule(lr){6-7} \cmidrule(lr){8-9} \cmidrule(lr){10-11} \cmidrule(lr){12-13}
 & $n$ & $p$ & AVG & ACC &  AVG & ACC &  AVG & ACC &  AVG & ACC  &  AVG & ACC\\
\midrule\noalign{}
\multirow{4}{*}{Setting 1} & 100 & 500 &   3.00 & 1.00  & 3.00 & 1.00 &3.00 & 1.00 & 1.00 & 0.00 & 3.00 & 1.00  \\ 
         & 100 & 1000   &   3.01 & 0.99 &  2.91 & 0.91  & 3.00 & 1.00 & 1.00 & 0.00 & 2.89 & 0.88 \\
       
 & 500 & 500 &   3.00 & 1.00   & 3.00 &1.00 &3.00 &1.00 & 3.00 &1.00& 3.00 & 1.00 \\
         & 500 & 1000   &   3.00 & 1.00  & 3.00 &1.00 & 3.00 &1.00  & 3.00 &1.00 & 3.00 & 1.00 \\
          
\addlinespace
\multirow{4}{*}{Setting 2} & 100 & 500 &  3.01 & 0.99  & 3.00 & 1.00 & 2.49 & 0.6 & 3.00 & 1.00 & 3.00 & 1.00 \\ 
         & 100 & 1000   &   3.00 & 1.00 & 2.74 & 0.74 & 1.49 & 0.06 & 3.00 & 1.00 & 2.78 & 0.78 \\
       
 & 500 & 500 &   3.58 &0.45  &3.00 &1.00 &3.00 &1.00 &3.00 &1.00 & 3.00 & 1.00 \\
         & 500 & 1000   & 3.44 &0.59 & 3.00 &1.00&3.00 &1.00  & 3.00 &1.00& 3.00 & 1.00 \\
          
\bottomrule\noalign{}
\end{tabular}
}
\label{tab:ms}
\end{center}
\end{table}

Table~\ref{tab:ms} shows that the proposed method and $\text{IC}_{p2}$ exhibit high accuracy across most scenarios, whereas GR and BEMA0 display sensitivity to the sample size and the underlying spike structure. The proposed method generally demonstrates stable performance, though it slightly underestimates the spike number in the challenging $n=100$, $p=1000$ case. Specifically, for the first setting, most methods perform perfectly when $n$ is large, with ACC values close to 1.00, whereas ACPCA fails completely regardless of $p$ in the $n=100$ setting. For the second setting, performance differences are more pronounced: GR severely underestimates the spike number when $n$ is small, particularly for $p=1000$, whereas BEMA0 tends to overestimate when $n$ is large. Overall, these results confirm that the proposed method maintains high accuracy and stability across various settings, including challenging high-dimensional scenarios.

\subsection{Real data analysis}

In this subsection, we evaluate the performance of the proposed model using real financial data by analyzing the \textit{Absorption Ratio} (AR), a widely used measure of systemic risk in economics and finance. The Absorption Ratio, introduced by \cite{kritzman2010principal}, is defined as the fraction of total variance explained by the top $K$ eigenvalues of the covariance matrix:
\(
\mathrm{AR}(K) = {\sum_{k=1}^K \hat{\lambda}_k}/{\sum_{k=1}^p \hat{\lambda}_k},
\)
where $\hat{\lambda}_1 \geq \hat{\lambda}_2 \geq \cdots \geq \hat{\lambda}_p$ denote the estimated eigenvalues of the asset return covariance matrix.  A higher absorption ratio indicates that the market is more tightly connected and therefore more vulnerable to external shocks, whereas a lower absorption ratio suggests that the market is more dispersed and thus more resilient to disturbances.

We compute the absorption ratio based on the eigenvalues estimated by the proposed model (IW-PHC). Specifically, we use the posterior mean of the eigenvalues to calculate $\mathrm{AR}(K)$, where the number of factors $K$ is chosen as the maximum a posteriori (MAP) estimate from its posterior distribution. For comparison, we also compute the absorption ratios based on the inverse-Wishart posterior without bias correction (IW).

We collect monthly adjusted closing prices from January 1999 to December 2023 and focus on four sectors in the S\&P 500 index: \textit{Information Technology}, \textit{Financials}, \textit{Health Care}, and \textit{Energy}. To analyze the temporal dynamics of market co-movement, we first compute log returns based on the adjusted closing prices. Then, using a 12-month sliding window that shifts forward by one month at a time, we calculate the average absorption ratio within each window, thereby obtaining a time series that tracks the evolution of systemic risk throughout the study period.
\begin{figure}[ht!]
     \centering
       \includegraphics[width=\textwidth]{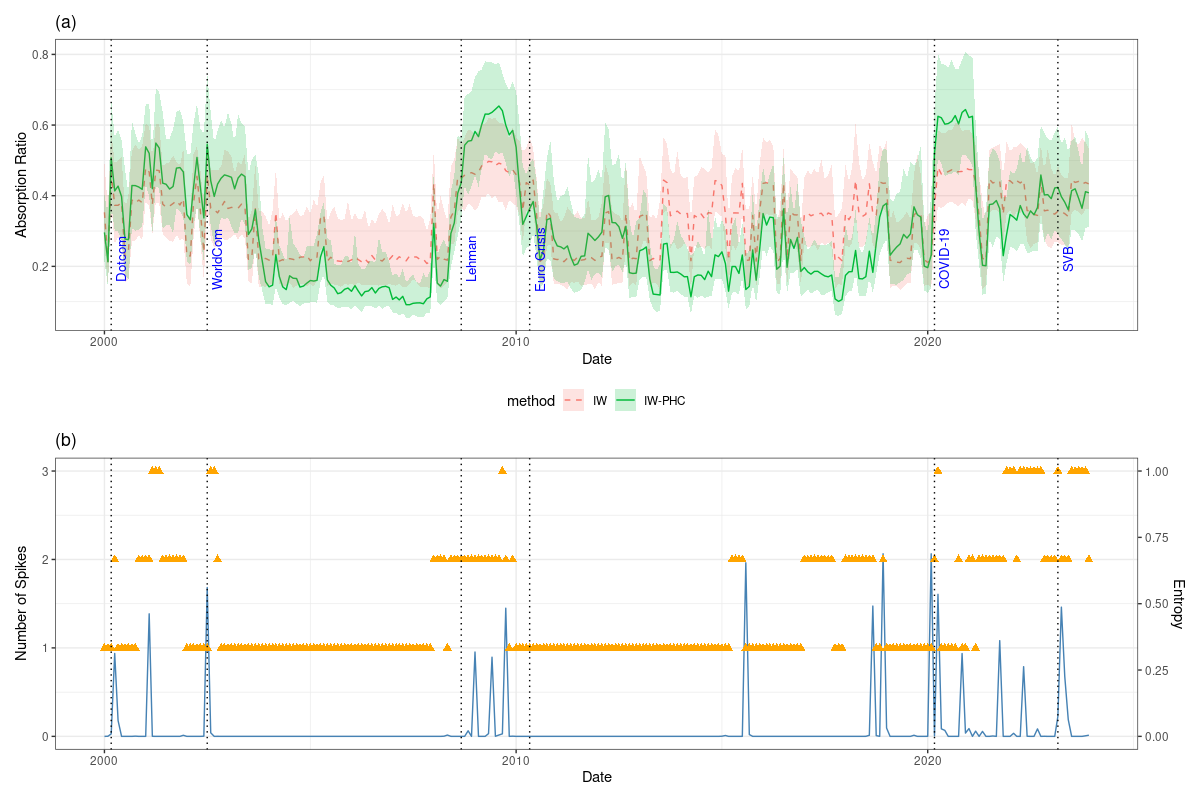}
       \caption{ (a) The absorption ratio (AR) of log returns, derived from the adjusted closing prices of the S\&P 500 over time. The dashed line represents the IW estimator, while the solid line denotes the bias-corrected IW-PHC estimator. Shaded regions indicate the 95\% credible intervals. (b) The estimated number of spikes (triangles, left axis) is shown alongside the posterior entropy (solid line, right axis). Vertical dotted lines indicate major U.S. financial events.}  
        \label{fig:real_results1}
\end{figure}

In Figure~\ref{fig:real_results1}(a), the IW method generally produces higher AR values during non-crisis periods, reflecting the eigenvalue inflation in high-dimensional settings. However, during major events---such as the dot-com bubble, the Lehman collapse, the European debt crisis, and the COVID-19 pandemic---the proposed IW-PHC method yields larger AR values than IW, capturing the heightened systemic risk more prominently in these turbulent periods. Figure~\ref{fig:real_results1}(b) shows that the number of spikes increases around crisis periods (e.g., 2000, 2008, 2020), indicating structural changes in the covariance structure of S\&P 500 stock returns. During stable periods, the number of spikes remains close to one. The entropy, derived from the posterior distribution of the number of spikes, quantifies the uncertainty in spike estimation: higher entropy reflects greater uncertainty, whereas lower entropy indicates more confident and stable inference. While entropy remains low throughout non-crisis periods, it rises sharply during financial crises, reflecting increased uncertainty in estimating the number of spikes due to structural shifts in the market under heightened systemic risk.

\section{Concluding remarks}\label{sec:concludingremark}

In this study, we developed a Bayesian framework for spiked covariance models, focusing on the posterior distribution of the spiked eigenvalues, eigenvectors, and the number of spikes \(K\).  
We employed the inverse-Wishart prior to derive the posterior distribution of the spiked eigenvalues and eigenvectors.  
However, since the eigenvalues from the inverse-Wishart posterior exhibit inflation in high-dimensional settings, we proposed two complementary bias-correction strategies: a prior calibration method that tunes the degrees-of-freedom parameter and a post-hoc correction method.  
We also introduced a Bayesian procedure for estimating \(K\) based on a BIC approximation. These approaches are theoretically supported by our eigenvalue perturbation analysis and posterior contraction results.  

A main advantage of our framework is that it enables the estimation of various functionals of the eigenvalues along  with uncertainty quantification, providing reliable inference beyond point estimation.  Another key computational advantage is that it generates independent posterior samples without relying on iterative MCMC algorithms. This independence allows for accurate posterior summaries using a relatively small number of samples, leading to a significant reduction in computational cost, especially in high-dimensional regimes. An interesting future direction is to extend the proposed methods to other structured covariance models, such as factor models or dynamic covariance structures.

\bibliographystyle{dcu}
\bibliography{main.bib}

@article{minka2000automatic,
  title={Automatic choice of dimensionality for PCA},
  author={Minka, Thomas},
  journal={Advances in neural information processing systems},
  volume={13},
  year={2000}
}

@article{hoyle2008automatic,
  title={Automatic PCA dimension selection for high dimensional data and small sample sizes.},
  author={Hoyle, David C},
  journal={Journal of Machine Learning Research},
  volume={9},
  number={12},
  year={2008}
}

@article{schrodinger1926quantisierung,
  title={Quantisierung als eigenwertproblem},
  author={Schr{\"o}dinger, Erwin},
  journal={Annalen der physik},
  volume={386},
  number={18},
  pages={109--139},
  year={1926},
  publisher={Wiley Online Library}
}

@book{bai2010spectral,
  title={Spectral analysis of large dimensional random matrices},
  author={Bai, Zhidong and Silverstein, Jack W},
  volume={20},
  year={2010},
  publisher={Springer}
}

@article{bao2013local,
  title={Local density of the spectrum on the edge for sample covariance matrices with general population},
  author={Bao, ZG and Pan, GM and Zhou, Wang},
  journal={preprint},
  year={2013}
}

@article{kass1995bayes,
  title={Bayes factors},
  author={Kass, Robert E and Raftery, Adrian E},
  journal={Journal of the american statistical association},
  volume={90},
  number={430},
  pages={773--795},
  year={1995},
  publisher={Taylor \& Francis}
}

@article{bao2015universality,
  title={Universality for the largest eigenvalue of sample covariance matrices with general population},
  author={Bao, Zhigang and Pan, Guangming and Zhou, Wang},
  journal={Annals of Statistics},
  volume={43},
  number={1},
  year={2015}
}

@book{messiah2014quantum,
  title={Quantum mechanics},
  author={Messiah, Albert},
  year={2014},
  publisher={Courier Corporation}
}

@article{fan2021robust,
  title={Robust high dimensional factor models with applications to statistical machine learning},
  author={Fan, Jianqing and Wang, Kaizheng and Zhong, Yiqiao and Zhu, Ziwei},
  journal={Statistical science: a review journal of the Institute of Mathematical Statistics},
  volume={36},
  number={2},
  pages={303},
  year={2021},
  publisher={NIH Public Access}
}

@article{wang2017asymptotics,
  title={Asymptotics of empirical eigenstructure for high dimensional spiked covariance},
  author={Wang, Weichen and Fan, Jianqing},
  journal={Annals of statistics},
  volume={45},
  number={3},
  pages={1342},
  year={2017},
  publisher={NIH Public Access}
}

@article{johnstone2009consistency,
  title={On consistency and sparsity for principal components analysis in high dimensions},
  author={Johnstone, Iain M and Lu, Arthur Yu},
  journal={Journal of the American Statistical Association},
  volume={104},
  number={486},
  pages={682--693},
  year={2009},
  publisher={Taylor \& Francis}
}

@book{horn1994topics,
  title={Topics in matrix analysis},
  author={Horn, Roger A and Johnson, Charles R},
  year={1994},
  publisher={Cambridge university press}
}

@article{ipsen2008perturbation,
  title={Perturbation bounds for determinants and characteristic polynomials},
  author={Ipsen, Ilse CF and Rehman, Rizwana},
  journal={SIAM Journal on Matrix Analysis and Applications},
  volume={30},
  number={2},
  pages={762--776},
  year={2008},
  publisher={SIAM}
}

@article{bai2007asymptotics,
  title={On Asymptotics of Eigenvectors of Large Sample Covariance Matrix},
  author={Bai, ZD and Miao, BQ and Pan, GM},
  journal={The Annals of Probability},
  pages={1532--1572},
  year={2007},
  publisher={JSTOR}
}

@article{ma2022posterior,
  title={On posterior consistency of Bayesian factor models in high dimensions},
  author={Ma, Yucong and Liu, Jun S},
  journal={Bayesian Analysis},
  volume={17},
  number={3},
  pages={901--929},
  year={2022},
  publisher={International Society for Bayesian Analysis}
}

@article{bishop1998bayesian,
  title={Bayesian {PCA}},
  author={Bishop, Christopher},
  journal={Advances in neural information processing systems},
  volume={11},
  year={1998}
}

@article{lee2023postecon,
  title={Post-processed posteriors for sparse covariances},
  author={Lee, Kwangmin and Lee, Jaeyong},
  journal={Journal of Econometrics},
  volume={236},
  number={1},
  pages={105475},
  year={2023},
  publisher={Elsevier}
}

@article{lee2023post,
  title={Post-processed posteriors for banded covariances},
  author={Lee, Kwangmin and Lee, Kyoungjae and Lee, Jaeyong},
  journal={Bayesian Analysis},
  volume={18},
  number={3},
  pages={1017--1040},
  year={2023},
  publisher={International Society for Bayesian Analysis}
}

@article{cai2020limiting,
  title={Limiting laws for divergent spiked eigenvalues and largest nonspiked eigenvalue of sample covariance matrices},
  author={Cai, T Tony and Han, Xiao and Pan, Guangming},
  journal={The Annals of Statistics},
  volume={48},
  number={3},
  pages={1255--1280},
  year={2020},
  publisher={JSTOR}
}

@article{ponnusamy2005foundations,
  title={Foundations of complex analysis},
  author={Ponnusamy, Saminathan},
  journal={(No Title)},
  year={2005}}

@book{ghosal2017fundamentals,
  title={Fundamentals of nonparametric Bayesian inference},
  author={Ghosal, Subhashis and Van der Vaart, Aad},
  volume={44},
  year={2017},
  publisher={Cambridge University Press}
}

@book{wainwright2019high,
  title={High-dimensional statistics: A non-asymptotic viewpoint},
  author={Wainwright, Martin J},
  volume={48},
  year={2019},
  publisher={Cambridge University Press}
}

@article{bai2002determining,
  title={Determining the number of factors in approximate factor models},
  author={Bai, Jushan and Ng, Serena},
  journal={Econometrica},
  volume={70},
  number={1},
  pages={191--221},
  year={2002},
  publisher={Wiley Online Library}
}

@article{cai2010optimal,
  title={Optimal rates of convergence for covariance matrix estimation},
  author={Cai, T Tony and Zhou, Harrison H },
  journal={The Annals of Statistics},
  volume={38},
  number={4},
  pages={2118--2144},
  year={2010},
  publisher={Institute of Mathematical Statistics}
}

@article{cai2013sparse,
  title={Sparse PCA: Optimal rates and adaptive estimation},
  author={Cai, T Tony and Ma, Zongming and Wu, Yihong and others},
  journal={The Annals of Statistics},
  volume={41},
  number={6},
  pages={3074--3110},
  year={2013},
  publisher={Institute of Mathematical Statistics}
}

@article{fan2013large,
  title={Large covariance estimation by thresholding principal orthogonal complements},
  author={Fan, Jianqing and Liao, Yuan and Mincheva, Martina},
  journal={Journal of the Royal Statistical Society: Series B (Statistical Methodology)},
  volume={75},
  number={4},
  pages={603--680},
  year={2013},
  publisher={Wiley Online Library}
}

@article{lee2018optimal,
	title={Optimal {B}ayesian minimax rates for inconstrained large covariance matrices},
	author={Lee, Kyoungjae and Lee, Jaeyong},
	journal={Bayesian Analysis},
	volume={13},
	number={4},
	pages={1211--1229},
	year={2018},
	publisher={International Society for Bayesian Analysis}
}

@article{ma2013sparse,
  title={Sparse principal component analysis and iterative thresholding},
  author={Ma, Zongming and others},
  journal={The Annals of Statistics},
  volume={41},
  number={2},
  pages={772--801},
  year={2013},
  publisher={Institute of Mathematical Statistics}
}

@article{yata2012effective,
  title={Effective PCA for high-dimension, low-sample-size data with noise reduction via geometric representations},
  author={Yata, Kazuyoshi and Aoshima, Makoto},
  journal={Journal of multivariate analysis},
  volume={105},
  number={1},
  pages={193--215},
  year={2012},
  publisher={Elsevier}
}

@article{yin1988limit,
  title={On the limit of the largest eigenvalue of the large dimensional sample covariance matrix},
  author={Yin, Yong-Qua and Bai, Zhi-Dong and Krishnaiah, Pathak R},
  journal={Probability theory and related fields},
  volume={78},
  number={4},
  pages={509--521},
  year={1988},
  publisher={Springer}
}

@article{ahn2013eigenvalue,
  title={Eigenvalue ratio test for the number of factors},
  author={Ahn, Seung C and Horenstein, Alex R},
  journal={Econometrica},
  volume={81},
  number={3},
  pages={1203--1227},
  year={2013},
  publisher={Wiley Online Library}
}

@inproceedings{hauberg2014grassmann,
  title={Grassmann averages for scalable robust PCA},
  author={Hauberg, Soren and Feragen, Aasa and Black, Michael J},
  booktitle={Proceedings of the IEEE Conference on Computer Vision and Pattern Recognition},
  pages={3810--3817},
  year={2014}
}

@article{bai2018consistency,
  title={Consistency of AIC and BIC in estimating the number of significant components in high-dimensional principal component analysis},
  author={Bai, Zhidong and Choi, Kwok Pui and Fujikoshi, Yasunori},
  journal={The Annals of Statistics},
  volume={46},
  number={3},
  pages={1050--1076},
  year={2018},
  publisher={JSTOR}
}

@article{lopes2004bayesian,
  title={Bayesian model assessment in factor analysis},
  author={Lopes, Hedibert Freitas and West, Mike},
  journal={Statistica Sinica},
  pages={41--67},
  year={2004},
  publisher={JSTOR}
}

@article{ke2023estimation,
  title={Estimation of the number of spiked eigenvalues in a covariance matrix by bulk eigenvalue matching analysis},
  author={Ke, Zheng Tracy and Ma, Yucong and Lin, Xihong},
  journal={Journal of the American Statistical Association},
  volume={118},
  number={541},
  pages={374--392},
  year={2023},
  publisher={Taylor \& Francis}
}

@article{kritzman2010principal,
  title={Principal components as a measure of systemic risk},
  author={Kritzman, Mark and Li, Yuanzhen and Page, Sebastien and Rigobon, Roberto},
  journal={Available at SSRN 1582687},
  year={2010}
}

@article{berger2020bayesian,
  title={Bayesian analysis of the covariance matrix of a multivariate normal distribution with a new class of priors},
  author={Berger, James O and Sun, Dongchu and Song, Chengyuan},
  journal={The Annals of statistics},
  volume={48},
  number={4},
  pages={2381--2403},
  year={2020},
  publisher={JSTOR}
}

\newpage

\appendix
\section{Concentration of eigenvalues under low-dimensional case}\label{ssec:eigenvalueineq}




In this section, we provide Theorem \ref{thm:loweigen} for the eigenvalue concentration inequality when $p||\bm\Sigma_0^{-1/2}\bm\Sigma\bm\Sigma_0^{-1/2} -\bm{I}_p||_2$ is small enough. 
Theorem \ref{thm:loweigen} is used for the proof of Theorem \ref{thm:higheigen} in the manuscript.
\begin{theorem}\label{thm:loweigen}

    Let $\lambda_1\ge \ldots \ge \lambda_p >0$ be the eigenvalues of $\bm\Sigma$ and let $d_1\ge \ldots \ge d_p>0$ be the eigenvalues of $\bm\Sigma_0$ with $ \min_{l=1,\ldots,p-1 }\frac{d_{l}}{d_{l+1}} >c $ for some constant $c>1$.
    If $p||\bm\Sigma_0^{-1/2}\bm\Sigma\bm\Sigma_0^{-1/2} -\bm{I}_p||_2 < \delta$ for some positive constant $\delta$ dependent on $c$, then
    \bea 
    \sup_{k=1,\ldots,p} \left|\frac{\lambda_k }{d_k} -1 \right| \le Cp ||\bm\Sigma_0^{-1/2}\bm\Sigma\bm\Sigma_0^{-1/2} - \bm{I}_p||_2,\\
    \sup_{k=1,\ldots,p} \left|\frac{\sqrt{\lambda_k} }{\sqrt{d_k}} -1 \right| \le C p ||\bm\Sigma_0^{-1/2}\bm\Sigma\bm\Sigma_0^{-1/2} - \bm{I}_p||_2,
    \eea     
    for some positive constant $C$ dependent on $c$.
\end{theorem}

The proof of Theorem \ref{thm:loweigen} is given below with the following lemma.

\begin{lemma}\label{lemma:dratio}
    Suppose $d_1 \ge d_2 \ge \ldots \ge d_p \ge 0$,  $\lambda >0$ and $ 1 \leq k \leq p$. If $ \min_{l=1,\ldots,p-1 }\frac{d_{l}}{d_{l+1}} >c $ 
    and $|\lambda/d_k -1 | < \frac{c-1}{2c}$ with $c>1$, 
    then 
    \bea 
    \lambda /d_{j} -1 &\ge& \frac{c-1}{2}, ~\textup{when } j>k,\\
    1- \lambda /d_{j}  &\ge&  \frac{c-1}{c} \Big( 1- \frac{1}{2c}\Big) , ~\textup{when } j<k.
    \eea 
    
\end{lemma}

\begin{proof}[Proof of Lemma \ref{lemma:dratio}]
    
    First, we consider $j>k$ which implies $d_k/d_j> c$. We have
    \bea 
    \lambda /d_{j} -1   &=& \frac{d_k}{ d_{j}} \left|\frac{\lambda}{d_k} \right| - 1 \\
    &\ge& c \left(1 -  \left|\frac{\lambda}{d_k} - 1 \right| \right) -1 \\
    &\ge& c\left(1-\frac{c-1}{2c} \right) -1 \\
    &=& \frac{c-1}{2}.
    \eea 
    Next, we consider $j<k$ which implies $d_j/d_k >c$. We have
    \bea 
    1-\lambda /d_{j}   &=& 1- \frac{\lambda}{d_{k}}\frac{d_k}{d_{j}}\\
    &\ge& 1-  \frac{1}{c} \left(1 + \left|\frac{\lambda}{d_{k}} -1\right| \right) \\
    &=& 1- \frac{1}{c} - \frac{1}{c}\left|\frac{\lambda}{d_{k}} -1 \right|\\
    &\ge& 1- \frac{1}{c} -\frac{c-1}{2c^2} \\
    &=& \frac{c-1}{c} \Big( 1- \frac{1}{2c}\Big) ,
    \eea 
    where the second inequality is satisfied by the given condition $|\lambda/d_k -1 |< (c-1)/(2c)$.
\end{proof}

\begin{lemma}\label{lemma:vineq}
Suppose the same setting and assumption on $d_1,\ldots, d_p$ and $\lambda$ in Lemma \ref{lemma:dratio}.
Let $\omega_{j,j}>0$ and $v_{j,j} = \omega_{j,j} -\lambda/d_{j} $, $j=1,\ldots, p$ and 
suppose $\max_{j=1,\ldots, p}|\omega_{j,j}-1 |\le c_1/2 $, where $c_1 = \frac{c-1}{2} \wedge\frac{c-1}{c} \Big( 1- \frac{1}{2c}\Big)$. 
Then,
    \bea
    \begin{cases}
        v_{j,j} \ge c_1/2, & j<k\\
        v_{j,j} \le -c_1/2, & j>k
    \end{cases}.
    \eea
\end{lemma}

\begin{proof}
First, we consider the case $j<k $.
    \bea 
    v_{j,j} &=& \omega_{j,j} -\lambda/d_{j}  \\
    &\ge& 1-\lambda/d_{j}   - |1-\omega_{j,j}| \\
    &\ge&  c_1- |1-\omega_{j,j}| \\
    &\ge& c_1/2,
    \eea 
    where the second inequality is satisfied by Lemma \ref{lemma:dratio}, and the third inequality is satisfied the assumption of $\omega_{j,j}$.
    The following inequality is also shown similarly when $j>k $.
    \bea 
    v_{j,j} &=& \omega_{j,j} -\lambda/d_{j}  \\
    &\le& 1-\lambda/d_{j}   + |1-\omega_{j,j}| \\
    &\le&  -c_1 +  |1-\omega_{j,j}|\\
    &\le& -c_1/2.
    \eea 
\end{proof}

\begin{lemma}\label{lemma:chracterineq}
Suppose the same setting and assumption of Lemma \ref{lemma:vineq}.
        Let $\bm{V}= \bm\Omega - \lambda \bm{D}^{-1}  = [\bm{v}_1,\ldots, \bm{v}_p]$, where $\bm{v}_j = (v_{j,1},\ldots, v_{j,p})^T$, and let $\bm{J}_k = \textup{diag}( v_{1,1}^{-1},\ldots, v_{k-1,k-1}^{-1}, 1,v_{k+1,k+1}^{-1},v_{p,p}^{-1})$ and $\tilde\phi(\lambda )  = \textup{det}(\bm{V}\bm{J}_k)$.
If $p||\bm\Omega - \bm{I}_p||_2\le c_1/8$, then
    \bea 
|\tilde\phi(\lambda ) - (\omega_{k,k} -\lambda/d_k)|   \le (12\vee 8/c_1) p ||\bm\Omega - \bm{I}_p||_2,
    \eea 
    where $\omega_{k,k}$ is the $k$th diagonal element of $\bm\Omega$.

\end{lemma}

\begin{proof}

First, we consider the case when $v_{k,k}\ge0 $.
We have $w_{k,k} -\lambda/d_k = \textup{det} ( \textup{diag}(\bm{V}) \bm{J}_k)$ where 
$ \textup{diag}(\bm{V}) \bm{J}_k = \textup{diag}(1,\ldots, 1, v_{k,k},1,\ldots, 1)$. 
We have 
    \bea
    |\tilde\phi(\lambda ) - (\omega_{k,k} -\lambda/d_k)|  &=& 
    |\textup{det}(\bm{V} \bm{J}_k) - \textup{det} ( \textup{diag}(\bm{V}) \bm{J}_k)  | .
    \eea

        Let $\bm{E} = (\bm{V} - \textup{diag}(\bm{V}) )\bm{J}_k$. Since the off-diagonal elements of $\bm{V}$ and $\bm\Omega$ are equals, 
\bean \label{eq:Eineq}
||\bm{E} ||_2 &=&
||(\bm\Omega - \textup{diag}(\bm\Omega))\bm{J}_k ||_2 \nonumber\\
&\le&   ||\bm\Omega - \textup{diag}(\bm\Omega)||_2  ||\bm{J}_k||_2 \nonumber\\
&=&   ||\bm\Omega -\bm{I}_p - \textup{diag}(\bm\Omega-\bm{I}_p)||_2  ||\bm{J}_k||_2 \nonumber\\
&\le&   (||\bm\Omega -\bm{I}_p||_2 +||\textup{diag}(\bm\Omega-\bm{I}_p)||_2  )||\bm{J}_k||_2 \nonumber\\
&\le& 2||\bm\Omega -\bm{I}_p||_2 ||\bm{J}_k||_2 \nonumber\\
&\le& \frac{2||\bm\Omega -\bm{I}_p||_2}{ (\min_{j\neq k} |v_{j,j}|) \wedge 1 }\nonumber\\
&\le& \frac{4||\bm\Omega -\bm{I}_p||_2}{ c_1 },
\eean
where the third inequality is satisfied because $||\textup{diag}(\bm\Omega-\bm{I}_p)||_2$ is the absolute maximum value of the diagonal element of $\textup{diag}(\bm\Omega-\bm{I}_p)$, which is smaller than or equal to $||\bm\Omega-\bm{I}_p||_2$.

Since $||\bm{E}||_2 <1/2$ by the condition of $||\bm\Omega -\bm{I}_p||_2$, 
Corollary 2.7 in \cite{ipsen2008perturbation} gives
\bea 
|\textup{det}(\bm{V} \bm{J}_k) - \textup{det} ( \textup{diag}(\bm{V}) \bm{J}_k)  | \le \sum_{j=1}^p s_{p-j} ||\bm{E}||_2^j ,
\eea 
where $s_{p-j} = \sum_{1\le i_1< \ldots < i_{p-j} \le p}  \sigma_{i_1}\ldots \sigma_{i_{p-j}}$ and $\sigma_j$ is the $j$th singular value of $\textup{diag}(\bm{V}) \bm{J}_k$.
Here, $\sigma_1= v_{k,k}$, $\sigma_2=\ldots = \sigma_p =1$ when $v_{k,k}\ge 1$, and $\sigma_{p}=v_{k,k}$, $\sigma_1=\ldots, \sigma_{p-1}=1$ when $v_{k,k}<1$.
Since
\bea 
|v_{k,k}| 
&\le&  |\omega_{k,k}-1| + |\lambda/d_k -1| \\
&\le& c_1/2 + \frac{c-1}{2c} \\
&\le& 3c_1/2,
\eea 
where the last inequality is satisfied since $(c-1)/(2c) \le (c-1)/2 $, $(c-1)/(2c) \le \frac{(c-1)}{c}\Big( 1-\frac{1}{2c} \Big) $ and consequently $(c-1)/(2c) \le (c-1)/2\wedge \frac{(c-1)}{c}\Big( 1-\frac{1}{2c} \Big) =c_1$,
we have
\bea 
s_{p-j}  &\le&
\sum_{1\le i_1< \ldots < i_{p-j} \le p}(3c_1/2\vee 1) \\
&\le & (3c_1/2\vee 1)\binom{p}{p-j} \\
&\le& (3c_1/2\vee 1) p^j.
\eea 

We have
\bea 
\sum_{j=1}^p s_{p-j} ||\bm{E}||_2^j &\le& (3c_1/2\vee 1) \sum_{j=1}^p (p||\bm{E}||_2)^j \\
&\le& (3c_1/2\vee 1) \frac{p||\bm{E}||_2}{1-p||\bm{E}||_2} \\
&\le& 2(3c_1/2\vee 1) p||\bm{E}||_2,
\eea 
where the last inequality is satisfied because condition $p||\bm\Omega - \bm{I}_p||_2\le c_1/8$ implies $p||\bm{E}||_2\le 1/2$.
Thus, we obtain
\bea 
    |\tilde\phi(\lambda ) - (\omega_{k,k} -\lambda/d_k)| 
    &\le& (3c_1\vee 2) p||\bm{E}||_2\\ 
    &\le& (12\vee 8/c_1) p ||\bm\Omega - \bm{I}_p||_2. 
    \eea 

When $v_{k,k}<0$, we have
\bea 
|\tilde\phi(\lambda ) - (\omega_{k,k} -\lambda/d_k)|  &=&
|\textup{det}(\bm{V}\bar{\bm{J}}_k) - \textup{det}(\textup{diag}(\bm{V})\bar{\bm{J}}_k)|,
\eea 
where $\bar{\bm{J}}_k = \textup{diag}(1,\ldots,1,-v_{k,k},1,\ldots, 1)$. 
Since every element of $\textup{diag}(\bm{V})\bar{\bm{J}}_k$ is larger than or equal to $0$, Corollary 2.7 in \cite{ipsen2008perturbation} gives 
\bea 
    |\tilde\phi(\lambda ) - (\omega_{k,k} -\lambda/d_k)| 
    &\le& (3c_1\vee 2) p||\bar{\bm{E}}||_2 \\
    &\le& (12\vee 8/c_1) p ||\bm\Omega - \bm{I}_p||_2,
    \eea 
    where $\bar{\bm{E}}=(\bm{V}-\textup{diag}(\bm{V}))\bar{\bm{J}}_k $ and the last inequality is satisfied because
    \bea 
    ||\bar{\bm{E}}||_2 &\le& ||(\bm\Omega - \textup{diag}(\bm\Omega))\bar{\bm{J}}_k||_2 \\
    &\le& ||\bm\Omega - \textup{diag}(\bm\Omega)||_2 ||\bar{\bm{J}}_k||_2 \\
    &=& ||\bm\Omega - \textup{diag}(\bm\Omega)||_2 ||\bm{J}_k||_2 \\
    &\le&  \frac{4||\bm\Omega -\bm{I}_p||_2}{ c_1 } ~ \textup{(See \eqref{eq:Eineq})}.
    \eea

\end{proof}

\begin{proof}[Proof of Theorem \ref{thm:loweigen}]
    The eigenvalues $\lambda_1,\ldots, \lambda_p$ are the roots of the characteristic polynomial $\textup{det}(\bm\Sigma - \lambda \bm{I}_p) $.
    The spectral decomposition gives $\bm\Sigma_0 = \bm{U}\bm{D}\bm{U}^T$, where $\bm{D}= \textup{diag}(d_1,\ldots, d_p)$. 
    Since $\textup{det}(\bm\Sigma - \lambda \bm{I}_p) = \textup{det}(\bm\Sigma_0)\textup{det}(\bm\Sigma_0^{-1/2}\bm\Sigma\bm\Sigma_0^{-1/2} - \lambda \bm{U}\bm{D}^{-1}\bm{U}^T) = 
    \textup{det}(\bm\Sigma_0)\textup{det}(\bm\Omega - \lambda\bm{D}^{-1})$, where $\bm\Omega = (\omega_{ij})_{1\leq i,j \leq p} = \bm{U}^T\bm\Sigma_0^{-1/2}\bm\Sigma\bm\Sigma_0^{-1/2} \bm{U}$,
    $\lambda_1,\ldots,\lambda_p$ are also the roots of  
\bea 
    \phi(\lambda ) = \textup{det} (\bm\Omega - \lambda \bm{D}^{-1})  = 0 .
    \eea 
    For arbitrary $k=1,\ldots, p$, we show that there exists a root of $\phi(\lambda ) $ in $A_k = \{ \lambda : |\lambda/d_k -1 | < (c-1)/(2c)\}$.
    Since $\phi(\lambda)$ is a continuous function, it suffices to show $\phi(\lambda_+) \phi(\lambda_-) <0 $, where $\lambda_+ = d_k\{1+ (c-1)/(2c)\}$ and $\lambda_- = d_k\{1- (c-1)/(2c)\}$.

    Let $\bm{V}= \bm\Omega - \lambda \bm{D}^{-1}  = [\bm{v}_1,\ldots, \bm{v}_p]$, where $\bm{v}_j = (v_{j,1},\ldots, v_{j,p})^T$, and let $$\bm{J}_k = \textup{diag}( 1/v_{1,1},\ldots, 1/v_{k-1,k-1}, 1,1/v_{k+1,k+1},1/v_{p,p}).$$ 

    We set $\delta = \frac{c-1}{2c\{ (12\vee 8/c_1) +1 \}} $.
    When $\lambda \in A_k$, 
    since $|\omega_{j,j}-1 |\le ||\bm\Omega -\bm{I}_p||_2 = ||\bm\Sigma_0^{-1/2}\bm\Sigma\bm\Sigma_0^{-1/2}- \bm{I}_p||_2 \le \delta \le c_1/2$, where $c_1$ is defined in Lemma \ref{lemma:vineq}, Lemma \ref{lemma:vineq} gives 
    \bean\label{eq:signsv}
    \begin{cases}
        v_{j,j} >0, & j<k\\
        v_{j,j} <0, & j>k
    \end{cases}.
    \eean 
    Thus, since $v_{j,j}$ is not zero when $j\neq k$, $\bm{J}_k$ is well-defined.

    Define
    \bea 
    \tilde\phi(\lambda) &=& 
    \textup{det}(\bm{V}\bm{J}_k)
    \\
    &=&
    \textup{det} ( (\bm\Omega - \lambda \bm{D}^{-1} )\bm{J}_k) \\
    &=& \prod_{j\neq k}v_{j,j}^{-1}\textup{det} ( \bm\Omega - \lambda \bm{D}^{-1} ) \\
    &=& \prod_{j\neq k}v_{j,j}^{-1} \phi(\lambda).
    \eea 
Then, by \eqref{eq:signsv}, it suffices to show $\tilde\phi(\lambda_+)\tilde\phi(\lambda_-) <0$. 

      When $\lambda_- \le \lambda \le \lambda_+$, since $ p||\bm\Omega - \bm{I}_p||_2 = p||\bm\Sigma_0^{-1/2}\bm\Sigma\bm\Sigma_0^{-1/2}-\bm{I}||_2\le  \delta \le c_1/8$, 
      Lemma \ref{lemma:chracterineq} gives
      \bea 
     \omega_{k,k} - \lambda/d_k  - C_1p ||\bm\Omega - \bm{I}_p||_2 \le \tilde\phi(\lambda)  \le \omega_{k,k} - \lambda/d_k  + C_1 p ||\bm\Omega - \bm{I}_p||_2,
    \eea 
    where $C_1= (12\vee 8/c_1)$.
    Then,
    \bea 
    \tilde\phi(\lambda_+) &\le&
     \omega_{k,k}  - (1+(c-1)/(2c)) + C_1 p ||\bm\Omega - \bm{I}_p||_2  \\
     &\le& 
    -(c-1)/(2c) +  (C_1p +1)||\bm\Omega - \bm{I}_p||_2\\
    &<& 0,\\
    \tilde\phi(\lambda_-)  &\ge&
    \omega_{k,k}  + (1+(c-1)/(2c)) - C_1p ||\bm\Omega - \bm{I}_p||_2  \\
     &\ge& 
    (c-1)/(2c) - (C_1p +1)||\bm\Omega - \bm{I}_p||_2\\
    &>& 0,
    \eea 
    where the last inequalities are satisfied by 
    $p||\bm\Sigma_0^{-1/2}\bm\Sigma\bm\Sigma_0^{-1/2}-\bm{I}||_2\le \frac{c-1}{2c\{ (12\vee 8/c_1) +1 \}}$.
    Thus, there exists $\lambda\in (\lambda_- ,\lambda_+)$ such that $\tilde\phi(\lambda )=0$, and the solution is denoted by $\hat\lambda$.
    That is, for each $k = 1, \ldots, p$, there exists at least one value of $\lambda$ such that $\tilde\phi(\lambda) = 0$ within the interval 
    $$(d_k\{1 -(c - 1)/(2c)\}, d_k\{1 + (c - 1)/(2c)\}),~ k=1,\ldots, p.$$
    Note that the maximum number of roots of the characteristic polynomial is $p$. Thus, if these intervals do not overlap for different values of $k$, then there exists only one root in each interval.
    
    Under the assumption that $\min_{l = 1, \ldots, p - 1} \frac{d_l}{d_{l+1}} > c,$ we have
    $$d_k \{1 + (c - 1)/(2c)\} \leq d_{k - 1} \{1 -(c - 1)/(2c)\},$$
    since $c > \dfrac{1 +(c - 1)/(2c)}{1 - (c - 1)/(2c)}$, which ensures that the intervals are disjoint. Therefore, $\hat\lambda$ is the $k$-th eigenvalue of the matrix $\bm\Omega - \lambda \bm{D}^{-1}$.
        
    We have
     \bea 
     |\hat\lambda/ d_k-1| &\le&  | \omega_{k,k} - \hat\lambda/ d_k| + ||\bm\Omega -\bm{I}_p||_2 \\
     &=&|\tilde\phi(\hat\lambda) -  (\omega_{k,k} - \hat\lambda/ d_k)| + ||\bm\Omega -\bm{I}_p||_2 \\
    &\le&     \{(12\vee 8/c_1) p +1\}||\bm\Omega - \bm{I}_p||_2.
    \eea 
    where the first equality is satisfied since $\tilde\phi(\hat\lambda)=0$ and the last inequality is satisfied by Lemma \ref{lemma:chracterineq}.
    Since we have proved the inequality for arbitrary $k\in \{1,\ldots, p\}$, we obtain
    \bea 
    \sup_{l=1,\ldots,p}|\frac{\lambda_l }{d_l} -1 | &\le&   \{(12\vee 8/c_1) p +1\}||\bm\Omega - \bm{I}_p||_2 \\
    &=&   \{(12\vee 8/c_1) p +1\}||\bm\Sigma_0^{-1/2}\bm\Sigma\bm\Sigma_0^{-1/2} - \bm{I}_p||_2 
    \eea 
    
    Finally, we give the upper bound of $\sup_{l=1,\ldots,p}\left|\frac{\sqrt{\lambda_l }}{\sqrt{d_l}} -1 \right|$.
    Let $\delta_0 = \{(12\vee 8/c_1) p +1\}||\bm\Sigma_0^{-1/2}\bm\Sigma\bm\Sigma_0^{-1/2} - \bm{I}_p||_2 $
    that is smaller than $1$ by the condition of $||\bm\Sigma_0^{-1/2}\bm\Sigma\bm\Sigma_0^{-1/2} - \bm{I}_p||_2$.
 Since $\frac{\lambda_k}{d_k} > 1- \delta_0$, we have 
    \bea 
    \left|\frac{\sqrt{\lambda_l}}{\sqrt{d_l}} - 1\right| &=&
    \frac{|\lambda_l/d_l - 1|}{\sqrt{\lambda_l}/\sqrt{d_l} + 1 } \\
    &\le& \frac{\delta_0}{1+ \sqrt{1-\delta_0}}\\
    &\le& \delta_0.
    \eea 
    
\end{proof}

\section{Proof of Theorem \ref{thm:higheigen}} \label{sec:pf3.3}

We give the proof of Theorem \ref{thm:higheigen} using the following lemma.

\begin{lemma}\label{lemma:sqrteigen}
    Let $\bm\Omega\in \calC_p$, and define $\bm\Gamma\in \bbR^{p\times K}$ and $\bm\Gamma_\perp \in \bbR^{p\times (p-K)}$ such that $\bar{\bm\Gamma} = [\bm\Gamma,\bm\Gamma_\perp]$ is an orthogonal matrix. Then, 
    \bea 
    \left|\sqrt{\lambda_k(\bm\Sigma)} - \sqrt{\lambda_k(\bm\Gamma^T\bm\Sigma\bm\Gamma)}\right| \le 
    ||\bm\Gamma_\perp^T \bm\Sigma \bm\Gamma_\perp ||^{1/2}, ~k=1,\ldots, K.
    \eea 
\end{lemma}
\begin{proof}[Proof of Lemma \ref{lemma:sqrteigen}]

    Since $\lambda_k(\bm\Sigma) = \lambda_k(\bar{\bm\Gamma}^T \bm\Sigma \bar{\bm\Gamma})$, 
    without loss of generality, it suffices to show 
    \bea 
    \left|\sqrt{\lambda_k(\bm\Sigma)} - \sqrt{\lambda_k(\bm\Sigma_{11})} \right| \le 
    ||\bm\Sigma_{22}||^{1/2}, 
    \eea 
where $\bm\Sigma = \begin{pmatrix}
        \bm\Sigma_{11} & \bm\Sigma_{12}\\
        \bm\Sigma_{21} & \bm\Sigma_{22}
    \end{pmatrix} \in \calC_p$ with $\bm\Sigma_{11}\in \calC_K$.

    We have 
    \bea 
    \lambda_k(\bm\Sigma_{11}) &=& \lambda_k \left( \begin{pmatrix}
        \bm\Sigma_{11} & \bm{O}\\
        \bm{O} & \bm{O}
    \end{pmatrix} \right) \\
    &=& \lambda_k \left( \begin{pmatrix}
        \bm{I}_K & \bm{O}\\
        \bm{O} & \bm{O}
    \end{pmatrix}\bm\Sigma\begin{pmatrix}
        \bm{I}_K & \bm{O}\\
        \bm{O} & \bm{O}
    \end{pmatrix} \right) \\
    &=& \left[\sigma_k\left(\bm\Sigma^{1/2}\begin{pmatrix}
        \bm{I}_K & \bm{O}\\
        \bm{O} & \bm{O}
    \end{pmatrix} \right)\right]^2.
    \eea 
The Weyl's inequality for singular values (Theorem 3.3.16 in \cite{horn1994topics}) gives
    \bea 
    \left|\sqrt{\lambda_k(\bm\Sigma)} - \sqrt{\lambda_k(\bm\Sigma_{11})} \right| &=&
    \left|\sigma_k\left(\bm\Sigma^{1/2} \right) - \sigma_k\left(\bm\Sigma^{1/2}\begin{pmatrix}
        \bm{I}_K & \bm{O}\\
        \bm{O} & \bm{O}
    \end{pmatrix} \right) \right|\\
    &\le& \left \| \bm\Sigma^{1/2}\begin{pmatrix}
         \bm{O}&\bm{O} \\
        \bm{O} & \bm{I}_{p-K}
    \end{pmatrix} \right \|_2 \\
    &=& || \bm\Sigma_{22}||_2^{1/2}.
    \eea 


\end{proof}

\begin{proof}[Proof of Theorem \ref{thm:higheigen}]

    First, we show that $\left|\frac{\sqrt{\lambda_l(\bm\Sigma) }}{\sqrt{\lambda_{0,l}}}  -1 \right| \le Ct$ when $ K || \bm\Gamma^T \bm\Omega \bm\Gamma  -\bm{I}_K ||_2\le t$ and $\frac{\sqrt{\lambda_{0,K+1}}||\bm\Gamma_\perp^T \bm\Omega \bm\Gamma_\perp ||^{1/2}}{\sqrt{\lambda_{0,l}} } \le t$ with $t \le \delta$ for some positive constant $\delta$ and $C$. 
    
    We have
    \bea 
    \left|\frac{\sqrt{\lambda_l(\bm\Sigma) }}{\sqrt{\lambda_{0,l}}}  -1 \right| &\le& 
    \left|\frac{\sqrt{\lambda_l(\bm\Sigma) } - \sqrt{\lambda_l(\bm\Gamma^T\bm\Sigma\bm\Gamma) }}{\sqrt{\lambda_{0,l}}} \right|+ 
    \left|\frac{\sqrt{\lambda_l(\bm\Gamma^T\bm\Sigma\bm\Gamma)}}{\sqrt{\lambda_{0,l}}} -1  \right| \\
    &\le& \frac{||\bm\Gamma_\perp^T \bm\Sigma \bm\Gamma_\perp ||^{1/2}}{\sqrt{\lambda_{0,l}} }+ \left|\frac{\sqrt{\lambda_l(\bm\Gamma^T\bm\Sigma\bm\Gamma)}}{\sqrt{\lambda_{0,l}}} -1 \right| ,
    \eea
    where the last inequality is satisfied by Lemma \ref{lemma:sqrteigen}.
    We have $|| \bm{D}^{-1/2} \bm\Gamma^T\bm\Sigma\bm\Gamma \bm{D}^{-1/2} -\bm{I}_K||_2 = || \bm\Gamma^T \bm\Omega \bm\Gamma  -\bm{I}_K ||_2 $ since  $\bm\Sigma_0^{-1/2} \bm\Gamma = \bm\Gamma \bm{D}^{-1/2}$, where 
    $\bm\Omega = \bm\Sigma_0^{-1/2}\bm\Sigma \bm\Sigma_0^{-1/2}$ and $\bm{D} = \bm\Gamma^T\bm\Sigma_0\bm\Gamma = \textup{diag}(\lambda_{0,1},\ldots, \lambda_{0,K})$.
    Then, we have $K|| \bm{D}^{-1/2} \bm\Gamma^T\bm\Sigma\bm\Gamma \bm{D}^{-1/2} -\bm{I}_K||_2 \le t$. Theorem \ref{thm:loweigen} gives, when $t$ is smaller than some positive constant dependent on $c$, 
    \bea 
    \left|\frac{\sqrt{\lambda_l(\bm\Gamma^T\bm\Sigma\bm\Gamma)}}{\sqrt{\lambda_{0,l}}} -1 \right| &\le& C_1 K|| \bm{D}^{-1/2} \bm\Gamma^T\bm\Sigma\bm\Gamma \bm{D}^{-1/2} -\bm{I}_K||_2  \\
    &=& C_1K || \bm\Gamma^T \bm\Omega \bm\Gamma  -\bm{I}_K ||_2,
    \eea 
    for some positive constant $C_1$ dependent on $c$.
    Since $\bm\Sigma_0^{-1/2}\bm\Gamma_\perp  = \bm\Gamma_\perp\bm{D}_\perp^{-1/2}$, where $\bm{D}_\perp = \textup{diag}(\lambda_{0,K+1},\ldots, \lambda_{0,p})$,
    \bea 
    ||\bm\Gamma_\perp^T \bm\Sigma \bm\Gamma_\perp  || &=&
    ||\bm{D}_\perp^{1/2}\bm{D}_\perp^{-1/2}\bm\Gamma_\perp \bm\Sigma \bm\Gamma_\perp \bm{D}_\perp^{-1/2} \bm{D}_\perp^{1/2}|| \\
    &\le& ||\bm\Gamma_\perp^T \bm\Sigma_0^{-1/2}  \bm\Sigma \bm\Sigma_0^{-1/2} \bm\Gamma_\perp ||_2 ||\bm{D}_\perp||_2 \\
    &=& ||\bm\Gamma_\perp^T \bm\Omega \bm\Gamma_\perp ||_2 \lambda_{0,K+1}.
    \eea 
    Collecting the inequalities, 
    \bean\label{eq:evalueineq} 
        \left|\frac{\sqrt{\lambda_l(\bm\Sigma) }}{\sqrt{\lambda_{0,l}}}  -1 \right|&\le& ||\bm\Gamma_\perp^T \bm\Omega \bm\Gamma_\perp ||_2^{1/2} \frac{\sqrt{\lambda_{0,K+1}}}{\sqrt{\lambda_{0,l}}}+   C_1K || \bm\Gamma^T \bm\Omega \bm\Gamma  -\bm{I}_K ||_2\nonumber\\
        &\le& (C_1+1)t.
    \eean 

    Next, we have
    \bean\label{eq:evalueineq2} 
    \left|\frac{\lambda_l(\bm\Sigma)}{\lambda_{0,l}} -1 \right| &\le&
     \left|\frac{\sqrt{\lambda_l(\bm\Sigma) }}{\sqrt{\lambda_{0,l}}}  -1 \right|
    \left|\frac{\sqrt{\lambda_l(\bm\Sigma)}}{\sqrt{\lambda_{0,l}}} +1 \right|  \nonumber\\
    &\le& \left|\frac{\sqrt{\lambda_l(\bm\Sigma) }}{\sqrt{\lambda_{0,l}}}  -1 \right|( 2 + (C_1+1)t) \nonumber\\
    &\le& C_2\left|\frac{\sqrt{\lambda_l(\bm\Sigma) }}{\sqrt{\lambda_{0,l}}}  -1 \right|\nonumber\\
    &\le& C_2 \left( ||\bm\Gamma_\perp^T \bm\Omega \bm\Gamma_\perp ||_2^{1/2} \frac{\sqrt{\lambda_{0,K+1}}}{\sqrt{\lambda_{0,l}}}+   C_1K || \bm\Gamma^T \bm\Omega \bm\Gamma  -\bm{I}_K ||_2 \right),
    \eean
    for some positive constant $C_2$ dependent on $c$.
    
    Then, we obtain 
    \bea 
    \sup_{l=1,\ldots, k}\left|\frac{\lambda_l(\bm\Sigma)}{\lambda_l(\bm\Sigma_0)} -1 \right| 
    &\le& C_2 (C_1+1)t.
    \eea 
    Thus, we obtain
    \bea 
    P\left( \sup_{l=1,\ldots, k}\left|\frac{\lambda_l(\bm\Sigma)}{\lambda_l(\bm\Sigma_0)} -1 \right| > C_2 (C_1+1) t \right) &\le& P(K||  \bm\Gamma^T\bm\Omega\bm\Gamma  -\bm{I}_K||_2 > t  ) \\
    &&+ P\left( \frac{\sqrt{\lambda_{0,K+1}}||\bm\Gamma_\perp^T \bm\Omega \bm\Gamma_\perp ||^{1/2}}{\sqrt{\lambda_{0,k}} }  > t \right),
    \eea 
    for all $t\le \delta$ for some positive constant $\delta$ dependent on $c$.


\end{proof}

\section{Proof of Theorem \ref{corr:IWeigenvalue}}\label{sec:pf3.4}
Next, we give the proof of Theorem \ref{corr:IWeigenvalue} using the following lemmas.
\begin{lemma}\label{lemma:blocknorm0}
Let  
    \bea 
    \bm{A} = \begin{pmatrix}
        \bm{A}_{11} & \bm{A}_{12}\\
        \bm{A}_{21} & \bm{A}_{22}
    \end{pmatrix} \in \calC_p,
    \eea 
    where $\bm{A}_{11}\in \calC_{p_1}$ and $\bm{A}_{22}\in \calC_{p_2}$. 
    Then, 
    \bea 
    ||\bm{A}||_2 &\le &2 (||\bm{A}_{11}||_2  + ||\bm{A}_{22}||_2 ).
    \eea

\end{lemma}

\begin{proof}
    Let $\bm{u}_1 = (\bm{1}_{p_1}^T ,\bm{0}_{p_2}^T)^T\in\bbR^p$ and $\bm{u}_2 = (\bm{0}_{p_1}^T ,\bm{1}_{p_2}^T)^T\in \bbR^p $. 
For any $\bm{v}\in \bbR^p$ with $||\bm{v}||_2 = 1$,
\bea 
\bm{v}^T \bm{A} \bm{v} &=& (\bm{v} \odot \bm{u}_1 + \bm{v} \odot \bm{u}_2 )^T \bm{A}(\bm{v} \odot \bm{u}_1 + \bm{v} \odot \bm{u}_2 )\\
&\le& (\bm{v} \odot \bm{u}_1 + \bm{v} \odot \bm{u}_2 )^T \bm{A}(\bm{v} \odot \bm{u}_1 + \bm{v} \odot \bm{u}_2 ) \\
&&+ (\bm{v} \odot \bm{u}_1 - \bm{v} \odot \bm{u}_2 )^T \bm{A}(\bm{v} \odot \bm{u}_1 -\bm{v} \odot \bm{u}_2 )\\
&= & 2(\bm{v} \odot \bm{u}_1 )^T \bm{A}(\bm{v} \odot \bm{u}_1 ) +2(\bm{v} \odot \bm{u}_2)^T \bm{A}(\bm{v} \odot \bm{u}_2 )\\
&\le& 2||\bm{A}_{11}||_2 + 2 ||\bm{A}_{22}||_2 ,
\eea 
where $\odot$ represents the Hadamard product.
Thus, 
\bea 
||\bm{A}||_2 \le 2||\bm{A}_{11}||_2 + 2 ||\bm{A}_{22}||_2 .
\eea

\end{proof}

\begin{lemma}\label{lemma:IWconcent}
    Suppose $\bm{Z}_1,\ldots, \bm{Z}_n$ are independent sub-Gaussian random vector with $E(\bm{Z}_i) = \bm{0} $ and $Var(\bm{Z}_i) = \bm{I}_p$, and consider the distribution of $\bm\Sigma$ given $\bbZ_n= (\bm{Z}_1,\ldots, \bm{Z}_n)$ as
    \bea 
    \bm\Sigma\mid \bbZ_n \sim  IW_p\left( \sum_{i=1}^n \bm{Z}_i \bm{Z}_i^T + \bm{A}_n , n+\nu_n \right).
    \eea 
    Let $\pi(\cdot\mid \bbZ_n)$ and $P(\cdot)$ denote the probabilities of $\bm\Sigma\mid \bbZ_n$ and $\bbZ_n$, respectively.
    If $||\bm{A}_n||=o(n)$, $\nu_n-2p = o(n)$ and $\nu_n>2p+1$, then there exists positive constants $C_1$, $C_2$, $C_3$ and $C_4$ such that
    \bea 
    P\left( \pi \left(  ||\bm\Sigma ||_2 \ge C_1 \Big(\frac{p}{n} \vee 1 \Big) +\epsilon_n\Bigm| \bbZ_n \right) >\delta_n \right) \le C_2\exp\{ -C_3n \min (\epsilon_n,\epsilon_n^2 ) \},
    \eea 
    for all $\epsilon_n>0$ and all sufficiently large $n$, where $\delta_n = 4 \exp \left( -C_4\frac{(n+\nu_n-2p-1)^2}{n+\nu_n-p-m-1} \right) + 4\exp(-C_4n)$.
\end{lemma}

\begin{proof}[Proof of Lemma \ref{lemma:IWconcent}]
    Let $\bm{S}_n=\sum_{i=1}^n \bm{Z}_i \bm{Z}_i^T /n$.
    Let $c_1$ denote a positive constant to be determined in this proof. We have
    \bean 
&&    P \left( \pi\left( ||\bm\Sigma ||_2 >  C_1\Big(\frac{p}{n}\vee 1\Big) +\epsilon_n\Bigm|\bbZ_n \right) >  \delta_n \right)   \nonumber\\&\le&
    P\left( ||\bm{S}_n ||_2 > c_1\Big(\frac{p}{n}\vee 1\Big) + \epsilon_n \right) \nonumber\\
    &&+ P\left(||\bm{S}_n ||_2 \le c_1\Big(\frac{p}{n}\vee 1\Big) + \epsilon_n ,\pi\left( ||\bm\Sigma ||_2 >  C_1\Big(\frac{p}{n}\vee 1\Big) + \epsilon_n \Bigm| \bbZ_n \right) >  \delta_n \right)  ,\label{eq:Sigmaupper1}
    \eean 
    and 
    \bea 
    P\left( ||\bm{S}_n ||_2 > c_1(\frac{p}{n}\vee 1)+\epsilon_n \right) &\le& 
     P\left( ||\bm{S}_n -\bm{I}_p||_2 > c_1(\frac{p}{n}\vee 1)-1 + \epsilon_n \right) \\
     &\le& P\left( ||\bm{S}_n -\bm{I}_p||_2 > \frac{c_1-1}{2}\Big(\frac{p}{n} + \sqrt{\frac{p}{n}} \Big) +\epsilon_n \right)\\ 
    &\le& C_2\exp\{ -C_3n \min ( \epsilon_n,\epsilon_n^2 ) \},
    \eea 
    for some positive constants $C_2$ and $C_3$, where the last inequality is satisfied by Theorem 6.5 in \cite{wainwright2019high} by setting $c_1$ larger than the positive constant appears in Theorem 6.5 of \cite{wainwright2019high}.
    The second inequality is satisfied by setting the constant $c_1$ to be larger than $1$.
    When $p>n$ and $c_1>1$,
    \bea 
     c_1(\frac{p}{n}\vee 1)-1  &\ge& c_1p/n   - p/n \\
     &=& \frac{(c_1-1)}{2} 2p/n\\
     &\ge&\frac{(c_1-1)}{2} (p/n  + \sqrt{p/n}),
    \eea 
    and, when $p\le n$ and $c_1>1$,
    \bea 
c_1(\frac{p}{n}\vee 1)-1  &=& \frac{(c_1-1)}{2} 2 \\
&\ge& \frac{(c_1-1)}{2} (p/n  + \sqrt{p/n}) .
    \eea

Next, we show $\eqref{eq:Sigmaupper1}=0$ for all sufficiently large $n$.
    Let $m$ be the number of non-zero eigenvalues of $\bm{S}_n$ and $m\le n$. 
    Let $\bm{S}_n = \bm{\hat{U}} \bm{\hat\Lambda}\bm{\hat{U}}^T $ by the spectral decomposition, where $\bm{\hat\Lambda} = \textup{diag}( \bm{\hat\Lambda}^{(1)}, \bm{O}_{p-m})$ and $\bm{\hat\Lambda}^{(1)}$ is a diagonal matrix consists of the nonzero eigenvalues of $\bm{S}_n$. 
    Let $\bm{\hat{U}}  = [ \bm{\hat{U}}^{(1)},\bm{\hat{U}}^{(2)}]  $ with $\bm{\hat{U}}^{(1)}\in \bbR^{p\times m}$.
    Since, by Lemma \ref{lemma:blocknorm0},
    \bea 
    \pi\left( ||\bm\Sigma ||_2 > C_1(\frac{p}{n}\vee 1)+\epsilon_n\Bigm|\bbZ_n \right) &\le& 
    \pi\left(  || \bm{(\hat{U}}^{(1)})^T \bm\Sigma \bm{\hat{U}}^{(1)}||_2 >C_1(\frac{p}{n}\vee 1)/4 +\epsilon_n/4\Bigm|\bbZ_n \right)  \\
    &&+ \pi \left(|| \bm{(\hat{U}}^{(2)})^T \bm\Sigma \bm{\hat{U}}^{(2)}||_2> C_1(\frac{p}{n}\vee 1)/4 +\epsilon_n/4\Bigm|\bbZ_n \right) ,
    \eea 
    we have
    \bean 
   \eqref{eq:Sigmaupper1}
    &\le& P\left(||\bm{S}_n || \le  c_1\Big(\frac{p}{n}\vee 1\Big) , \nonumber \right.\\
    &&~~~~\left.\pi\left(  || \bm{(\hat{U}}^{(1)})^T \bm\Sigma \bm{\hat{U}}^{(1)}||_2 >C_1(\frac{p}{n}\vee 1)/4 +\epsilon_n/4\Bigm|\bbZ_n \right)  > \delta_n/2 \right) \label{eq:upperU1}\\
    &&+ P\left(  \pi\left(  || \bm{(\hat{U}}^{(2)})^T \bm\Sigma \bm{\hat{U}}^{(2)}||_2 >C_1(\frac{p}{n}\vee 1)/4 +\epsilon_n/4\Bigm|\bbZ_n \right)  > \delta_n/2 \right).\label{eq:upperU2}
    \eean 

    First, we show $\eqref{eq:upperU1}=0$ for all sufficiently large $n$.
    We have
    \bea 
    \bm{(\hat{U}}^{(1)})^T \bm\Sigma \bm{\hat{U}}^{(1)} \mid\bbZ_n\sim IW_m \left( n \bm{\hat\Lambda}^{(1)} + (\bm{\hat{U}}^{(1)})^T \bm{A}_n \bm{\hat{U}}^{(1)} , n+\nu_n - 2p+2m \right).
    \eea 
    The spectral decomposition gives
    \bea 
\bm{\hat\Lambda}^{(1)} + (\bm{\hat{U}}^{(1)})^T \bm{A}_n \bm{\hat{U}}^{(1)}/n &=&
\tilde{\bm{U}} \tilde{\bm\Lambda} \tilde{\bm{U}}^T .
    \eea 
    Let $\bm\Omega_1$ be a random matrix with $\bm\Omega_1 \sim W_m((n+\nu_n-2p+m-1)^{-1}\bm{I}_m,n+\nu_n-2p+2m)$. 
    Then, 
    \bea 
    \bm{(\hat{U}}^{(1)})^T \bm\Sigma \bm{\hat{U}}^{(1)}  \equiv 
    \frac{n}{n+\nu_n-2p+m-1}\tilde{\bm{U}} \tilde{\bm\Lambda}^{1/2}\bm\Omega_1^{-1}\tilde{\bm\Lambda}^{1/2}\tilde{\bm{U}}^T,
    \eea 
    where $\equiv$ denotes equality in distribution.
    
    When $||\bm{S}_n||_2 \le c_1\Big(\frac{p}{n}\vee 1\Big) + \epsilon_n$ and $||\bm{A}_n||_2 = o(n)$, 
    \bea 
    ||\tilde{\bm\Lambda} ||_2 \le ||\bm{S}_n||_2 + ||\bm{A}_n||_2/n \le  2c_1(p/n \vee 1 +\epsilon_n),
    \eea 
    for all sufficiently large $n$. Then, there exists a positive constant $c_2$ such that
        \bea 
    &&\pi\left(||\bm{(\hat{U}}^{(1)})^T \bm\Sigma \bm{\hat{U}}^{(1)}||_2 >  C_1\left(\frac{p}{n}\vee 1 + \epsilon_n \right)/4\Bigm|\bbZ_n \right)  \\
    &\le& 
    \pi\left( 2c_1\left(\frac{p}{n}\vee 1 +\epsilon_n \right) \frac{n}{n+\nu_n-2p + m-1} ||\bm\Omega_1^{-1} ||_2 > C_1\left(\frac{p}{n}\vee 1 + \epsilon_n \right)/4 \Bigm|\bbZ_n \right) \\
    &\le& \pi\left( ||\bm\Omega_1^{-1} ||_2 >C_1/(8c_1) \mid\bbZ_n \right) \\
    &\le& \pi \left(\lambda_{\min} (\bm\Omega_1) < 8c_1/C_1 \mid\bbZ_n \right) \\
    &\le& 2\exp( -c_2n ) ,
    \eea 
    for all sufficiently large $n$, where the last inequality is satisfied by Lemma B.7 in \cite{lee2018optimal}.
    To apply the lemma, we set $C_1$ to satisfy $C_1 >(8c_1)4/(1-1/\sqrt{2})^2$, which gives $8c_1 /C_1\le (1-1/\sqrt{2})^2/4 \le ( 1- \sqrt{m/ (n+\nu_n-2p+m-1)})^2/4$ for all sufficiently large $n$ because $m/ (n+\nu_n-2p+m-1) \le m/(2m) =1/2 $.
    Thus, 
    \bea 
     \eqref{eq:upperU1} &\le&
      P\left(2\exp(-c_2n)\ge \pi\left(|| \bm{(\hat{U}}^{(1)})^T \bm\Sigma \bm{\hat{U}}^{(1)}||_2 >C_1\left(\frac{p}{n}\vee 1 \right)/4 +\epsilon_n/4\Bigm|\bbZ_n \right)  > \delta_n/2 \right),
    \eea 
    for all sufficiently large $n$, and this becomes $0$ by setting $\delta_n > 4\exp(-c_2n)$.
    

    Next, we show $\eqref{eq:upperU2}=0$ for all sufficiently large $n$.
    When $n\ge p$, $m= p$ with probability $1$. Thus, it suffices to show $\eqref{eq:upperU2}=0$ only when $p>n$.
    We have
    \bea 
(\bm{\hat{U}}^{(2)})^T \bm\Sigma \bm{\hat{U}}^{(2)} \mid\bbZ_n\sim IW_{p-m} \left( (\bm{\hat{U}}^{(2)})^T \bm{A}_n \bm{\hat{U}}^{(2)} , n+\nu_n - 2m \right).
    \eea 
    Let 
    \bea 
    \bm\Omega_2 \sim W_{p-m}\left( (n+\nu_n-p-m-1)^{-1} \bm{I}_{p-m},n+\nu_n-2m \right).
    \eea 
    We have
    \bea 
    \frac{1}{(n+\nu_n-p-m-1)}\{(\bm{\hat{U}}^{(2)})^T \bm{A}_n \bm{\hat{U}}^{(2)}\}^{1/2}\bm\Omega_2^{-1} \{(\bm{\hat{U}}^{(2)})^T \bm{A}_n \bm{\hat{U}}^{(2)}\}^{1/2} \equiv (\bm{\hat{U}}^{(2)})^T \bm\Sigma \bm{\hat{U}}^{(2)}.
    \eea 
    Thus,
    \bea 
    &&\pi\left(  || \bm{(\hat{U}}^{(2)})^T \bm\Sigma \bm{\hat{U}}^{(2)}||_2 >C_1\left(\frac{p}{n}\vee 1 \right)/4 +\epsilon_n/4\Bigm|\bbZ_n\right) \\
    &\le &\pi\left( ||\bm\Omega_2^{-1} ||_2 > C_1\left(\frac{p}{n}\vee 1 \right)\frac{n+\nu_n -p- m-1}{ ||\bm{A}_n||_2 }\Bigm|\bbZ_n \right) \\
    &\le &\pi \left(\lambda_{\min} (\bm\Omega_2) <C_1^{-1} \left(\frac{p}{n}\vee 1 \right)^{-1}\frac{ ||\bm{A}_n||_2 }{n+\nu_n -p- m-1}\Bigm|\bbZ_n \right).
    \eea 
We have
    \bea 
    \left(\frac{p}{n}\vee 1 \right)^{-1}\frac{ ||\bm{A}_n||_2 }{n+\nu_n -p-m-1} &=&
    \frac{ n^2||\bm{A}_n||_2/n }{p (n+\nu_n-p-m-1)} \\
    &\le& \frac{ n^2 }{(n+\nu_n-p-m-1)^2} \\
    &\le& \Big(\frac{ n+\nu_n-2p-1}{n+\nu_n-p-m-1}\Big)^2,
    \eea 
    for all sufficiently large $n$,
    and
    \bea 
    1-\sqrt{(p-m)/(n+\nu_n-p-m-1)} &=&
    \frac{1-(p-m)/(n+\nu_n-p-m-1)}{1+\sqrt{(p-m)/(n+\nu_n-p-m-1)} } \\
    &=& \frac{(n+\nu_n-2p-1)/(n+\nu_n-p-m-1)}{1+\sqrt{(p-m)/(n+\nu_n-p-m-1)} } \\
    &\ge& \frac{1}{2}\frac{n+\nu_n-2p-1}{n+\nu_n-p-m-1}.
    \eea 
    Then,
    \bea 
    C_1^{-1} \left(\frac{p}{n}\vee 1 \right)^{-1}\frac{ ||\bm{A}_n||_2 }{n+\nu_n -p- m-1} \le 
    C_1^{-1}4 \left(    1-\sqrt{(p-m)/(n+\nu_n-p-m-1)} \right)^2.
    \eea 
    When $C_1 > 16$, by Lemma B.7 in \cite{lee2018optimal},
    \bea 
    && \pi \left(\lambda_{\min} (\bm\Omega^{-1}) <C_1^{-1} \left(\frac{p}{n}\vee 1+\epsilon_n \right)^{-1}\frac{ ||\bm{A}_n||_2 }{n+\nu_n -p- m-1} \right)\\
    &\le& 2\exp( - (n+\nu_n-p-m-1) ( 1- \sqrt{(p-m)/(n+\nu_n-p-m-1)})^2/8)\\
     &\le& 2 \exp \left( -\frac{(n+\nu_n-2p-1)^2}{32(n+\nu_n-p-m-1)} \right),
    \eea 
    for all sufficiently large $n$. 
    Thus,
    \bea 
\eqref{eq:upperU2} &\le& P \left(  2 \exp \left( -\frac{(n+\nu_n-2p-1)^2}{32(n+\nu_n-p-m-1)} \right)  \right. \\
&&~~~~~~~\left.\ge \pi\left( || \bm{(\hat{U}}^{(2)})^T \bm\Sigma \bm{\hat{U}}^{(2)}||_2 >C_1\left(\frac{p}{n}\vee 1 \right)/4 +\epsilon_n/4\Bigm|\bbZ_n \right) \ge \delta_n/2 \right) .
    \eea 
    By setting $\delta_n/2 \ge 2 \exp \left( -\frac{(n+\nu_n-2p-1)^2}{32(n+\nu_n-p-m-1)} \right)$, $\eqref{eq:upperU2}=0$ for all sufficiently large $n$.

\end{proof}

\begin{lemma}\label{lem:IWeigenvalue}
Suppose 
\bea 
\bm\Sigma\mid \bbX_n &\sim & IW_p\left( \sum \bm{X}_i\bm{X}_i^T + \bm{A}_n , n+\nu_n \right).\\
\bm{X}_i &\stackrel{iid}{\sim} & N_p(\bm{0},\bm{\Sigma}_0 ) ,~i=1,\ldots, n.
\eea 
If $p/n^2=o(1)$, $||\bm\Sigma_0^{-1/2}\bm{A}_n \bm\Sigma_0^{-1/2}||_2 = o(n)$, $\nu_n-2p =o(n)$, $\nu_n>2p+1$ and $K =o(n)$, then
\bea 
\pi\left( ||\bm\Gamma_\perp^T \bm\Sigma_0^{-1/2}\bm\Sigma \bm\Sigma_0^{-1/2} \bm\Gamma_\perp ||^{1/2}\frac{\sqrt{\lambda_{0,K+1}}}{\sqrt{\lambda_{0,k}} } > M_n\frac{\sqrt{\lambda_{0,K+1}}}{\sqrt{\lambda_{0,k}} } \Big( \sqrt{\frac{p}{n}}\vee 1 \Big) \Bigm|\bbX_n \right) 
\eea 
and 
\bea 
\pi\left( ||  \bm\Gamma^T \bm\Sigma_0^{-1/2}\bm\Sigma \bm\Sigma_0^{-1/2}\bm\Gamma -\bm{I}_K||_2 > M_n\sqrt{\frac{K}{n}} \Bigm|\bbX_n \right) 
\eea 
converges to $0$ in probability for any positive sequence $M_n$ with $M_n\lra \infty$.

\end{lemma} 

\begin{proof}
We have
    \bea 
     \bm\Sigma_0^{-1/2} \bm\Sigma\bm\Sigma_0^{-1/2} \mid\bbX_n  &\sim& IW_{p} \left( \sum_{i=1}^n \bm{Z}_i\bm{Z}_i^T + \bm\Sigma_0^{-1/2}\bm{A}_n \bm\Sigma_0^{-1/2} ,n + \nu_n \right),\\
    \bm{Z}_i = \bm\Sigma_0^{-1/2} \bm{X}_i 
    &\sim&  N_{p} (\bm{0}, \bm{I}_p).
    \eea 
    
    We have
    \bea
&&\pi\left( ||\bm\Gamma_\perp^T \bm\Sigma_0^{-1/2}\bm\Sigma \bm\Sigma_0^{-1/2} \bm\Gamma_\perp ||^{1/2}\frac{\sqrt{\lambda_{0,K+1}}}{\sqrt{\lambda_{0,k}} } >  M_n\frac{\sqrt{\lambda_{0,K+1}}}{\sqrt{\lambda_{0,k}} } \Big( \sqrt{\frac{p}{n}}\vee 1 \Big) \Bigm| \bbX_n \right) \\
&\le& \pi\left( ||\bm\Gamma_\perp^T \bm\Sigma_0^{-1/2}\bm\Sigma \bm\Sigma_0^{-1/2} \bm\Gamma_\perp || >  M_n^2\Big(\frac{p}{n}\vee 1 \Big)\Bigm|\bbX_n\right).
    \eea
    By Lemma \ref{lemma:IWconcent},
    \bea 
    P\left(\pi\left( ||\bm\Gamma_\perp^T \bm\Sigma_0^{-1/2}\bm\Sigma \bm\Sigma_0^{-1/2} \bm\Gamma_\perp || >  M_n^2\Big(\frac{p}{n}\vee 1 \Big)\Bigm|\bbX_n \right) >\delta_n \right)
    \eea 
    converges to $0$, where $\delta_n = 2 \exp \left( -C_1\frac{(n+\nu_n-2p-1)^2}{n+\nu_n-p-m-1} \right) + 2\exp(-C_1n)$ for some positive constant $C_1$.
    Since $\delta_n\longrightarrow 0$, $\pi\left( ||\bm\Gamma_\perp^T \bm\Sigma_0^{-1/2}\bm\Sigma \bm\Sigma_0^{-1/2} \bm\Gamma_\perp || >  M_n^2\Big(\frac{p}{n}\vee 1 \Big)\Bigm|\bbX_n \right) $ converges to $0$ in probability.
    
Next, we have
\bea 
    \bm\Gamma^T\bm\Sigma_0^{-1/2} \bm\Sigma\bm\Sigma_0^{-1/2}\bm\Gamma \mid\bbX_n &\sim& IW_{K} \left(   \sum_{i=1}^n \bm\Gamma^T\bm{Z}_i\bm{Z}_i^T\bm\Gamma +\bm\Gamma^T\bm\Sigma_0^{-1/2}\bm{A}_n \bm\Sigma_0^{-1/2} \bm\Gamma,n+\nu_n -p +K \right), \\
    \bm\Gamma^T\bm{Z}_i 
    &\sim&  N_{K} (\bm{0}, \bm{I}_K).
    \eea   
    By Theorem 1 in \cite{lee2018optimal}, we have
    \bea
    P\left(\pi \left(||\bm\Gamma^T\bm\Sigma_0^{-1/2} \bm\Sigma\bm\Sigma_0^{-1/2}\bm\Gamma -\bm{I}_K||_2  >M_n \sqrt{K/n}\Bigm|\bbX_n\right) \right)&\lesssim& 
      \frac{1}{M_n^2}.
    \eea
    Thus,  
    \bea 
     \pi\left( ||  \bm\Gamma^T \bm\Sigma_0^{-1/2}\bm\Sigma \bm\Sigma_0^{-1/2}\bm\Gamma -\bm{I}_K||_2 > M_n\sqrt{\frac{K}{n}} \Bigm| \bbX_n \right) 
     \eea 
     converges to $0$ in probability.

\end{proof}

\begin{proof}[Proof of Theorem \ref{corr:IWeigenvalue}]

 
We show \eqref{eq:IWeigen1} by the following steps.
    For arbitrary $M_n$ with $M_n\lra \infty$, let $\tilde{M}_n = M_n \wedge 1/\sqrt{\epsilon_n}$. Then, $\tilde{M_n}\le M_n$ and $\tilde{M}_n \epsilon_n \lra 0$. 
    By Theorem \ref{thm:higheigen} with $t=\tilde{M}_n\epsilon_n$, there exists a positive constant $C_1$ such that
    \bean 
      &&\pi\left( \sup_{l=1,\ldots, k}\left|\frac{\lambda_l(\bm\Sigma) }{\lambda_{0,l}} -1 \right|  > M_n\epsilon_n \Bigm|\bbX_n \right) \nonumber\\
    &\le&
    \pi\left( \sup_{l=1,\ldots, k}\left|\frac{\lambda_l(\bm\Sigma) }{\lambda_{0,l}} -1 \right|  > \tilde{M_n}\epsilon_n \Bigm|\bbX_n \right) \nonumber\\
    &\le& 
    \pi \left( ||\bm\Gamma_\perp^T \bm\Sigma_0^{-1/2}\bm\Sigma \bm\Sigma_0^{-1/2} \bm\Gamma_\perp ||^{1/2}\frac{\sqrt{\lambda_{0,K+1}}}{\sqrt{\lambda_{0,k}} } > C_1\tilde{M}_n\frac{\sqrt{\lambda_{0,K+1}}}{\sqrt{\lambda_{0,k}} } \Big( \sqrt{\frac{p}{n}}\vee 1 \Big) \Bigm| \bbX_n \right) \nonumber\\
    &&+  \pi \left( K||  \bm\Gamma^T \bm\Sigma_0^{-1/2}\bm\Sigma \bm\Sigma_0^{-1/2}\bm\Gamma -\bm{I}_K||_2 > C_1\tilde{M}_n\sqrt{\frac{K^3}{n}} \Bigm|\bbX_n \right) ,\label{eq:biasnu1}
    \eean 
    for all sufficiently large $n$. 
    Here, Theorem \ref{thm:higheigen} can be applied because $\tilde{M_n}\epsilon_n$ is smaller than any arbitrary positive constant for all sufficiently large $n$.
    The upper bound converges to $0$ in probability by Lemma \ref{lem:IWeigenvalue}.
    
    Next, we consider the posterior contraction rate of $\lambda_l^{adj}(\bm\Sigma)$ with $l=1,\ldots, k$, i.e., we show \eqref{eq:IWeigen2}.
    We have 
    \bean 
      &&\pi\left( \sup_{l=1,\ldots, k}\left|\frac{\lambda_l^{adj}(\bm\Sigma) }{\lambda_{0,l}} -1 \right|  > M_n\epsilon_n^{(2)} \Bigm|\bbX_n \right) \nonumber\\
      &\le& \pi\left( \sup_{l=1,\ldots, k} 
\left|\frac{\lambda_l(\bm\Sigma) }{\lambda_{0,l}} -1 \right|  > \tilde M_n\epsilon_n /4\Bigm|\bbX_n \right) \label{eq:adjupper1}\\
&&+\pi\left( \sup_{l=1,\ldots, k} \left|\frac{\gamma_2(\lambda_k(\bm S_n),\hat c)}{\tilde\gamma_1(\nu_n,\lambda_k(\bm S_n),\hat c)}
-1 \right|  > M_n\epsilon_n^{(2)} /2\Bigm|\bbX_n \right)\label{eq:adjupper2},
    \eean 
    for all sufficiently large $n$. The convergence of \eqref{eq:adjupper1} can be shown by the following the steps of \eqref{eq:biasnu1}.

    Since $\frac{(p-K)\hat c}{n\lambda_k(\bm S_n)} - \frac{(p-K)\bar c}{n\lambda_k(\bm\Sigma_0)}  = O_p(n^{-1})$ (See Lemma 7 in \cite{yata2012effective} and \cite{wang2017asymptotics}),
    we have 
    \bea 
    \tilde\gamma_1(\nu_n,\lambda_k(\bm S_n),\hat c) = \frac{n}{n+\nu_n-2p-2} \Big[ 1+ \frac{(p-K)\hat c}{(n+\nu_n-2p-2)\lambda_k (\bm S_n)} \Big]= O_p(1),
    \eea 
    and
\bea 
&&\Big|\frac{\gamma_2(\lambda_k(\bm S_n),\hat c)}{\tilde\gamma_1(\nu_n,\lambda_k(\bm S_n),\hat c)}
-1 \Big| \\
&=& \frac{1}{\tilde\gamma_1(\nu_n,\lambda_k(\bm S_n),\hat c)} \Big| \frac{\nu_n-2p-2}{n+\nu_2-2p-2} - \frac{\hat c p}{n\lambda_k(\bm S_n)}  - \frac{n}{n+\nu_n-2p-2} \frac{(p-K)\hat c}{(n+\nu_n-2p-2)\lambda_k(\bm S_n)} \Big|\\
&\lesssim& \frac{\nu_n-2p-2}{n+\nu_2-2p-2} + \frac{\bar c p}{n\lambda_k(\bm\Sigma_0)} .
\eea 
Since $\epsilon_n^{(2)}$ has the term $\frac{\nu_n-2p-2}{n+\nu_2-2p-2}  + \frac{\bar c p}{n\lambda_{0,k}}$, 
\eqref{eq:adjupper2} converges to $0$ in probability.

\end{proof}

\section{Proof of Theorem \ref{thm:highevector}} \label{sec:pf3.5}
We give the proof of Theorem \ref{thm:highevector}. For the proof of Theorem \ref{thm:highevector}, we provide Lemmas \ref{lemma:evector}-\ref{lemma:A-B}.

\begin{lemma}\label{lemma:evector}
    Let $\bm\Sigma$ and $\bm\Sigma_0$ denote $p\times p$ positive-definite matrices.
    Let $\lambda_k$ and $\lambda_{0,k}$ denote the $k$th eigenvalues of $\bm\Sigma$ and $\bm\Sigma_0$, respectively, and let $\bm{u}_k$ and $\bm{u}_{0,k}$ denote the corresponding eigenvectors. 
    If $\lambda_{0,k} ||\bm\Lambda_{0,-k}^{-1/2}(\bm{U}_{0,-k}^T \bm\Omega \bm{U}_{0,-k} - \lambda_k \bm\Lambda_{0,-k}^{-1})^{-1} \bm{U}_{0,-k}^T\bm\Omega  \bm{u}_{0,k}||_2^2 <1$, 
    then
    \bea 
   (\bm{u}_{k}^T \bm{u}_{0,k})^2 = 
    \frac{1}{1 +  || \sqrt{\lambda_{0,k}}\bm\Lambda_{0,-k}^{-1/2}(\bm{U}_{0,-k}^T \bm\Omega \bm{U}_{0,-k} - \lambda_k \bm\Lambda_{0,-k}^{-1})^{-1} \bm{U}_{0,-k}^T \bm\Omega \bm{u}_{0,k}||_2^2}
    \eea 
    where 
      $\bm\Omega=\bm\Sigma_0^{-1/2} \bm\Sigma \bm\Sigma_0^{-1/2} $,
$\bm{U}_{0,-k} = [\bm{u}_{0,1},\ldots,\bm{u}_{0,k-1},\bm{u}_{0,k+1},\ldots, \bm{u}_{0,p}] \in \bbR^{p\times (p-1)}$ and $\bm\Lambda_{0,-k} = \textup{diag}(\lambda_{0,1},\ldots, \lambda_{0,k-1},\lambda_{0,k+1},\ldots, \lambda_{0,p})$.
\end{lemma} 

\begin{proof}[Proof]
    Let $\bm{U} = [\bm{u}_1,\ldots, \bm{u}_p]$ and $\bm\Lambda= \textup{diag}(\lambda_1,\ldots, \lambda_p)$. The $\bm{U}_0$ and $\bm\Lambda_0$ are defined similarly.
    For any $z\in \bbC$, we have
    \bea 
    (\bm\Sigma - z\bm{I}_p)^{-1} 
    &=& \bm{U} (\bm\Lambda -  z\bm{I}_p)^{-1}\bm{U}^T \\
    &=& \sum_{i=1}^p \frac{1}{\lambda_i - z} \bm{u}_i\bm{u}_i^T,
    \eea 
    and 
    \bean\label{eq:evec1}
\bm{u}_{0,k}^T(\bm\Sigma - z\bm{I}_p)^{-1}\bm{u}_{0,k} &=&
\sum_{i=1}^p \frac{1}{\lambda_i - z} \bm{u}_{0,k}^T\bm{u}_i\bm{u}_i^T\bm{u}_{0,k}.
    \eean 
    Let $\gamma_k$ be a simple closed curve in $\bbC$ containing only $\lambda_k$ among $\{\lambda_1,\ldots, \lambda_ p\}$.
    The Cauchy's residue theorem gives 
    \bean\label{eq:evec2} 
    \oint_{\gamma_k}\bm{u}_{0,k}^T(\bm\Sigma - z\bm{I}_p)^{-1}\bm{u}_{0,k} dz = - 2\pi i\bm{u}_{0,k}^T\bm{u}_k\bm{u}_k^T\bm{u}_{0,k}.
    \eean 
    Since $\bm\Sigma -z\bm{I}_p  =   \bm{U}_0 (\bm\Lambda_0^{1/2} \bm{U}_0^T \bm\Sigma_0^{-1/2} \bm\Sigma \bm\Sigma_0^{-1/2} \bm{U}_0 \bm\Lambda_0^{1/2} - z\bm{I}_p) \bm{U}_0^T$, we obtain 
    \bean \label{eq:evec3} 
    \oint_{\gamma_k}\bm{e}_{k}^T(\bm\Lambda_0^{1/2} \bm{U}_0^T \bm\Omega \bm{U}_0 \bm\Lambda_0^{1/2} - z\bm{I}_p)^{-1}\bm{e}_{k} dz = - 2\pi i\bm{u}_{0,k}^T\bm{u}_k\bm{u}_k^T\bm{u}_{0,k},
    \eean 
    where $\bm\Omega=\bm\Sigma_0^{-1/2} \bm\Sigma \bm\Sigma_0^{-1/2} $ and $\bm{e}_k$ is the $k$th standard coordinate vector in $\bbR^p$. 
    The $\bm{e}_{k}^T(\bm\Lambda_0^{1/2} \bm{U}_0^T \bm\Omega \bm{U}_0 \bm\Lambda_0^{1/2} - z\bm{I}_p)^{-1}\bm{e}_{k}$ 
    equals to the $(1,1)$ element of   
    \bea 
\begin{pmatrix}
    \lambda_{0,k} \bm{u}_{0,k}^T \bm\Omega \bm{u}_{0,k} -z  & \lambda_{0,k}^{1/2} \bm{u}_{0,k}^T \bm\Omega \bm{U}_{0,-k} \bm\Lambda_{0,-k}^{1/2} \\
    (\lambda_{0,k}^{1/2} \bm{u}_{0,k}^T \bm\Omega \bm{U}_{0,-k} \bm\Lambda_{0,-k}^{1/2} )^T & \bm\Lambda_{0,-k}^{1/2} \bm{U}_{0,-k}^T \bm\Omega \bm{U}_{0,-k}\bm\Lambda_{0,-k}^{1/2} -z \bm{I}_{p-1}
\end{pmatrix}^{-1},
\eea 
which is $\{    \lambda_{0,k} \bm{u}_{0,k}^T \bm\Omega \bm{u}_{0,k} -z -\lambda_{0,k} \bm{u}_{0,k}^T \bm\Omega \bm{U}_{0,-k} \bm\Lambda_{0,-k}^{1/2} (\bm\Lambda_{0,-k}^{1/2} \bm{U}_{0,-k}^T \bm\Omega \bm{U}_{0,-k}\bm\Lambda_{0,-k}^{1/2} -z \bm{I}_{p-1})^{-1}
\bm\Lambda_{0,-k}^{1/2} \bm{U}_{0,-k}^T \bm\Omega  \bm{u}_{0,k} \}^{-1}$. 
Since, by \eqref{eq:evec1} and \eqref{eq:evec2},
\bea 
\bm{u}_{k}^T\bm{u}_{0,k}\bm{u}_{0,k}^T\bm{u}_{k}  = -\frac{1}{2\pi i} 
\oint_{\gamma_k} \sum_{i=1}^p \frac{1}{\lambda_i - z} \bm{u}_{0,k}^T\bm{u}_i\bm{u}_i^T\bm{u}_{0,k} dz,
\eea 
$\lambda_k$ is the only singular point inside of curve $\gamma_k$. 
Then,
\bea 
\bm{u}_{k}^T\bm{u}_{0,k}\bm{u}_{0,k}^T\bm{u}_{k} 
  &=& 
  -\frac{1}{2\pi i}\oint_{\gamma_k} \frac{1}{h(z)} dz \\
  &=& -\textup{Res}(\frac{1}{h(z)},\lambda_k),
  \eea 
  where, by \eqref{eq:evec3}, $h(z)$ is defined as below:
    \bea 
h(z) &=&   \lambda_{0,k} \bm{u}_{0,k}^T \bm\Omega \bm{u}_{0,k} -z \\
&&-\lambda_{0,k} \bm{u}_{0,k}^T \bm\Omega \bm{B}_{k} (\bm{B}_{k}^T \bm\Omega \bm{B}_{k}-z \bm{I}_{p-1})^{-1}
\bm{B}_{k}^T \bm\Omega  \bm{u}_{0,k}, 
    \eea 
    where $\bm{B}_k =  \bm{U}_{0,-k}\bm\Lambda_{0,-k}^{1/2}$.
  
   We obtain $\textup{Res}(\frac{1}{h(z)},\lambda_k)$ using Theorem 8.13 of \cite{ponnusamy2005foundations}, which states  
   \bea 
    \textup{Res}(\frac{1}{h(z)},\lambda_k) = \frac{1}{h'(\lambda_k)},
   \eea 
   when $h(z)$ has a simple zero at $\lambda_k$. 
    We have
    \bea 
    \frac{d h(z)}{dz} &=& -1 - [\nabla_{\textup{vec}(\bm{X})}  g(\bm{X}) ]^T\frac{d\textup{vec}(\bm{X})}{dz},\\
\bm{X}    &=&  \bm{B}_{k}^T \bm\Omega\bm{B}_{k} -z \bm{I}_{p-1} ,\\
    g(\bm{X})&=&  \bm{v}^T \bm{X}^{-1} \bm{v},
\eea 
where $\bm{v} =  \sqrt{\lambda_{0,k}} \bm{B}_{k}^T\bm\Omega\bm{u}_{0,k}$,
and 
\bea 
\nabla_{\textup{vec}(\bm{X})}  g(\bm{X})  &=& 
- \textup{vec}(\bm{X}^{-1} \bm{v} \bm{v}^T \bm{X}^{-1} ) ,\\
\frac{d\textup{vec}(\bm{X})}{dz} &=&- \textup{vec}(\bm{I}_{p-1}) ,\\
\textup{vec}(\bm{X}^{-1} \bm{v} \bm{v}^T \bm{X}^{-1} )^T \textup{vec}(\bm{I}_{p-1}) &=& \textup{tr} ( \bm{X}^{-1} \bm{v} \bm{v}^T \bm{X}^{-1} ) ,\\
&=& \bm{v}^T \bm{X}^{-2}  \bm{v}.
\eea 
We obtain
\bea 
h'(\lambda_k)
  &=&  -1 - \bm{v}^T (\bm{B}_{k}^T \bm\Omega \bm{B}_{k} - \lambda_k \bm{I}_{p-1})^{-2} \bm{v}.
\eea 
When $|\bm{v}^T (\bm{B}_{k}^T \bm\Omega\bm{B}_{k} - \lambda_k \bm{I}_{p-1})^{-2} \bm{v}|<1$,  $|h'(\lambda_k)|\neq 0$, which implies $h(z)$ has a simple zero at $\lambda_k$ and 
\bea 
\bm{u}_{k}^T\bm{u}_{0,k}\bm{u}_{0,k}^T\bm{u}_{k}  &=&
\frac{1}{1 +  || (\bm{B}_{k}^T \bm\Omega \bm{B}_{k} - \lambda_k \bm{I}_{p-1})^{-1} \bm{v}||_2^2}\\
&=&\frac{1}{1 +  || \sqrt{\lambda_{0,k}}(\bm\Lambda_{0,-k}^{1/2}\bm{U}_{0,-k}^T \bm\Omega \bm{U}_{0,-k}\bm\Lambda_{0,-k}^{1/2} - \lambda_k \bm{I}_{p-1})^{-1} \bm\Lambda_{0,-k}^{1/2}\bm{U}_{0,-k}^T \bm\Omega \bm{u}_{0,k}||_2^2}\\
&=&\frac{1}{1 +  || \sqrt{\lambda_{0,k}}\bm\Lambda_{0,-k}^{-1/2}(\bm{U}_{0,-k}^T \bm\Omega \bm{U}_{0,-k} - \lambda_k \bm\Lambda_{0,-k}^{-1})^{-1} \bm{U}_{0,-k}^T \bm\Omega \bm{u}_{0,k}||_2^2}
\eea 
by Theorem 8.13 of \cite{ponnusamy2005foundations}.

\end{proof}

\begin{lemma}\label{lemma:blocknorm}
Let  
    \bea 
    \bm{A} = \begin{pmatrix}
        \bm{A}_{11} & \bm{A}_{12}\\
        \bm{A}_{21} & \bm{A}_{22}
    \end{pmatrix} \in \bbR^{p\times p},
    \eea 
    where $\bm{A}_{11}\in \bbR_{p_1\times p_1 }$ and $\bm{A}_{22}\in \bbR_{p_2\times p_2}$. 
    Then, 
    \bea 
    ||\bm{A}_{12}||_2 &\le&  ||\bm{A}||_2 .
    \eea 
\end{lemma}

\begin{proof}[Proof]
    Let $\bm{v}=(\bm{v}_1^T,\bm{v}_2^T)^T \in \bbR^p$ with $||\bm{v}||_2 =1 $ and $\bm{v}_1  = \bm{0}_{p_1}$.
\bea 
||\bm{A}||_2 &\ge&
||\bm{A}\bm{v}||_2 \\
&\ge & 
||    \bm{A}_{11}\bm{v}_1   + \bm{A}_{12}\bm{v}_2   ||_2
\\
&=& ||\bm{A}_{12}\bm{v}_2  ||_2.
\eea 
Since $\bm{v}_2\in \bbR^{p_2} $ can be an arbitrary vector with $||\bm{v}_2||_2=1$,
\bea 
||\bm{A}||_2  \ge ||\bm{A}_{12}||_2.
\eea 

\end{proof}

\begin{lemma}\label{lemma:blocknorm2}
    Let $\bm\Omega\in \calC_p$, and let $\bm\Gamma = [\bm\Gamma_1 ,\bm\Gamma_2]\in \bbR^{p\times q}$ be an orthogonal matrix with $\bm\Gamma_1\in \bbR^{p\times q_1}$ and $\bm\Gamma_2\in \bbR^{p\times q_2}$.
    \bea 
    ||\bm\Gamma_1^T \bm\Omega \bm\Gamma_2 ||_2  \le ||\bm\Gamma^T (\bm\Omega -\bm{I}_p) \bm\Gamma ||_2 .
    \eea 
\end{lemma}

\begin{proof}[Proof]

    \bea
   || \bm\Gamma^T (\bm\Omega -\bm{I}_p) \bm\Gamma ||_2 &=& 
    \begin{pmatrix}
        \bm\Gamma_1^T (\bm\Omega -\bm{I}_p)\bm\Gamma_1  & \bm\Gamma_1^T (\bm\Omega -\bm{I}_p)\bm\Gamma_2 \\
        \bm\Gamma_2^T (\bm\Omega -\bm{I}_p)\bm\Gamma_1  & \bm\Gamma_2^T (\bm\Omega -\bm{I}_p)\bm\Gamma_2 
    \end{pmatrix} \\
    &\ge &  ||\bm\Gamma_1^T (\bm\Omega -\bm{I}_p)\bm\Gamma_2 ||_2\\
    &= &||\bm\Gamma_1^T \bm\Omega  \bm\Gamma_2||_2,
    \eea 
    where the inequality is satisfied by Lemma \ref{lemma:blocknorm}.
\end{proof}

\begin{lemma}\label{lemma:eigenratios}
 
 Suppose  $\max\Big\{ \frac{\lambda_{0,k}}{\lambda_{0,l}}, \frac{\lambda_{0,l}}{\lambda_{0,k}} \Big\} >c > 1$. 
 If $ \left|\frac{\lambda_k}{\lambda_{0,k}} -1 \right| \le \delta:=\delta(c) =  (1-1/c)/4 \wedge  (1-1/c)/(2(2-1/c)) \wedge 1/2$,  then  
\bea 
|\frac{\lambda_{k}}{\lambda_{0,l}} -1 |  &\ge& (1-1/c)/2,\\
 \frac{\sqrt{\lambda_k\lambda_{0,l}}}{|\lambda_k-\lambda_{0,l} |}   &\le& 
 \frac{2\sqrt{2}}{1-1/c} \min\Big\{ \sqrt{\frac{\lambda_{0,l}}{\lambda_{0,k}} }, \sqrt{\frac{\lambda_{0,k}}{\lambda_{0,l}} } \Big\}.
\eea 
\end{lemma} 

\begin{proof}[Proof]

    First, suppose $\lambda_{0,k}/\lambda_{0,l} >c$.


    We have
    \bea 
|\frac{\lambda_k}{\lambda_{0,l}} -1| &=& \frac{\lambda_{0,k}}{\lambda_{0,l}}| \frac{\lambda_k}{\lambda_{0,k}}-\frac{\lambda_{0,l}}{\lambda_{0,k}} |\\
&\ge& 
\frac{\lambda_{0,k}}{\lambda_{0,l}} \Big(| \frac{\lambda_{0,l}}{\lambda_{0,k}} -1| -
|\frac{\lambda_k}{\lambda_{0,k}}-1|
\Big)\\
&\ge& \frac{\lambda_{0,k}}{\lambda_{0,l}}(1-1/c- \delta) \\
&\ge& \frac{\lambda_{0,k}}{\lambda_{0,l}}(1-1/c)/2,
\eea 
where the last inequality is satisfied when $\delta\le (1-1/c)/2$.

Thus, since $\lambda_{0,k}/\lambda_{0,l} >c$ ,
\bea 
\left|\frac{\lambda_k}{\lambda_{0,l}} -1 \right| &\ge& c(1-1/c)/2\\
&\ge & (1-1/c)/2.
\eea 

We have
\bea 
\frac{\lambda_k}{\lambda_{0,l}} &=& \frac{\lambda_{0,k}}{\lambda_{0,l}} \frac{\lambda_k}{\lambda_{0,k}}  \\
&\le&  \frac{\lambda_{0,k}}{\lambda_{0,l}}  ( 1+ \delta)\\
&\le&  2\frac{\lambda_{0,k}}{\lambda_{0,l}}  .
\eea 
Then,
\bea 
\frac{\sqrt{\lambda_k\lambda_{0,l}}}{|\lambda_k-\lambda_{0,l}|}&=&\frac{\sqrt{\lambda_k/\lambda_{0,l}}}{|\lambda_k/\lambda_{0,l} -1|}\\
&\le&  \sqrt{\frac{\lambda_{0,l}}{\lambda_{0,k}} }   \frac{2\sqrt{2}}{1-1/c}.
\eea

Next, suppose $\lambda_{0,l}/\lambda_{0,k} >c$. 
\bea 
\left|\frac{\lambda_{k}}{\lambda_{0,l}} -1 \right| &=&
\left|\frac{\lambda_k}{\lambda_{0,k}} \Big( \frac{\lambda_{0,k}}{\lambda_{0,l}}-1 \Big) + \Big( \frac{\lambda_k}{\lambda_{0,k}}-1\Big) \right|\\
&\ge& \frac{\lambda_k}{\lambda_{0,k}}  \left|\Big( \frac{\lambda_{0,k}}{\lambda_{0,l}}-1 \Big) \right| - \left|\Big( \frac{\lambda_k}{\lambda_{0,k}}-1\Big) \right| \\
&\ge&  \frac{\lambda_k}{\lambda_{0,k}}(1-1/c) - \delta\\
&\ge&  (1-1/c)\left(1-  \left|\frac{\lambda_k}{\lambda_{0,k}}-1 \right| \right) - \delta\\
&\ge&  (1-1/c)(1-  \delta) - \delta\\
&\ge& (1-1/c) - \delta ( 2-1/c) \\
&\ge& (1-1/c)/2,
\eea 
where the last inequality is satisfied since $\delta \le (1-1/c)/(2(2-1/c))$.
We have
\bea 
\left|\frac{\lambda_{0,k}}{\lambda_k} -1\right| &=& \frac{\lambda_{0,k}}{\lambda_k}\left|\frac{\lambda_{k}}{\lambda_{0,k}} -1 \right|\\
&\le& \left|\frac{\lambda_{k}}{\lambda_{0,k}} -1 \right| + \left|\frac{\lambda_{0,k}}{\lambda_k}-1\right| \left|\frac{\lambda_{k}}{\lambda_{0,k}} -1 \right|\\
&\le& \delta + \left|\frac{\lambda_{0,k}}{\lambda_k} -1 \right| \delta,
\eea 
which gives
\bea 
\left|\frac{\lambda_{0,k}}{\lambda_k} -1 \right|&\le& \frac{\delta}{1-\delta} \\
&\le& 2\delta
\eea 
by assuming $\delta<1/2$.

We have 
\bea 
\left|\frac{\lambda_{0,l}}{\lambda_{k}} -1 \right| &=& \frac{\lambda_{0,l}}{\lambda_{0,k}}  \left|\frac{\lambda_{0,k}}{\lambda_{k}} -  \frac{\lambda_{0,k}}{\lambda_{0,l}} \right|\\
&\ge& \frac{\lambda_{0,l}}{\lambda_{0,k}}  \left( 1 -  \frac{\lambda_{0,k}}{\lambda_{0,l}}  -\left|\frac{\lambda_{0,k}}{\lambda_{k}}-1 \right|\right)\\
&\ge&   \frac{\lambda_{0,l}}{\lambda_{0,k}} (1-1/c - 2\delta) \\
&=& \frac{\lambda_{0,l}}{\lambda_{0,k}} (1-1/c)/2  ,
\eea 
where the last inequality is satisfied by setting $\delta\le (1-1/c)/4$, and
\bea 
\frac{\lambda_{0,l}}{\lambda_{k}} &=& \frac{\lambda_{0,l}}{\lambda_{0,k}} \frac{\lambda_{0,k}}{\lambda_{k}}  \\
&\le&  \frac{\lambda_{0,l}}{\lambda_{0,k}}  ( 1+ 2\delta)\\
&\le& 2\frac{\lambda_{0,l}}{\lambda_{0,k}}.
\eea 
Thus, we obtain 
\bea 
\frac{\sqrt{\lambda_k\lambda_{0,l}}}{|\lambda_k-\lambda_{0,l}|}&=&\frac{\sqrt{\lambda_{0,l}/\lambda_k}}{|\lambda_{0,l}/\lambda_k -1|}\\
&\le&  \sqrt{\frac{\lambda_{0,k}}{\lambda_{0,l}} }   \frac{2\sqrt{2}}{1-1/c}.
\eea 

\end{proof}

\begin{lemma}\label{lemma:A-B}
    Let $\bm{A},\bm{B}\in \calC_p$. If $\lambda_{\min}(\bm{A}) > ||\bm{B}||_2$, then  
    \bea 
    ||(\bm{A}-\bm{B})^{-1}||_2 \le \frac{1}{\lambda_{\min}(\bm{A})-||\bm{B}||_2}.
    \eea 
\end{lemma}

\begin{proof}[Proof]
We have 
    \bea 
    ||(\bm{A}-\bm{B})^{-1}||_2 &=& \frac{1}{\min_{j=1,\ldots,p}|\lambda_j(\bm{A}-\bm{B})|} ,\\
    \lambda_j(\bm{A}-\bm{B})  &=& \bm{v}_j^T \bm{A}\bm{v}_j -\bm{v}_j^T \bm{B}\bm{v}_j \\
    &\ge& \lambda_{\min}(\bm{A}) - ||\bm{B}||_2 ,
    \eea 
    where $\bm{v}_j$ is the $j$th eigenvector.
    Since $\lambda_{\min}(\bm{A}) > ||\bm{B}||_2$,
    \bea 
    ||(\bm{A}-\bm{B})^{-1}||_2 \le \frac{1}{\lambda_{\min}(\bm{A})-||\bm{B}||_2}.
    \eea 
\end{proof}


Using these lemmas, we prove Theorem \ref{thm:highevector}.
\begin{proof}[Proof of Theorem \ref{thm:highevector}]
 
We show $1-(\bm{u}_l^T \bm{u}_{0,l})^2 \le Ct$ when
$\left|\frac{\lambda_l}{\lambda_{0,l}}-1 \right| \le \delta_1$,
$ || \bm\Gamma^T \bm\Omega \bm\Gamma  -\bm{I}_K ||_2\le \frac{\sqrt{t}}{\sqrt{B_k}} \wedge \delta_2$ and $\frac{\sqrt{\lambda_{0,K+1}}||\bm\Gamma_\perp^T \bm\Omega \bm\Gamma_\perp ||^{1/2}}{\sqrt{\lambda_{0,l}} } \le \sqrt{t}$ $\left(\text{or}\,\, \frac{\sqrt{\lambda_{0,K+1}}}{\sqrt{\lambda_{0,l}} }  || \bm\Omega - \bm{I}_p||_2 \le \sqrt{t} \right)$, where $\bm\Omega = \bm\Sigma_0^{-1/2} \bm\Sigma  \bm\Sigma_0^{-1/2} $, $t \le \Big(\frac{\lambda_{0,K+1}}{\lambda_{0,l}}\Big)^{1/4}(1-\delta_1) \wedge (d^{-1/8}-d^{1/2})^2\wedge \left( \frac{\delta_2}{2c_l(1+\delta_2)} \right)$, and $C$ is some positive constant dependent on $c$ and $d$. 

    
We obtain $\lambda_{0,l} \le \lambda_{l}/(1-\delta_1) \le 2 \lambda_{l}$ from $\left|\frac{\lambda_l}{\lambda_{0,l}}-1 \right| \le \delta_1$ by setting $\delta_1<1/2$.
    Then, by Lemma \ref{lemma:evector}, we have
    \bea
    1-(\bm{u}_l^T \bm{u}_{0,l})^2 &\le&  ||\sqrt{\lambda_{0,l}}\bm\Lambda_{0,-l}^{-1/2}(\bm{U}_{0,-l}^T \bm\Omega \bm{U}_{0,-l} - \lambda_l \bm\Lambda_{0,-l}^{-1})^{-1} \bm{U}_{0,-l}^T \bm\Omega \bm{u}_{0,l}||_2^2 \\
    &\le& 2||\sqrt{\lambda_{l}}\bm\Lambda_{0,-l}^{-1/2}(\bm{U}_{0,-l}^T \bm\Omega \bm{U}_{0,-l} - \lambda_k \bm\Lambda_{0,-l}^{-1})^{-1} \bm{U}_{0,-l}^T \bm\Omega \bm{u}_{0,l}||_2^2,
    \eea 
where $\bm\Lambda_{0,-l} = \textup{diag} (\bm\Lambda_{(1)}, \bm\Lambda_{(2)})$ and $\bm{U}_{0,-l} = [\bm{U}_{(1)}, \bm{U}_{(2)}]$. 

In this expression, $\bm\Lambda_{(1)} = \textup{diag}(\lambda_{0,1}, \ldots, \lambda_{0,l-1}, \lambda_{0,l+1}, \ldots, \lambda_{0,K}) \in \calC_{K-1}$ and $\bm\Lambda_{(2)} = \textup{diag}(\lambda_{0,K+1}, \ldots, \lambda_{0,p}) \in \calC_{p-K}$ are diagonal matrices. 
Similarly, $\bm{U}_{(1)} = [\bm{u}_{0,1}, \ldots, \bm{u}_{0,l-1}, \bm{u}_{0,l+1}, \ldots, \bm{u}_{0,K}] \in \bbR^{p \times (K-1)}$ and $\bm{U}_{(2)} = \bm\Gamma_\perp = [\bm{u}_{0,K+1}, \ldots, \bm{u}_{0,p}] \in \bbR^{p \times (p-K)}$ are orthogonal matrices.

We expand $( \bm{U}_{0,-l}^T \bm\Omega \bm{U}_{0,-l} -\lambda_l \bm\Lambda_{0,-l}^{-1})^{-1}$ as 
    \bea 
    &&(\bm{U}_{0,-l}^T \bm\Omega \bm{U}_{0,-l} -\lambda_l \bm\Lambda_{0,-l}^{-1})^{-1} \\
    &=& \begin{pmatrix}
\bm{U}_{(1)}^T \bm\Omega \bm{U}_{(1)}  - \lambda_l \bm\Lambda_{(1)}^{-1} & \bm{U}_{(1)}^T \bm\Omega  \bm{U}_{(2)} \\
        (\bm{U}_{(1)}^T \bm\Omega  \bm{U}_{(2)})^T & \bm{U}_{(2)}^T \bm\Omega  \bm{U}_{(2)} -  \lambda_l\bm\Lambda_{(2)}^{-1}
    \end{pmatrix}^{-1}\\
    &=& \begin{pmatrix}
    \bm{A} & \bm{B} \\
    \bm{B}^T & \bm{C}
    \end{pmatrix}^{-1} \\
    &=& \begin{pmatrix}
    \bm{B}_{11} & \bm{B}_{12} \\
    \bm{B}_{12}^T & \bm{B}_{22}
    \end{pmatrix},
    \eea 
    where $\bm{B}_{11} = (\bm{A} -\bm{B} \bm{C}^{-1} \bm{B}^T)^{-1}$, $\bm{B}_{12} = -(\bm{A} -\bm{B} \bm{C}^{-1} \bm{B}^T)^{-1} \bm{B} \bm{C}^{-1} = -\bm{B}_{11}\bm{B} \bm{C}^{-1}$ and $\bm{B}_{22} = \bm{C}^{-1} + \bm{C}^{-1}\bm{B}^T (\bm{A} -\bm{B} \bm{C}^{-1} \bm{B}^T)^{-1} \bm{B} \bm{C}^{-1}$.
    Then, 
            \bean
     &&||\sqrt{\lambda_{l}}\bm\Lambda_{0,-l}^{-1/2}( \bm{U}_{-l}^T \bm\Omega \bm{U}_{-l} -\lambda_l  \bm\Lambda_{0,-l}^{-1})^{-1}
 \bm{U}_{-l}^T \bm\Omega  \bm{u}_{0,l}||_2 \nonumber\\
&=& 
||\begin{pmatrix}
     \sqrt{\lambda_{l}}\bm\Lambda_{(1)}^{-1/2}\bm{B}_{11}\bm{U}_{(1)}^T \bm\Omega \bm{u}_{0,l}  +  \sqrt{\lambda_{l}}\bm\Lambda_{(1)}^{-1/2}\bm{B}_{12} \bm{U}_{(2)}^T \bm\Omega \bm{u}_{0,l}  \\
     \sqrt{\lambda_{l}}\bm\Lambda_{(2)}^{-1/2}\bm{B}_{12}^T \bm{U}_{(1)}^T \bm\Omega \bm{u}_{0,l}  +  \sqrt{\lambda_{l}}\bm\Lambda_{(2)}^{-1/2}\bm{B}_{22}\bm{U}_{(2)}^T \bm\Omega \bm{u}_{0,l} 
    \end{pmatrix}||_2 \nonumber\\
&\le& ||\sqrt{\lambda_{l}}\bm\Lambda_{(1)}^{-1/2}\bm{B}_{11}||_2 ||\bm{U}_{(1)}^T \bm\Omega \bm{u}_{0,l}||_2 + 
||\sqrt{\lambda_{l}}\bm\Lambda_{(1)}^{-1/2}\bm{B}_{12} ||_2 ||\bm{U}_{(2)}^T \bm\Omega \bm{u}_{0,l}||_2 \nonumber\\
&&+ ||\sqrt{\lambda_{l}}\bm\Lambda_{(2)}^{-1/2}\bm{B}_{12}^T||_2 || \bm{U}_{(1)}^T \bm\Omega \bm{u}_{0,l} ||_2 +  ||\sqrt{\lambda_{l}}\bm\Lambda_{(2)}^{-1/2}\bm{B}_{22}||_2 ||\bm{U}_{(2)}^T \bm\Omega \bm{u}_{0,l} ||_2.\label{eq:factoruppers}
    \eean

    In the following, we provide upper bounds of the factors in \eqref{eq:factoruppers}.
    Since $\bm\Gamma$ consists of the column vectors of $\bm{U}_{(1)}$ and $\bm{u}_{0,l}$, by Lemma \ref{lemma:blocknorm2},
    \bean 
    ||\bm{U}_{(1)}^T \bm\Omega \bm{U}_{(1)} - \bm{I}_{K-1}||_2 &\le& ||\bm\Gamma^T \bm\Omega \bm\Gamma -\bm{I}_K ||_2\le \sqrt{t}/\sqrt{B_k} \wedge \delta_2\label{eq:U1OU1}\\
    ||\bm{U}_{(1)}^T \bm\Omega \bm{u}_{0,l} ||_2  &\le& ||\bm\Gamma^T \bm\Omega \bm\Gamma -\bm{I}_K ||_2\le \sqrt{t}/\sqrt{B_k} \wedge \delta_2.\nonumber
    \eean 
    Under condition $\frac{\sqrt{\lambda_{0,K+1}}||\bm\Gamma_\perp^T \bm\Omega \bm\Gamma_\perp ||_2^{1/2}}{\sqrt{\lambda_{0,l}} } \le \sqrt{t}$,
    we have
    \bea 
    \frac{\lambda_{0,K+1}}{\lambda_{l}}||\bm{U}_{(2)}^T \bm\Omega \bm{U}_{(2)}||_2 &\le&   
       \frac{1}{1-\delta_1}\frac{\lambda_{0,K+1}}{\lambda_{0,l}}||\bm{U}_{(2)}^T \bm\Omega \bm{U}_{(2)}||_2 \\
       &=&
    \frac{1}{1-\delta_1}\frac{\lambda_{0,K+1}}{\lambda_{0,l}}||\bm\Gamma_{\perp}^T \bm\Omega \bm\Gamma_\perp ||_2 ,\\
    &\le&   \frac{t}{1-\delta_1}\\
    &\le&  \frac{\delta}{1-\delta_1},
    \eea 
    where the last inequality is satisfied since $t\le \delta$, where $\delta = \Big(\frac{\lambda_{0,K+1}}{\lambda_{0,l}}\Big)^{1/4}(1-\delta_1) \wedge (d^{-1/8}-d^{1/2})^2\wedge \left( \frac{\delta_2}{2c_l(1+\delta_2)} \right)$.
    Under condition $\frac{\sqrt{\lambda_{0,K+1}}}{\sqrt{\lambda_{0,l}} }  || \bm\Omega - \bm{I}_p||_2 \le \sqrt{t}$,
    \bea 
\frac{\lambda_{0,K+1}}{\lambda_{l}}||\bm{U}_{(2)}^T \bm\Omega \bm{U}_{(2)}||_2 &\le&
\frac{1}{1-\delta_1}\frac{\lambda_{0,K+1}}{\lambda_{0,l}}||\bm{U}_{(2)}^T \bm\Omega \bm{U}_{(2)}||_2 \\
&\le&
  \frac{1}{1-\delta_1}\frac{\lambda_{0,K+1}}{\lambda_{0,l}}(1+||\bm\Omega -\bm{I}_p||_2) \\
  &\le& \frac{1}{1-\delta_1}\left(\frac{\lambda_{0,K+1}}{\lambda_{0,l}}+\sqrt{\frac{\lambda_{0,K+1}}{\lambda_{0,l}}} \sqrt{t} \right)\\
  &\le& \Big(\frac{\lambda_{0,K+1}}{\lambda_{0,l}}\Big)^{1/4},
    \eea 
    where the last inequality is satisfied by the following.
    Let $x = \frac{\lambda_{0,K+1}}{\lambda_{0,l}}$. We have $x \le  d <1 $ and set $\delta_1$ and $\delta$ to satisfy
 $\delta_1\le  1-d^{1/8} \le 1-x^{1/8}$ and $\sqrt{t} \le \sqrt{\delta} \le  d^{-1/8}-d^{1/2} \le x^{-1/8} - x^{1/2} $. Then, 
    \bea 
    \frac{1}{1-\delta_1}\left(\frac{\lambda_{0,K+1}}{\lambda_{0,l}}+\sqrt{\frac{\lambda_{0,K+1}}{\lambda_{0,l}}} \sqrt{t} \right) &=&\sqrt{x} \frac{\sqrt{x}+\sqrt{t}}{1-\delta_1} \\
    &\le&
    \sqrt{x} \frac{x^{-1/8} }{x^{1/8}} \\
    &=& x^{1/4}\\
    &=&\Big(\frac{\lambda_{0,K+1}}{\lambda_{0,l}}\Big)^{1/4}.
    \eea 
Thus, when $\frac{\sqrt{\lambda_{0,K+1}}||\bm\Gamma_\perp^T \bm\Omega \bm\Gamma_\perp ||_2^{1/2}}{\sqrt{\lambda_{0,l}} } \le \sqrt{t}$ or $\frac{\sqrt{\lambda_{0,K+1}}}{\sqrt{\lambda_{0,l}} }  || \bm\Omega - \bm{I}_p||_2 \le \sqrt{t}$, we obtain
\bean\label{eq:U2OU2} 
\frac{\lambda_{0,K+1}}{\lambda_{l}}||\bm{U}_{(2)}^T \bm\Omega \bm{U}_{(2)}||_2 &\le& \Big(\frac{\lambda_{0,K+1}}{\lambda_{0,l}}\Big)^{1/4} \vee  \frac{\delta}{1-\delta_1} \nonumber\\
&\le& \Big(\frac{\lambda_{0,K+1}}{\lambda_{0,l}}\Big)^{1/4} \vee  \frac{1}{2},
\eean 
by setting $\delta$ to satisfy $\frac{\delta}{1-\delta_1} \le \frac{1}{2}$.

    Under condition $\frac{\sqrt{\lambda_{0,K+1}}||\bm\Gamma_\perp^T \bm\Omega \bm\Gamma_\perp ||_2^{1/2}}{\sqrt{\lambda_{0,l}} } \le \sqrt{t}$,
    we have
    \bea 
\sqrt{\frac{\lambda_{0,K+1}}{\lambda_{l}}}||\bm{U}_{(1)}^T  \bm\Omega \bm{U}_{(2)}||_2 
&\le&  
\sqrt{\frac{2\lambda_{0,K+1}}{\lambda_{0,l}}}||\bm{U}_{(1)}^T  \bm\Omega^{1/2}||_2||  \bm\Omega^{1/2}\bm{U}_{(2)}||_2 
\\
&=&
||\bm{U}_{(1)}^T  \bm\Omega \bm{U}_{(1)} ||_2^{1/2}
\left(\frac{2\lambda_{0,K+1}}{\lambda_{0,l}}||\bm{U}_{(2)}^T  \bm\Omega \bm{U}_{(2)} ||_2\right)^{1/2} \\
&\le& ||\bm\Gamma^T \bm\Omega \bm\Gamma ||_2^{1/2}
 \left(\frac{2\lambda_{0,K+1}}{\lambda_{0,l}}||\bm\Gamma_\perp^T \bm\Omega \bm\Gamma_\perp ||_2\right)^{1/2}\\
 &\le& ( 1 + ||\bm{I}_K - \bm\Gamma^T \bm\Omega \bm\Gamma||_2)^{1/2}\sqrt{2t}\\
 &\le&  \sqrt{2 (1+\delta_2)t} ,
    \eea 
    and under condition $\frac{\sqrt{\lambda_{0,K+1}}}{\sqrt{\lambda_{0,l}} }  || \bm\Omega - \bm{I}_p||_2 \le \sqrt{t}$,
    \bea 
    \sqrt{\frac{\lambda_{0,K+1}}{\lambda_{l}}}||\bm{U}_{(1)}^T  \bm\Omega \bm{U}_{(2)}||_2 &\le& \sqrt{\frac{\lambda_{0,K+1}}{\lambda_{l}}}|| \bm\Omega -\bm{I}_p||_2 \\
    &\le& \sqrt{\frac{2\lambda_{0,K+1}}{\lambda_{0,l}}}|| \bm\Omega -\bm{I}_p||_2 \\
    &\le& \sqrt{2t}\\
    &\le& \sqrt{2 (1+\delta_2)t},
    \eea
    where the first inequality is satisfied by Lemma \ref{lemma:blocknorm2}.
Thus, when $\frac{\sqrt{\lambda_{0,K+1}}||\bm\Gamma_\perp^T \bm\Omega \bm\Gamma_\perp ||_2^{1/2}}{\sqrt{\lambda_{0,l}} } \le \sqrt{t}$ or $\frac{\sqrt{\lambda_{0,K+1}}}{\sqrt{\lambda_{0,l}} }  || \bm\Omega - \bm{I}_p||_2 \le \sqrt{t}$,
\bean\label{eq:U1OU2} 
\sqrt{\frac{\lambda_{0,K+1}}{\lambda_{l}}}||\bm{U}_{(1)}^T  \bm\Omega \bm{U}_{(2)}||_2 \le \sqrt{2 (1+\delta_2)t}.
\eean

Likewise, we obtain 
\bea 
\sqrt{\frac{\lambda_{0,K+1}}{\lambda_{l}}}||\bm{U}_{(2)}^T \bm\Omega \bm{u}_{0,l} ||_2 \le \sqrt{2 (1+\delta_2)t}.
\eea

    We have 
\bea 
 \bm{B}_{11}  
&=& (\bm{U}_{(1)}^T \bm\Omega \bm{U}_{(1)}  - \bm{I}_{K-1}
 +\bm{I}_{K-1} -\lambda_l \bm\Lambda_{(1)}^{-1} - \bm{B}\bm{C}^{-1}\bm{B}^T)^{-1}\\
 &=& \bm{D}_1   (\bm{I}_{K-1}  +( \bm{U}_{(1)}^T \bm\Omega \bm{U}_{(1)}  - \bm{I}_{K-1} -\bm{B}\bm{C}^{-1}\bm{B}^T )\bm{D}_1)^{-1},
\eea 
where $\bm{D}_1= (\bm{I}_{K-1} -\lambda_l \bm\Lambda_{(1)}^{-1})^{-1} = \textup{diag}\left(\frac{\lambda_{0,1}}{\lambda_{0,1}-\lambda_l},\ldots, \frac{\lambda_{0,l-1}}{\lambda_{0,l-1}-\lambda_l},\frac{\lambda_{0,l+1}}{\lambda_{0,l+1}-\lambda_l},\ldots, \frac{\lambda_{0,K}}{\lambda_{0,K}-\lambda_l} \right)$
and 
\bean\label{eq:D1upper}
||\bm{D}_1||_2 \le 2/(1-1/c) = 1/(4\delta_2)
\eean 
by Lemma \ref{lemma:eigenratios} 
where $\delta_2$ is set to satisfy $2/(1-1/c) = 1/(4\delta_2)$.

Let $c_l = \frac{1}{1- (\lambda_{0,K+1}/\lambda_{0,l})^{1/4}\vee \frac{1}{2}}$.
We have 
\bea 
\lambda_l||\bm{C}^{-1} ||_2/\lambda_{0,K+1} 
    &\le& \frac{\lambda_l}{\lambda_{0,K+1}\{\lambda_l/||\bm\Lambda_{(2)}||_2 - ||\bm{U}_{(2)}^T \bm\Omega  \bm{U}_{(2)} ||_2 \}} \\
    &=& \frac{1}{1- \lambda_{0,K+1}||\bm{U}_{(2)}^T \bm\Omega  \bm{U}_{(2)} ||_2/\lambda_l} \\
    &\le& \frac{1}{1-  (\lambda_{0,K+1}/\lambda_{0,l})^{1/4}\vee 1/2}\\
    &\le& c_l,
    \eea 
    where the first inequality is satisfied by Lemma \ref{lemma:A-B}, and the second inequality is satisfied by \eqref{eq:U2OU2}.
    To apply Lemma \ref{lemma:A-B}, we show $\lambda_l/||\bm\Lambda_{(2)}||_2 > ||\bm{U}_{(2)}^T \bm\Omega  \bm{U}_{(2)} ||_2$.
    Since $\lambda_{0,K+1} = ||\bm\Lambda_{(2)}||_2$, it suffices to show $1> \lambda_{0,K+1}||\bm{U}_{(2)}^T \bm\Omega  \bm{U}_{(2)} ||_2/\lambda_l$, which is shown by \eqref{eq:U2OU2}. 
We also have
\bean\label{eq:UomeagaU} 
|| \bm{U}_{(1)}^T \bm\Omega \bm{U}_{(1)}  - \bm{I}_{K-1} -\bm{B}\bm{C}^{-1}\bm{B}^T || &\le&
|| \bm{U}_{(1)}^T \bm\Omega \bm{U}_{(1)} - \bm{I}_{K-1}|| + 
||\bm{B}||_2^2 ||\bm{C}^{-1}||_2\nonumber\\
& \le&|| \bm{U}_{(1)}^T \bm\Omega \bm{U}_{(1)} - \bm{I}_{K-1}|| + 
c_l|| \bm{U}_{(1)}^T \bm\Omega  \bm{U}_{(2)}  ||^2\lambda_{0,K+1}/\lambda_l  \nonumber\\
&\le&\delta_2+  2c_l (1+\delta_2)t \nonumber\\
&\le& 2\delta_2,
\eean 
where the third inequality is satisfied by \eqref{eq:U1OU1} and \eqref{eq:U1OU2}, and the last inequality is satisfied by setting $\delta$ to satisfy $t \le \delta\le   \Big(\frac{\delta_2}{2c_l(1+\delta_2)}\Big)$.
Thus,
\bea
 &&||     \sqrt{\lambda_{l}}\bm\Lambda_{(1)}^{-1/2} \bm{B}_{11}  ||_2 \\
  &\le &||\sqrt{\lambda_{l}}\bm\Lambda_{(1)}^{-1/2}  \bm{D}_1 ||_2  
  ||(\bm{I}_{K-1}  +( \bm{U}_{(1)}^T \bm\Omega \bm{U}_{(1)}  - \bm{I}_{K-1} -\bm{B}\bm{C}^{-1}\bm{B}^T )\bm{D}_1)^{-1} ||_2 \\
 &\le& \frac{ ||\sqrt{\lambda_{l}}\bm\Lambda_{(1)}^{-1/2}  \bm{D}_1 ||_2   }{1 - 
 ||\bm{D}_1||_2||\bm{U}_{(1)}^T \bm\Omega \bm{U}_{(1)}  - \bm{I}_{K-1} -\bm{B}\bm{C}^{-1}\bm{B}^T ||_2
 }\\
  &\le& 2||\sqrt{\lambda_{l}}\bm\Lambda_{(1)}^{-1/2}  \bm{D}_1 ||_2 ,
 \eea 
 where the last inequality is satisfied by \eqref{eq:D1upper} and \eqref{eq:UomeagaU}.


Next, we have
\bea 
|| \sqrt{\lambda_{l}}\bm\Lambda_{(1)}^{-1/2}\bm{B}_{12}  ||_2 &\le& ||  \sqrt{\lambda_{l}}\bm\Lambda_{(1)}^{-1/2}\bm{B}_{11}||_2  ||\bm{U}_{(1)}^T \bm\Omega \bm{U}_{(2)}||_2 ||\bm{C}^{-1}||_2 \\
&=&||  \sqrt{\lambda_{l}}\bm\Lambda_{(1)}^{-1/2}\bm{B}_{11}||_2  \lambda_{0,K+1}||\bm{U}_{(1)}^T \bm\Omega \bm{U}_{(2)}||_2/ \lambda_l  (\lambda_l||\bm{C}^{-1}||_2 /\lambda_{0,K+1}) \\
&\le& 2||  \sqrt{\lambda_{l}}\bm\Lambda_{(1)}^{-1/2}\bm{D}_{1}||_2  
\frac{\sqrt{\lambda_{0,K+1}}}{\sqrt{\lambda_l}} \sqrt{2 (1+\delta_2)t} c_l\\
&=&  2c_l \sqrt{2 (1+\delta_2)t} ||  \sqrt{\lambda_{l}}\bm\Lambda_{(1)}^{-1/2}\bm{D}_{1}||_2 \frac{\sqrt{\lambda_{0,K+1}}}{\sqrt{\lambda_l}} ,
\eea 
and
\bea 
|| \sqrt{\lambda_{l}}\bm\Lambda_{(2)}^{-1/2}\bm{C}^{-1} ||_2 &=& 
\left\|
 \sqrt{\lambda_l}\bm\Lambda_{(2)}^{-1/2} [\bm{U}_{(2)}^T \bm\Omega  \bm{U}_{(2)} - \lambda_l\bm\Lambda_{(2)}^{-1}]^{-1}\right\|_2
\\&=& 
 \left\|[ \bm{U}_{(2)}^T \bm\Omega  \bm{U}_{(2)} (\sqrt{\lambda_l}\bm\Lambda_{(2)}^{-1/2} )^{-1}-  \sqrt{\lambda_l}\bm\Lambda_{(2)}^{-1/2}]^{-1}\right\|_2\\
 &\le & 1/\left[\lambda_{\min} (\sqrt{\lambda_l}\bm\Lambda_{(2)}^{-1/2}) - \left\|\bm{U}_{(2)}^T \bm\Omega  \bm{U}_{(2)} (\sqrt{\lambda_l}\bm\Lambda_{(2)}^{-1/2})^{-1} \right\| \right]\\
 &=& 1/ \left[\sqrt{\lambda_l/ ||\bm\Lambda_{(2)} ||_2 } -  \left \|\bm{U}_{(2)}^T \bm\Omega  \bm{U}_{(2)} (\sqrt{\lambda_l}\bm\Lambda_{(2)}^{-1/2} )^{-1}\right\| \right] \\
 &\le& 1/ \left[\sqrt{\lambda_l/ ||\bm\Lambda_{(2)} ||_2 } -  \lambda_l^{-1/2}\left\|\bm{U}_{(2)}^T \bm\Omega  \bm{U}_{(2)}\right\|_2 \left\|\bm\Lambda_{(2)}\right\|_2^{1/2} \right] \\
 &=& \sqrt{\frac{\lambda_{0,K+1}}{\lambda_l}} \left(\frac{1}{1-||\bm{U}_{(2)}^T \bm\Omega  \bm{U}_{(2)}||_2 \lambda_{0,K+1} /\lambda_l } \right)\\
 &\le& c_l\sqrt{\frac{\lambda_{0,K+1}}{\lambda_l}} ,
\eea 
where the first inequality is satisfied by Lemma \ref{lemma:A-B}.
To apply Lemma \ref{lemma:A-B}, we have to show 
$\lambda_{\min} (\sqrt{\lambda_l}\bm\Lambda_{(2)}^{-1/2}) > ||\bm{U}_{(2)}^T \bm\Omega  \bm{U}_{(2)} (\sqrt{\lambda_l}\bm\Lambda_{(2)}^{-1/2})^{-1} ||_2$.
Since 
$$||\bm{U}_{(2)}^T \bm\Omega  \bm{U}_{(2)} (\sqrt{\lambda_l}\bm\Lambda_{(2)}^{-1/2} )^{-1}||_2 \le \lambda_l^{-1/2}||\bm{U}_{(2)}^T \bm\Omega  \bm{U}_{(2)}||_2 ||\bm\Lambda_{(2)}||_2^{1/2},$$ it suffices to show $\sqrt{\lambda_l/ ||\bm\Lambda_{(2)} ||_2 } > \lambda_l^{-1/2}||\bm{U}_{(2)}^T \bm\Omega  \bm{U}_{(2)}||_2 ||\bm\Lambda_{(2)}||_2^{1/2}$, equivalently $1>||\bm{U}_{(2)}^T \bm\Omega  \bm{U}_{(2)}||_2 \lambda_{0,K+1} /\lambda_l $, which is shown by \eqref{eq:U2OU2}.


We have
\bea
 &&||      \bm{B}_{11}  ||_2 \\
  &\le &||\bm{D}_1 ||_2  
  ||(\bm{I}_{K-1}  +( \bm{U}_{(1)}^T \bm\Omega \bm{U}_{(1)}  - \bm{I}_{K-1} -\bm{B}\bm{C}^{-1}\bm{B}^T )\bm{D}_1)^{-1} ||_2 \\
 &\le& \frac{ ||  \bm{D}_1 ||_2   }{1 - 
 ||\bm{D}_1||_2||\bm{U}_{(1)}^T \bm\Omega \bm{U}_{(1)}  - \bm{I}_{K-1} -\bm{B}\bm{C}^{-1}\bm{B}^T ||_2
 }\\
 &\le&2||  \bm{D}_1 ||_2\\
  &\le& 1/(2\delta_2),
 \eea 
 by \eqref{eq:D1upper} and \eqref{eq:UomeagaU}.
And, we have
\bea  
|| \sqrt{\lambda_{l}}\bm\Lambda_{(2)}^{-1/2}\bm{B}_{12}^T ||_2 &\le& ||\sqrt{\lambda_{l}}\bm\Lambda_{(2)}^{-1/2}\bm{C}^{-1}||_2  ||\bm{U}_{(1)}^T \bm\Omega \bm{U}_{(2)}||_2 ||\bm{B}_{11}||_2 \\
&\le& \frac{c_l}{2\delta_2}||\bm{U}_{(1)}^T \bm\Omega \bm{U}_{(2)}||_2 \sqrt{\lambda_{0,K+1}}/\sqrt{\lambda_l}\\
&\le& c_l\sqrt{2 (1+\delta_2)t}/(2\delta_2) ,
\eea 
\bea 
||\sqrt{\lambda_{l}}\bm\Lambda_{(2)}^{-1/2}\bm{B}_{22} ||_2 &\le& 
||\sqrt{\lambda_{l}}\bm\Lambda_{(2)}^{-1/2}\bm{C}^{-1}||_2  (1 + ||\bm{B}^T \bm{B}_{11} \bm{B}\bm{C}^{-1}||_2)  \\
&\le&  ||\sqrt{\lambda_{l}}\bm\Lambda_{(2)}^{-1/2}\bm{C}^{-1}||_2  \left(1 + ||\bm{B}_{11}||_2 \frac{\lambda_{0,K+1}|| \bm{U}_{(1)}^T \bm\Omega  \bm{U}_{(2)}  ||_2^2 }{\lambda_l} \frac{\lambda_l ||\bm{C}^{-1}||_2}{\lambda_{0,K+1}} \right)  \\
&\le& c_l \sqrt{\lambda_{0,K+1}/\lambda_l }  (1 +  c_l (1+\delta_2)t/\delta_2),
\eea 
\bea 
 \sqrt{\lambda_{l}}\bm\Lambda_{(1)}^{-1/2}  \bm{D}_1 &=&\textup{diag}(\frac{\sqrt{\lambda_l\lambda_{0,1}}}{\lambda_{0,1}-\lambda_l},\ldots, \frac{\sqrt{\lambda_l\lambda_{0,l-1}}}{\lambda_{0,l-1}-\lambda_l},\frac{\sqrt{\lambda_{l}\lambda_{0,l+1}}}{\lambda_{0,l+1}-\lambda_l},\ldots, \frac{\sqrt{\lambda_l\lambda_{0,K}}}{\lambda_{0,K}-\lambda_l} ),
\eea 
and 
\bea 
||\sqrt{\lambda_{l}}\bm\Lambda_{(1)}^{-1/2}  \bm{D}_1||_2 \le C_1\sqrt{B_k},
\eea 
for some positive constant $C_1$ dependent on $c$ by Lemma \ref{lemma:eigenratios}.

Collecting the inequalities, we obtain 
\bea 
&&||\sqrt{\lambda_{l}}\bm\Lambda_{(1)}^{-1/2}\bm{B}_{11}\bm{U}_{(1)}^T \bm\Omega \bm{u}_{0,l}  +  \sqrt{\lambda_{l}}\bm\Lambda_{(1)}^{-1/2}\bm{B}_{12} \bm{U}_{(2)}^T \bm\Omega \bm{u}_{0,l}  ||_2\\
&\le& ||\sqrt{\lambda_{l}}\bm\Lambda_{(1)}^{-1/2}\bm{B}_{11}||_2 ||\bm{U}_{(1)}^T \bm\Omega \bm{u}_{0,l}||_2 +
||\sqrt{\lambda_{l}}\bm\Lambda_{(1)}^{-1/2}\bm{B}_{12} ||_2 ||\bm{U}_{(2)}^T \bm\Omega \bm{u}_{0,l}  ||_2 \\
&\lesssim& ||\sqrt{\lambda_{l}}\bm\Lambda_{(1)}^{-1/2}  \bm{D}_1 ||_2 (\sqrt{t}/\sqrt{B_k} + \sqrt{t}) \\
&\lesssim& \sqrt{t} ,
\eea 
where the last inequality is satisfied since $B_k\le 1$,
and 
\bea 
&&||      \sqrt{\lambda_{l}}\bm\Lambda_{(2)}^{-1/2}\bm{B}_{12}^T \bm{U}_{(1)}^T \bm\Omega \bm{u}_{0,l}  +  \sqrt{\lambda_{l}}\bm\Lambda_{(2)}^{-1/2}\bm{B}_{22}\bm{U}_{(2)}^T \bm\Omega \bm{u}_{0,l}||_2 \\
&\le& ||      \sqrt{\lambda_{l}}\bm\Lambda_{(2)}^{-1/2}\bm{B}_{12}^T ||_2 ||\bm{U}_{(1)}^T \bm\Omega \bm{u}_{0,l}||_2 + ||\sqrt{\lambda_{l}}\bm\Lambda_{(2)}^{-1/2}\bm{B}_{22}||_2 ||\bm{U}_{(2)}^T \bm\Omega \bm{u}_{0,l}||_2\\
&\le&  ||      \sqrt{\lambda_{l}}\bm\Lambda_{(2)}^{-1/2}\bm{B}_{12}^T ||_2 ||\bm{U}_{(1)}^T \bm\Omega \bm{u}_{0,l}||_2  \\
&&+ \sqrt{\frac{\lambda_l}{\lambda_{0,K+1}}}||\sqrt{\lambda_{l}}\bm\Lambda_{(2)}^{-1/2}\bm{B}_{22}||_2 \sqrt{\frac{\lambda_{0,K+1}}{\lambda_l}}||\bm{U}_{(2)}^T \bm\Omega \bm{u}_{0,l}||_2\\
&\lesssim& c_l\sqrt{t}.
\eea
Then, 
we obtain
\bea 
1- (\bm{u}_l^T \bm{u}_{0,l})^2   \le  Ct ,
\eea 
for some positive constant $C$ dependent on $c$ and $d$.
Thus, we obtain
    \bea 
    P( \sup_{l=1,\ldots, k} \{1-(\bm{u}_l^T \bm{u}_{0,l})^2\}  > Ct ) &\le&
    P\left( ||\bm\Gamma_\perp^T \bm\Sigma_0^{-1/2}\bm\Sigma \bm\Sigma_0^{-1/2} \bm\Gamma_\perp ||_2^{1/2}\frac{\sqrt{\lambda_{0,K+1}}}{\sqrt{\lambda_{0,k}} } > \sqrt{t} \right) \\
    &&+   P\left( || \bm\Gamma^T \bm\Sigma_0^{-1/2}\bm\Sigma \bm\Sigma_0^{-1/2}\bm\Gamma -\bm{I}_K||_2 > \frac{\sqrt{t}}{\sqrt{B_k}} \wedge \delta_2 \right) \\
    &&+  P\left(\sup_{l=1,\ldots,k}|\frac{\lambda_l}{\lambda_{0,l}}-1 | > \delta_1 \right).
    \eea 



\end{proof}

\section{Proof of Theorem \ref{corr:IWeigenvector}}\label{sec:pf3.6}
We give the proof of Theorem \ref{corr:IWeigenvector}.

\begin{proof}[Proof of Theorem \ref{corr:IWeigenvector}]
    Suppose $p >n$.
    By Theorem \ref{thm:highevector} with $t = M_n\epsilon_n $, we obtain
    \bea 
    &&\pi\left( \sup_{l=1,\ldots, k} \{1-(\bm{u}_l^T \bm{u}_{0,l})^2\}  > CM_n\epsilon_n \mid \bbX_n \right) \\
    &\le&
    \pi\left( ||\bm\Gamma_\perp^T \bm\Sigma_0^{-1/2}\bm\Sigma \bm\Sigma_0^{-1/2} \bm\Gamma_\perp ||_2^{1/2}\frac{\sqrt{\lambda_{0,K+1}}}{\sqrt{\lambda_{0,k}} } > \sqrt{M_n\epsilon_n}\Bigm| \bbX_n \right) \\
    &&+   \pi\left( || \bm\Gamma^T \bm\Sigma_0^{-1/2}\bm\Sigma \bm\Sigma_0^{-1/2}\bm\Gamma -\bm{I}_K||_2 > \frac{\sqrt{M_n\epsilon_n}}{\sqrt{B}} \wedge \delta_2(c)\Bigm| \bbX_n \right) \\
    &&+ \pi\left(\sup_{l=1,\ldots,k}\left|\frac{\lambda_l}{\lambda_{0,l}}-1 \right| > \delta_1(c) \Bigm| \bbX_n \right),
    \eea 
    for all sufficiently large $n$.
    We show that each term in the upper bound converges to $0$ in probability.
    
    Let $\tilde\epsilon_n =  K^3/n + \frac{\lambda_{0,K+1}}{\lambda_{0,k}} \Big(\frac{p}{n}\vee 1\Big)$ and $\tilde{M}_n = 1/\sqrt{\tilde\epsilon_n} $. 
    Since $\tilde{M}_n\tilde\epsilon_n\lra 0$ and $\tilde{M}_n\lra \infty$, we have
    \bea \pi\left(\sup_{l=1,\ldots,k}\left|\frac{\lambda_l}{\lambda_{0,l}}-1 \right| > \delta_1(c) \Bigm| \bbX_n \right) 
&\le& \pi\left(\sup_{l=1,\ldots,k}\left|\frac{\lambda_l}{\lambda_{0,l}}-1 \right| > \tilde{M}_n\tilde\epsilon_n \Bigm| \bbX_n \right),
    \eea 
    for all sufficiently large $n$, and this converges to $0$ in probability by Theorem \ref{corr:IWeigenvalue}.

We have
    \bea 
    &&\pi\left( ||\bm\Gamma_\perp^T \bm\Sigma_0^{-1/2}\bm\Sigma \bm\Sigma_0^{-1/2} \bm\Gamma_\perp ||_2^{1/2}\frac{\sqrt{\lambda_{0,K+1}}}{\sqrt{\lambda_{0,k}} } > \sqrt{\epsilon_n}\Bigm| \bbX_n \right)\\
    &\le&\pi\left( ||\bm\Gamma_\perp^T \bm\Sigma_0^{-1/2}\bm\Sigma \bm\Sigma_0^{-1/2} \bm\Gamma_\perp ||_2^{1/2}\frac{\sqrt{\lambda_{0,K+1}}}{\sqrt{\lambda_{0,k}} } > \sqrt{M_n}\sqrt{\frac{\lambda_{0,K}}{\lambda_k}} \sqrt{\frac{p}{n}} \Bigm|\bbX_n \right),
    \eea 
    which converges to $0$ in probability by Lemma \ref{lem:IWeigenvalue}.
   
We have 
    \bea 
    &&\pi\left( || \bm\Gamma^T \bm\Sigma_0^{-1/2}\bm\Sigma \bm\Sigma_0^{-1/2}\bm\Gamma -\bm{I}_K||_2 > \frac{\sqrt{M_n\epsilon_n}}{\sqrt{B_k}} \wedge \delta_2(c)\Bigm| \bbX_n \right) \\
    &\le& \pi \left( || \bm\Gamma^T \bm\Sigma_0^{-1/2}\bm\Sigma \bm\Sigma_0^{-1/2}\bm\Gamma -\bm{I}_K||_2 > \frac{\sqrt{M_n\epsilon_n}}{\sqrt{B_k}} \Bigm| \bbX_n \right)\\
    &&+\pi\left( || \bm\Gamma^T \bm\Sigma_0^{-1/2}\bm\Sigma \bm\Sigma_0^{-1/2}\bm\Gamma -\bm{I}_K||_2 > \delta_2(c)\Bigm|\bbX_n \right)\\
    &\le& \pi\left( || \bm\Gamma^T \bm\Sigma_0^{-1/2}\bm\Sigma \bm\Sigma_0^{-1/2}\bm\Gamma -\bm{I}_K||_2 > \sqrt{M_n}\sqrt{K/n} \Bigm| \bbX_n \right)\\
    &&\pi\left( || \bm\Gamma^T \bm\Sigma_0^{-1/2}\bm\Sigma \bm\Sigma_0^{-1/2}\bm\Gamma -\bm{I}_K||_2 > (n/K)^{1/4}\sqrt{K/n} \Bigm|\bbX_n \right)
    \eea 
    for all sufficiently large $n$, where the last inequality satisfied since $ (n/K)^{1/4}\sqrt{K/n}  \le  (K/n)^{1/4}  \le \delta_2(c) $ for all sufficiently large $n$.
    By Lemma \ref{lem:IWeigenvalue}, the upper bound converges to $0$ in probability.

    Suppose $p\le n$. It suffices to show 
    \bea 
    &&\pi\left( || \bm\Sigma_0^{-1/2}\bm\Sigma \bm\Sigma_0^{-1/2} -\bm{I}_p||_2\frac{\sqrt{\lambda_{0,K+1}}}{\sqrt{\lambda_{0,k}} } > \sqrt{M_n\epsilon_n}\Bigm| \bbX_n\right)\\
    &\le&\pi\left( || \bm\Sigma_0^{-1/2}\bm\Sigma \bm\Sigma_0^{-1/2} -\bm{I}_p||_2\frac{\sqrt{\lambda_{0,K+1}}}{\sqrt{\lambda_{0,k}} } > 
    \sqrt{M_n}\frac{\sqrt{\lambda_{0,K+1}}}{\sqrt{\lambda_{0,k}} }\sqrt{p/n}
    \Bigm| \bbX_n \right) \\
    &=&\pi \left( || \bm\Sigma_0^{-1/2}\bm\Sigma \bm\Sigma_0^{-1/2} -\bm{I}_p||_2 > 
    \sqrt{M_n}\sqrt{p/n}
    \Bigm|\bbX_n \right) 
    \eea 
    converges to $0$ in probability.
By Theorem 1 in \cite{lee2018optimal}, we have
    \bea
    P\left(\pi\left(||\bm\Sigma_0^{-1/2} \bm\Sigma\bm\Sigma_0^{-1/2} -\bm{I}_p||_2  >\sqrt{M_n} \sqrt{p/n}\Bigm|\bbX_n \right) \right)&\lesssim& 
      \frac{1}{M_np/n}\\
      &\le& \frac{1}{M_n},
    \eea
    which converges to $0$.
    Thus,  
    \bea 
   \pi\left( || \bm\Sigma_0^{-1/2}\bm\Sigma \bm\Sigma_0^{-1/2} -\bm{I}_p||_2\frac{\sqrt{\lambda_{0,K+1}}}{\sqrt{\lambda_{0,k}} } > \sqrt{M_n\epsilon_n}\Bigm| \bbX_n \right)
     \eea 
     converges to $0$ in probability.

\end{proof}

\section{Proof of Proposition \ref{prop:minimaxlower}} \label{sec:pf3.7}

\begin{proof}[Proof of Proposition \ref{prop:minimaxlower}]
    First, we show 
    \bea 
1-  (\hat{\bm\xi}_p^T\bm\xi_p)^2  &\ge& \frac{1}{2}\min \{ ||\hat{\bm\xi}_p - \bm\xi_p||_2^2,  ||(-\hat{\bm\xi}_p) - \bm\xi_p||_2^2\}  .
    \eea 
    When $\hat{\bm\xi}_p^T\bm\xi_p \ge 0$, we have
    \bea 
    1-  (\hat{\bm\xi}_p^T\bm\xi_p)^2 &=&
    (1 + \hat{\bm\xi}_p^T\bm\xi_p) (1 - \hat{\bm\xi}_p^T\bm\xi_p) \\
    &\ge& (1 - \hat{\bm\xi}_p^T\bm\xi_p) \\
    &=& \frac{1}{2}||\hat{\bm\xi}_p - \bm\xi_p||_2^2.
    \eea 
    When $\hat{\bm\xi}_p^T\bm\xi_p < 0$, we have
    \bea 
    1-  (\hat{\bm\xi}_p^T\bm\xi_p)^2 &=&
    (1 + \hat{\bm\xi}_p^T\bm\xi_p) (1 - \hat{\bm\xi}_p^T\bm\xi_p) \\
    &\ge& (1 + \hat{\bm\xi}_p^T\bm\xi_p) \\
    &=& \frac{1}{2}||(-\hat{\bm\xi}_p) - \bm\xi_p||_2^2.
    \eea 
    Thus,
    \bea 
1-  (\hat{\bm\xi}_p^T\bm\xi_p)^2  &\ge& \frac{1}{2}\min \{ ||\hat{\bm\xi}_p - \bm\xi_p||_2^2,  ||(-\hat{\bm\xi}_p) - \bm\xi_p||_2^2\} .
    \eea 
    Then, 
    \bea 
    \inf_{\hat{\bm\xi}_p} \sup_{\bm\xi_p\in \mathbb{S}^{p-1}} E(1-  (\hat{\bm\xi}_p^T\bm\xi_p)^2  ) &\ge&
 \frac{1}{2}\inf_{\hat{\bm\xi}_p} \sup_{\bm\xi_p\in \mathbb{S}^{p-1}} E(||\hat{\bm\xi}_p - \bm\xi_p||_2^2),
    \eea 
    because $-\bm{\hat\xi}_p$ is also an estimator when $\bm{\hat\xi}_p$ is an estimator.
    Since   
    \bea 
    \inf_{\hat{\bm\xi}_p} \sup_{\bm\xi_p\in \mathbb{S}^{p-1}} E( ||\hat{\bm\xi}_p - \bm\xi_p||_2^2) 
    &\gtrsim & \min \Big\{ \frac{1+\nu_p}{\nu_p^2}\frac{p}{n}  ,1\Big\} 
    \eea 
    (see Example 15.19 in \cite{wainwright2019high}), the prove is completed.
\end{proof}

\section{Proof of Theorem \ref{thm:eigenbias}}
We give the proof of Theorem \ref{thm:eigenbias}.

\begin{proof}[Proof of Theorem \ref{thm:eigenbias}]
    Let 
    \bea 
    \bm\Omega = \begin{pmatrix}
        \bm\Omega_{11} & \bm\Omega_{12}\\
        \bm\Omega_{21} & \bm\Omega_{22}
    \end{pmatrix}, ~
    \bar{\bm\Omega}_{11}=\begin{pmatrix}
        \bm\Omega_{11} & \bm{O}\\
        \bm{O} & \bm{O}
    \end{pmatrix}
    ~,\bm{V} = \begin{pmatrix}
        \bm{O} & \bm\Omega_{12}\\
        \bm\Omega_{21} & \bm\Omega_{22}
    \end{pmatrix},
    \eea 
    and let $\lambda_1 \ge  \ldots \ge \lambda_K$ denote the eigenvalues of $\bm\Omega_{11}\in \calC_K$ and $\bm\xi_1,\ldots,\bm\xi_K\in \bbR^{K}$ denote the corresponding normalized eigenvectors. Then,
    $\lambda_1 \ge  \ldots \ge \lambda_K$ are the leading $K$ eigenvalues of $\bar{\bm\Omega}_{11}$ and the rest eigenvalues are $0$. The corresponding eigenvectors of $\bar{\bm\Omega}_{11}$ are 
    \bea 
    (\bm\xi_1^T,\bm{0}_{p-K})^T ,\ldots, (\bm\xi_K^T,\bm{0}_{p-K})^T, (\bm{0}_K^T, \bm{e}_1^T )^T, \ldots, (\bm{0}_K^T, \bm{e}_{p-K}^T )^T ,
    \eea 
    where $\bm{e}_j\in \bbR^{p-K}$ denotes the standard basis vector.
    We let $\bar{\bm\xi}_k$ denote $(\bm\xi_k^T,\bm{0}_{p-K}^T)^T$, $k=1,\ldots, K$.
    
    Applying (XVI.5) and the last display on page 720 in \cite{messiah2014quantum}, 
    we have
    \bea 
    \lambda_k(\bm\Omega) &=& \lambda_k(\bm\Omega_{11}) + \sum_{n=1}^\infty \epsilon_n ,~ k=1,\ldots, K,\\
    \epsilon_n &=& \sum_{(n-1)} \bar{\bm\xi}_k^T  \bm{V} \bm{S}^{d_1} \bm{V}\bm{S}^{d_2}\bm{V}\ldots \bm{V} \bm{S}^{d_n}\bm{V}  \bar{\bm\xi}_k,
    \eea 
    where
$\sum_{(n-1)}$ extends over all sets of the non-negative integers $d_1,\ldots,d_n$ such that $\sum_{i=1}^n d_i = n$, and $\bm{S}^d$ is derived as
    \bea 
    \bm{S}^d  =\begin{cases}
      -\begin{pmatrix}
\bm\xi_k\bm\xi_k^T  & \bm{O}\\
\bm{O} & \bm{O}
      \end{pmatrix} & d=0,\\
        \begin{pmatrix}
            \sum_{l\le K, l\neq k} \frac{1}{(\lambda_l - \lambda_k)^d}\bm\xi_l\bm\xi_l^T  & \bm{O}\\
            \bm{O} &  \frac{1}{\lambda_k^d}\bm{I}_{p-K} 
        \end{pmatrix}  & d >0.
    \end{cases}
    \eea 
    from (XVI.65) in \cite{messiah2014quantum}.

We have 
\bea 
    \epsilon_1 &=&  \bar{\bm\xi}_k^T \bm{V} \bm{S}^1 \bm{V} \bar{\bm\xi}_k\\
    &=& \begin{pmatrix}
        \bm{0}^T_K & \bm\xi_k^T \bm\Omega_{12} 
    \end{pmatrix} 
    \begin{pmatrix}
            \sum_{l\le K, l\neq k} \frac{1}{(\lambda_l - \lambda_k)}\bm\xi_l\bm\xi_l^T  & \bm{O}\\
            \bm{O} &  \frac{1}{\lambda_k}\bm{I}_{p-K} 
        \end{pmatrix}
        \begin{pmatrix}
        \bm{0}_K \\
        \bm\Omega_{21} \bm\xi_k 
    \end{pmatrix}
        \\
    &=&\frac{||  \bm\Omega_{21}\bm\xi_k ||^2}{\lambda_k}.
    \eea 

    Next, we give the upper bound of $\epsilon_n$ when $n\ge 2$.
    We have 
    \bean\label{eq:epsnupper} 
\epsilon_n &\le&  \sum_{(n-1)}  ||\bm{S}^{d_1/2}\bm{V}\bar{\bm\xi}_k||_2 ||\bm{S}^{d_n/2}\bm{V}\bar{\bm\xi}_k||_2  \prod_{j=1}^{n-1} || \bm{S}^{d_j/2}\bm{V}\bm{S}^{d_{j+1}/2} ||_2.
    \eean

    When $d_j,d_{j+1}>0$, we have 
    \bea 
    \bm{S}^{d_j/2}\bm{V} \bm{S}^{d_{j+1}/2} &=& 
   \frac{1}{\lambda_k^{(d_j+d_{j+1})/2}}  \begin{pmatrix}
        \bm{O} & \bm{O}\\
        \bm{O} & \bm\Omega_{22}
    \end{pmatrix} \\
    &&+ \frac{1}{\lambda_k^{(d_j+d_{j+1})/2}} \begin{pmatrix}
        \bm{O} & \sum_{l\le K, l\neq k} \frac{\lambda_k^{d_j/2}}{(\lambda_l - \lambda_k)^{d_j/2}} (\bm\Omega_{21} \bm\xi_l \bm\xi_l^T)^T\\
        \sum_{l\le K, l\neq k} \frac{\lambda_k^{d_{j+1}/2}}{(\lambda_l - \lambda_k)^{d_{j+1}/2}}\bm\Omega_{21} \bm\xi_l \bm\xi_l^T & \bm{O}
    \end{pmatrix},
    \eea 
and 
    \bea 
    &&||\bm{S}^{d_j/2}\bm{V} \bm{S}^{d_{j+1}/2} ||\\
    &\le& \frac{2}{\lambda_k^{(d_j+d_{j+1})/2}} \max \Big[
    ||\bm\Omega_{22}||_2 , \Big( \sum_{l\le K, l\neq k} \Big|\frac{\lambda_k}{\lambda_l-\lambda_k} \Big|^{d_j/2} \vee \sum_{l\le K, l\neq k} \Big|\frac{\lambda_k}{\lambda_l-\lambda_k} \Big|^{d_{j+1}/2}  \Big)||\bm\Omega_{21}\bm\xi_l||_2
    \Big] \\
    &\le& \Big(\frac{2C}{\lambda_k}\Big)^{(d_j+d_{j+1})/2} \max \Big[
    ||\bm\Omega_{22}||_2 , \Big( \sum_{l\le K, l\neq k} \Big|\frac{\lambda_k}{C(\lambda_l-\lambda_k)} \Big|^{d_j/2}  \vee \sum_{l\le K, l\neq k} \Big|\frac{\lambda_k}{C(\lambda_l-\lambda_k)} \Big|^{d_{j+1}/2} \Big) ||\bm\Omega_{21}\bm\xi_l||_2
    \Big],
    \eea 
    for arbitrary $C>1$.
    When $d_j=d_{j+1}=0$, 
    \bea 
    \bm{S}^{d_j/2}\bm{V} \bm{S}^{d_{j+1}/2} &=& \bm{O}.
    \eea 
    When $d_j>0$ and $d_{j+1}=0$,
    \bea 
\bm{S}^{d_j/2}\bm{V} \bm{S}^{d_{j+1}/2} &=& \begin{pmatrix}
    \frac{1}{\lambda_k^{d_j/2}}\bm\Omega_{21}\bm\xi_k\bm\xi_k^T & \bm{O}\\
    \bm{O} & \bm{O}
\end{pmatrix},
    \eea 
    and 
    \bea 
||\bm{S}^{d_j/2}\bm{V} \bm{S}^{d_{j+1}/2}||_2 &\le& \frac{1}{\lambda_k^{d_j/2}} ||\bm\Omega_{21}\bm\xi_k||_2\\
&=&\frac{1}{\lambda_k^{(d_j+d_{j+1})/2}} ||\bm\Omega_{21}\bm\xi_k||_2.
    \eea 
    Likewise, when $d_j=0$ and $d_{j+1}>0$, 
        \bea 
||\bm{S}^{d_j/2}\bm{V} \bm{S}^{d_{j+1}/2}||_2 &\le& \frac{1}{\lambda_k^{(d_j+d_{j+1})/2}} ||\bm\Omega_{21}\bm\xi_k||_2.
    \eea  
    Since $||\bm\Omega_{22}||_2 \le x$ and 
$\Big( 1\vee \sum_{l\le K, l\neq k} \Big|\frac{\lambda_k}{C(\lambda_l-\lambda_k)} \Big|^{d_j/2} \vee \sum_{l\le K, l\neq k} \Big|\frac{\lambda_k}{C(\lambda_l-\lambda_k)} \Big|^{d_{j+1}/2} \Big) ||\bm\Omega_{21}\bm\xi_l||_2 \le x$ by the assumption, collecting all the cases, we obtain 
    \bea 
||\bm{S}^{d_j/2}\bm{V} \bm{S}^{d_{j+1}/2}||_2  \le \Big(\frac{2C}{\lambda_k}\Big)^{(d_j+d_{j+1})/2}  x, ~ d_j\ge 0, ~d_{j+1}\ge 0.
    \eea 
We have 
\bea 
    \bm{S}^{d/2}  \bm{V} \bar{\bm\xi}_k = \begin{pmatrix}
        \bm{O} & \bm{O}\\
         \frac{1}{\lambda_k^{d/2}}\bm\Omega_{21}\bm\xi_k & \bm{O}
    \end{pmatrix}
    \eea 
and 
    \bea 
    ||\bm{S}^{d/2}  \bm{V} \bar{\bm\xi}_k ||_2 \le \frac{1}{\lambda_k^{d/2}} ||\bm\Omega_{21}\bm\xi_k|| \le \Big(\frac{2C}{\lambda_k}\Big)^{d/2}  x
    \eea 
Thus, 
\bea 
    \eqref{eq:epsnupper}
    &\le& \sum_{(n-1)} \Big(\frac{2C}{\lambda_k}\Big)^{\sum_{i=1}^n d_i }   x^{n+1} \\
    &\le & \lambda_k \Big( \frac{2Cx}{\lambda_k}\Big)^{n+1} \sum_{(n-1)}1\\
    &\le &  \lambda_k     \Big( \frac{4eCx}{\lambda_k}\Big)^{n+1} 
    \eea  
    where the last inequality is satisfied since   \bea 
\sum_{(n-1)}1
&=& \binom{2n-1}{n} \\
&\le& (2e)^n,
\eea
and we obtain 
\bea 
\lambda_k(\bm\Sigma) &=& \lambda_k(\bm\Omega_{11}) \Big[ 1+ \frac{||  \bm\Omega_{21}\bm\xi_k ||^2}{\lambda_k(\bm\Omega_{11})^2} +R \Big],\\
R&\le&  \sum_{n=2}^\infty     \Big( \frac{4eCx}{\lambda_k(\bm\Omega_{11})}\Big)^{n+1}\\
&=& \Big( \frac{4eCx}{\lambda_k(\bm\Omega_{11})}\Big)^{3} \Big(1- \frac{4eCx}{\lambda_k(\bm\Omega_{11})}\Big)^{-1}.
\eea 
    \end{proof}

\section{Proof of Theorem \ref{thm:IWeigenbias}}\label{sec:pf3.2}
For the proof Theorem \ref{thm:IWeigenbias}, we give the following lemma.  

\begin{lemma}\label{lemma:nonspikeeigen}
Suppose that $\bm{X}_i,\ldots,\bm{X}_n$ are i.i.d. samples with zero mean vector and $\textup{Cov}(\bm{X}_i) = \bm\Sigma$ with Assumptions 1-4, which are given in the main manuscript. 
Let $\hat\lambda_{K+1}$ and $\lambda_{0,K+1}$ denote the $(K+1)$th eigenvalues of the sample covariance $\bm{S}_n = \frac{1}{n} \sum_{i=1}^n \bm{X}_i \bm{X}_i^T $ and the population covariance $\bm\Sigma_0$, respectively.
For any $\delta$, there exist $C>1$ and $n_0$ such that when $n\ge n_0$ 
\bea 
P( \hat\lambda_{K+1} >  C\lambda_{0,K+1}(1+\sqrt{p/n})^2  ) <\delta.
\eea 

\end{lemma}
\begin{proof}

From Theorem 2.5 of \cite{cai2020limiting}, for any \( \varepsilon > 0 \), we have with high probability:
\[
|\hat\lambda_{K+1} - \nu_1| \le n^{-2/3 + \varepsilon},
\]
where $\nu_1$ is a positive real number satisfying \( \nu_1 \le  \lambda_{0,K+1}|| \frac{1}{n}  \sum_{i=1}^n \bm{Z}_i \bm{Z}_i^T||_2 \) and $\bm{Z}_i = \bm\Sigma_0^{-1/2}\bm{X}_i$.
The classical result in random matrix theory (e.g., \cite{bai2010spectral}) yields
\[
\left\| \frac{1}{n}  \sum_{i=1}^n \bm{Z}_i \bm{Z}_i^T \right\| \xrightarrow{\text{a.s.}} (1 + \sqrt{p/n})^2.
\]
Thus, for any $\delta$, there exist $C_1>0$ and $n_0$ such that when $n\ge n_0$ 
\bea 
P( \hat\lambda_{K+1} >  (1+C_1)\lambda_{0,K+1}(1+\sqrt{p/n})^2 + n^{-2/3 + \varepsilon} ) <\delta.
\eea 
Since $n^{-2/3 + \varepsilon} $ is dominated by the other term, the proof is completed.

\end{proof}

\begin{lemma}\label{lemma:loweigenvaluevector}
Suppose $\bm\Omega \mid \bbX_n \sim IW_K( (n+\nu_n-2p-2)\hat{\bm\Lambda}, n+ \nu_n -2p+2K ) $, where $\hat{\bm\Lambda} = \textup{diag}(\hat\lambda_1,\ldots,\hat\lambda_K)$ with $\sup_{k=1,\ldots,K-1}\frac{\hat\lambda_k}{\hat\lambda_{k+1}} >C$ for some positive constant $C>1$.
If $K^3/n = o(1)$, then 
\bean 
\pi( \Big| \frac{\lambda_k(\bm\Omega_{11})}{\hat\lambda_k} -1 \Big| > \sqrt{\frac{K^3}{n}}  \mid \bbX_n)&\to& 0 \label{eq:loweigenvalue}\\
\pi( 1-\xi_{k,k}^2 > \frac{K}{n} \mid \bbX_n) &\to& 0.\label{eq:loweigenvector}
\eean 
    
\end{lemma}

\begin{proof}

    By Theorem \ref{thm:loweigen}, we have
    \bea 
    \frac{\lambda_k(\bm\Omega)}{\hat\lambda_k} -1  \le  K || \bm{\tilde\Omega}  -\bm{I}||_2,
    \eea 
where $\bm{\tilde\Omega} = \bm{\hat\Lambda}^{-1/2}\bm\Omega \bm{\hat\Lambda}^{-1/2} \sim IW_K ((n+\nu_n-2p-2) \bm I, n+\nu_n -2p +2K)$.
Since $|| \bm{\tilde\Omega} - \bm{I}  ||^2 = O_p(\frac{K}{n})$ by Theorem 1 in \cite{lee2018optimal}, \eqref{eq:loweigenvalue} is satisfied.

    By Lemma \ref{lemma:evector}, we have
    \bea
    \xi_{k,k}^2 = \frac{1}{ 1+ || \sqrt{\hat\lambda_k} \hat{\bm\Lambda}_{-k}^{-1/2} ( \tilde{\bm\Omega} - \lambda_k \hat{\bm\Lambda}_{-k}^{-1})^{-1} \bm{U}_{0,-k}^T \bm{\tilde\Omega} \bm{u}_{0,k}  ||^2},
    \eea 
    where $\lambda_k=\lambda_k(\bm\Omega)$, $\bm{u}_{0,k} = \bm{e}_k$ and $\bm{U}_{0,-k} = [\bm{e}_1,\ldots, \bm{e}_{k-1},\bm{e}_{k+1},\ldots, \bm{e}_K]$.
    Let $\bm{D} = (\bm{I} - \lambda_k \hat{\bm\Lambda}_{-k}^{-1})^{-1}$. We have
    \bea 
    ( \tilde{\bm\Omega} - \lambda_k \hat{\bm\Lambda}_{-k}^{-1})^{-1} &=&
        ( \tilde{\bm\Omega} - \bm I + \bm I - \lambda_k \hat{\bm\Lambda}_{-k}^{-1})^{-1} \\
        &=&
    \bm{D} ( \bm{I} + \bm{D} (\tilde{\bm\Omega} - \bm{I} ))^{-1},
    \eea 
    and
    \bea 
    1- \xi_{k,k}^2 &\le& || \sqrt{\hat\lambda_k} \hat{\bm\Lambda}_{-k}^{-1/2} ( \tilde{\bm\Omega} - \lambda_k \hat{\bm\Lambda}_{-k}^{-1})^{-1} \bm{U}_{0,-k}^T \bm{\tilde\Omega} \bm{u}_{0,k}  ||^2\\
    &\le& 2|| \sqrt{\lambda_k} \hat{\bm\Lambda}_{-k}^{-1/2} ( \tilde{\bm\Omega} - \lambda_k \hat{\bm\Lambda}_{-k}^{-1})^{-1} \bm{U}_{0,-k}^T \bm{\tilde\Omega} \bm{u}_{0,k}  ||^2\\
    &= & 2|| \sqrt{\lambda_k} \hat{\bm\Lambda}_{-k}^{-1/2} \bm{D}  ( \bm{I}+ \bm{D} ( \tilde{\bm\Omega} - \bm{I}))^{-1} \bm{U}_{0,-k}^T \bm{\tilde\Omega} \bm{u}_{0,k}  ||^2\\
    &\le& 2 || \sqrt{\lambda_k} \hat{\bm\Lambda}_{-k}^{-1/2}\bm{D}||^2  ||\bm{U}_{0,-k}^T \bm{\tilde\Omega} \bm{u}_{0,k}  ||^2 \Big(\frac{1}{1- || \bm{D} ( \tilde{\bm\Omega} - \bm{I}) ||} \Big)^2 \\
    &\le& 4 || \sqrt{\hat\lambda_k} \hat{\bm\Lambda}_{-k}^{-1/2}\bm{D}||^2  || \bm{\tilde\Omega} - \bm{I}  ||^2  \Big(\frac{1}{1- || \bm{D} || || \tilde{\bm\Omega} - \bm{I} ||} \Big)^2 ,
    \eea 
    where the second and last inequalities are satisfied since $\lambda_k/\hat\lambda_k $ converges to $1$.

    We have 
    \bea 
    || \sqrt{\lambda_k} \hat{\bm\Lambda}_{-k}^{-1/2}\bm{D}|| &\le& 
    2|| \sqrt{\hat\lambda_k} \hat{\bm\Lambda}_{-k}^{-1/2}\bm{D}|| \\
    &=&2|| \textup{diag} ( \frac{\sqrt{\hat\lambda_k \hat\lambda_1}}{\hat\lambda_1 -\hat\lambda_k },\ldots,  \frac{\sqrt{\hat\lambda_k \hat\lambda_{k-1}}}{\hat\lambda_{k-1} -\hat\lambda_k }, \frac{\sqrt{\hat\lambda_k \hat\lambda_{k+1}}}{\hat\lambda_{k+1} -\hat\lambda_k },\ldots, \frac{\sqrt{\hat\lambda_k \hat\lambda_K}}{\hat\lambda_K -\hat\lambda_k } )||,
    \eea 
    which is bounded above by a positive constant.
    Since 
    \bea 
    \bm{D} &=& \textup{diag} ( \frac{\hat\lambda_1}{\hat\lambda_1-\lambda_k},\ldots, \frac{\hat\lambda_{k-1}}{\hat\lambda_{k-1}-\lambda_k}, \frac{\hat\lambda_{k+1}}{\hat\lambda_{k+1}-\lambda_k},\ldots,  \frac{\hat\lambda_K}{\hat\lambda_K-\lambda_k}  ),
    \eea 
    $||\bm{D}||$ is also bounded above by a positive constant.
    Thus, $1-\xi_{k,k}^2 \le  C_1 || \bm{\tilde\Omega} - \bm{I}  ||^2 \frac{1}{1- C_1|| \bm{\tilde\Omega} - \bm{I}  ||^2}$ for some positive constant $C_1$. 
    Since $|| \bm{\tilde\Omega} - \bm{I}  ||^2 = O_p(\frac{K}{n})$ by Theorem 1 in \cite{lee2018optimal}, \eqref{eq:loweigenvector} is satisfied.

\end{proof}

\begin{proof}[Proof of Theorem \ref{thm:IWeigenbias}]

Let $\hat{\bm\Sigma} = \frac{n\bm{S}_n + \bm{A}_n}{n+\nu_n -2p-2}$, and let $\hat{\bm\Gamma} = (\hat{\bm\Gamma}_1,\hat{\bm\Gamma}_2)\in \bbR^{p\times p}$ denote the eigenvector matrices of $\hat{\bm\Sigma}$ with $\hat{\bm\Gamma}_1\in \bbR^{p\times K}$ and $\hat{\bm\Gamma}_2\in \bbR^{p\times (p-K)}$. 
Let
\bea 
\hat{\bm\Gamma}\bm\Sigma \hat{\bm\Gamma}^T  = \bm\Omega = \begin{pmatrix}
    \bm\Omega_{11} & \bm\Omega_{12}\\
    \bm\Omega_{21} & \bm\Omega_{22}
\end{pmatrix} =
\begin{pmatrix}
    \hat{\bm\Gamma}_1^T \bm\Sigma \hat{\bm\Gamma}_1 & \hat{\bm\Gamma}_1^T \bm\Sigma \hat{\bm\Gamma}_2\\
    \hat{\bm\Gamma}_2^T \bm\Sigma \hat{\bm\Gamma}_1 & \hat{\bm\Gamma}_2^T \bm\Sigma \hat{\bm\Gamma}_2
\end{pmatrix}.
\eea 
Let $\bm\xi_k$ denote the $k$th eigenvector of $\bm\Omega_{11}$.
By applying Theorem \ref{thm:eigenbias}, 
\bea 
\lambda_k(\bm\Sigma) &=& \lambda_k(\bm\Omega_{11})  \Big[ 1 + \frac{||\bm\Omega_{21}\bm\xi_k||^2}{\lambda_k(\bm\Omega_{11})^2}  + R\Big]\\
&=&\lambda_k(\bm\Omega_{11})  \Big[ 1 + \frac{||[\bm\Omega_{21}]_k||^2}{\lambda_k(\bm\Omega_{11})^2} +\tilde{R}\Big],
\eea 
with $R \le \left( \frac{4eCx }{\lambda_k(\bm{\Omega}_{11})} \right)^3 \left(1 - \frac{4eCx}{\lambda_k(\bm{\Omega}_{11})} \right)^{-1}$,
$||\bm\Omega_{22}||_2 \le x$ and \\
$\Big( 1\vee \sum_{l\le K, l\neq k} \Big|\frac{\lambda_k(\bm{\Omega}_{11})}{C(\lambda_l(\bm{\Omega}_{11})-\lambda_k(\bm{\Omega}_{11}))} \Big|^{d/2}  \Big) ||\bm\Omega_{21}\bm\xi_l||_2 \le x$.
Here, $\tilde{R}= \frac{||\bm\Omega_{21}\bm\xi_k||^2-||[\bm\Omega_{21}]_k||^2}{\lambda_k(\bm\Omega_{11})^2}   + R$ and the constant $C$ is a sufficiently large constant to be specified in this proof.

Let $\pi(\cdot \mid \bbX_n)$ denote the posterior distribution of $\bm\Sigma\mid \bbX_n \sim IW_p((n+\nu_n -2p-2)\hat{\bm\Sigma}, n+\nu_n)$ and let $E(\cdot)$ denote the expectation of random observations $\bbX_n$.
We show that there exists a positive constant $C_1>0$ such that the followings converge to $0$
\bean 
E\Big[\pi (||\bm\Omega_{21} \bm\xi_k||_2 >  \Big(  C_1\frac{p\lambda_{0,k}\log n}{n}  \Big)^{1/2} \mid \bbX_n) \Big] , \label{eq:corrcond1}\\
E\Big[\pi (||\bm\Omega_{22}||_2 > C_1\frac{p}{n} \mid \bbX_n) \Big] ,\label{eq:corrcond3} \\
E\Big[ \pi (\Big|\frac{\lambda_k(\bm{\Omega}_{11})}{C(\lambda_l(\bm{\Omega}_{11})-\lambda_k(\bm{\Omega}_{11}))} \Big| >  1\mid \bbX_n ) \Big] \label{eq:corrcond2}\\
 E \Big[ \pi\Big( || \Big( \frac{||\bm\Omega_{21}\bm\xi_k||^2-||[\bm\Omega_{21}]_k||^2}{\lambda_k(\bm\Omega_{11})^2} ||  >  a_n \frac{p}{n\lambda_{0,k}}\Big)  \Big]\label{eq:corrcond4} ,
\eean 
for some positive sequence $a_n\to 0$.

By Lemma \ref{lemma:loweigenvaluevector}, $\lambda_k(\bm\Omega_{11})$ converges to $\lambda_k(\bm{\hat\Sigma})$. Since $n + \nu_n -2p-2 = O(n)$, $\lambda_k(\bm{\hat\Sigma})\lesssim \lambda_{0,k}$ and $\lambda_k(\bm{\hat\Sigma})\gtrsim \lambda_{0,k}$. Thus $\lambda_k(\bm\Omega_{11})$ is asymptotically equivalent to $\lambda_{0,k}$ with high probability.

By the convergence of \eqref{eq:corrcond1}-\eqref{eq:corrcond2}, we can set $x\asymp \sqrt{\frac{p\lambda_{0,k}\log n}{n}} \vee \frac{p}{n}$, which makes $R\lesssim \Big( \frac{p \log n}{n\lambda_{0,k}} \Big)^{3/2}$
Combining the convergence of \eqref{eq:corrcond4}, we obtain 
\bea 
E[\pi( \tilde{R}  > C_2\Big\{ a_n + (\log n)^{3/2}\Big(\frac{p}{n\lambda_{0,k}} \Big)^{1/2} \Big\} \frac{p}{n\lambda_{0,k}}   ) ] \to 0,
\eea 
for some positive constant $C_2$, which means $ \tilde{R}  = o_p(\frac{p}{n\lambda_{0,k}} )$ by the assumption $\frac{p}{n\lambda_{0,k}}(\log n)^3 = o(1)$.

Now, we show that \eqref{eq:corrcond1} converges to $0$.
Let $\hat{\bm\Lambda}_1 = \textup{diag}(\hat\lambda_1,\ldots,\hat\lambda_K)$ and $\hat{\bm\Lambda}_2 = \textup{diag}(\hat\lambda_{K+1},\ldots,\hat\lambda_p)$, where $\hat\lambda_j$ is the $j$th eigenvalue of $\hat{\bm\Sigma}$.
Define
\bea \bm{\tilde{\Omega}} &=& 
\begin{pmatrix}
    \tilde{\bm\Omega}_{11} & \tilde{\bm\Omega}_{12}\\
    \tilde{\bm\Omega}_{21} & \tilde{\bm\Omega}_{22}
\end{pmatrix}
\\
&=&  \begin{pmatrix}
    \hat{\bm\Lambda}_1^{-1/2} & \bm{O}\\
    \bm{O} & \bm{I}
\end{pmatrix}\bm\Omega \begin{pmatrix}
    \hat{\bm\Lambda}_1^{-1/2} & \bm{O}\\
    \bm{O} & \bm{I} 
\end{pmatrix} \\
&\sim& IW_p\Big[  (n+\nu_n-2p-2) \begin{pmatrix}
    \bm{I} & \bm{O}\\
    \bm{O} & \hat{\bm\Lambda}_2
\end{pmatrix} ,n+\nu_n\Big],
\eea 
and 
\bea 
 \bm{\breve{\Omega}} &=& 
\begin{pmatrix}
    \breve{\bm\Omega}_{11} & \breve{\bm\Omega}_{12}\\
    \breve{\bm\Omega}_{21} & \breve{\bm\Omega}_{22}
\end{pmatrix}
\\
&=&  \begin{pmatrix}
    \hat{\bm\Lambda}_1^{-1/2} & \bm{O}\\
    \bm{O} & \hat{\bm\Lambda}_2^{-1/2}
\end{pmatrix}\bm\Omega \begin{pmatrix}
    \hat{\bm\Lambda}_1^{-1/2} & \bm{O}\\
    \bm{O} & \hat{\bm\Lambda}_2^{-1/2}
\end{pmatrix} \\
&\sim& IW_p(  (n+\nu_n-2p-2)
    \bm{I}_p  ,n+\nu_n). 
\eea 
We have 
\bea 
||\bm\Omega_{21}  \bm\xi_k||_2^2 &=&
||\bm\Omega_{21} \hat{\bm\Lambda}_1^{-1/2}\hat{\bm\Lambda}_1^{1/2} \bm\xi_k||_2^2 \\
&=& ||\tilde{\bm\Omega}_{21} \hat{\bm\Lambda}_1^{1/2} \bm\xi_k||_2^2\\
&=& \Big| \Big|\sum_{j=1}^K \sqrt{\hat\lambda_{j}} \xi_{k,j}  [\tilde{\bm\Omega}_{21}]_j\Big| \Big|^2\\
&\le&  (\sum_{j=1}^K \sqrt{\hat\lambda_{j}}  |\xi_{k,j}|~  ||[\tilde{\bm\Omega}_{21}]_j||_2)^2
\eea 
where $[\tilde{\bm\Omega}_{21}]_j$ is the $j$th column vector of $\tilde{\bm\Omega}_{21}$.

By Lemma 7 in \cite{yata2012effective}, there exists $C_3$ such that 
\bea 
 P( \sum_{K+1\le i\le p} \hat\lambda_i> C_3p)  \lra 0,
\eea 
and we have  
    \bea 
&&E[\pi( || [\tilde{\bm\Omega}_{21}]_j||_2^2 > C_1\frac{p \log n}{n} \mid \bbX_n ) ]\\
&=& 
E[\pi( \sum_{K+1\le i\le p} \tilde{\omega}_{ij}^2 > C_1\frac{p \log n}{n} \mid\bbX_n ) ]\\
&\le& E[\pi( \sum_{K+1\le i\le n} \tilde{\omega}_{ij}^2 > \frac{C_1}{2}\frac{p \log n}{n} \mid\bbX_n ) ] +E[\pi( \sum_{n+1\le i\le p} \tilde{\omega}_{ij}^2 > \frac{C_1}{2}\frac{p \log n}{n} \mid\bbX_n ) ]\\
&\le& P( \sum_{K+1\le i\le p} \hat\lambda_i  > C_3 p ) + E[\pi( \sum_{n+1\le i\le p} \tilde{\omega}_{ij}^2 > \frac{C_1}{2}\frac{p \log n}{n} \mid\bbX_n ) ]\\ 
&&+ E[\pi( \sum_{K+1\le i\le n} \tilde{\omega}_{ij}^2 >   \frac{C_1}{2C_3} \log n \frac{\sum_{K+1\le i\le n} \hat\lambda_i }{n} \mid\bbX_n ) I(\sum_{K+1\le i\le p} \hat\lambda_i  \le C_3 p) ] \\
&\le& P( \sum_{K+1\le i\le p} \hat\lambda_i  > C_3 p ) + E[\pi( \sum_{n+1\le i\le p} \tilde{\omega}_{ij}^2 > \frac{C_1}{2}\frac{p \log n}{n} \mid\bbX_n ) ]\\ 
&&+  \sum_{K+1\le i\le n} E[\pi( \tilde{\omega}_{ij}^2/\hat\lambda_i>   \frac{C_1}{2C_3}  \frac{\log n }{n} \mid\bbX_n ) ] \\
&\le& P( \sum_{K+1\le i\le p} \hat\lambda_i  > C_3 p ) + \sum_{n+1\le i\le p}  E[\pi( \tilde{\omega}_{ij}^2 > \frac{C_1}{2}\frac{ \log n}{n} \mid\bbX_n ) ]\\ 
&&+  \sum_{K+1\le i\le n} E[\pi( \breve{\omega}_{ij}^2>   \frac{C_1}{2C_3}  \frac{\log n }{n} \mid\bbX_n ) ] 
\eea 
    where $\tilde\omega_{ij}$ and $\breve\omega_{ij}$ denote the $(i,j)$ elements of $\bm{\tilde\Omega}$ and $\bm{\breve\Omega}$, respectively.
By Lemma S1.4 in \cite{lee2023postecon}, there exists a positive constant $C_4$ such that
\bea 
\sum_{K+1\le i\le n } E[ \pi ( \breve\omega_{ij}^2 > \frac{C_1}{2C_3}  \frac{\log n}{n} \mid\bbX_n)]&\le& C_4 n [\exp(-C_4 n) +  \exp( -C_4 \log n)   ],\\
\sum_{n+1 \le  i\le p} E[ \pi ( \tilde\omega_{ij}^2 > \frac{C_1}{2} \frac{ \log n}{n } \mid\bbX_n)] &\le& C_4 p [\exp(-C_4 n) + \exp( -C_4 \frac{n \log n}{||\bm{A}||_2})],
\eea 
for all sufficiently large $n$, where the second inequality is satisfied because $||\hat{\bm\Lambda}_2||\le ||\bm{A}_n||/n$.
When $C_4$ is a sufficiently large positive constant and $||\bm{A}||_2$ is bounded, we obtain 
\bean\label{eq:tildeOmega21} 
E[\pi( || [\tilde{\bm\Omega}_{21}]_j||_2^2 > C_1\frac{p \log n}{n} \mid\bbX_n ) ] \to 0
\eean 
 We can make $C_4$ sufficiently large by setting $C_1$ large enough.
Thus, $E[\pi( || [\tilde{\bm\Omega}_{21}]_j||_2^2 > C_1\frac{p \log n}{n} \mid\bbX_n ) ] $ converges to $0$,
and we obtain 
\bea 
\eqref{eq:corrcond1}
&\le& E [ \pi \{  (\sum_{j=1}^K \sqrt{\hat\lambda_{j}}  |\xi_{k,j}|~  ||[\tilde{\bm\Omega}_{21}]_j||_2)^2  > \lambda_{0,k} \mid\bbX_n\} )]   + o(1) \\
&\le& E [ \pi \{  K \hat\lambda_{1} (1-\xi_{k,k}^2)  > \lambda_{0,k} \mid\bbX_n\} )]   + o(1) \\
&\le& E [ \pi \{   \frac{\lambda_{0,1} }{\lambda_{0,k}}\frac{ K}{n}   > C_5 \mid\bbX_n\} )]   + o(1),
\eea 
for some positive constant $C_5$, where the last inequality is satisfied by Lemma \ref{lemma:loweigenvaluevector}.
The upper bound of \eqref{eq:corrcond1} converges to $0$ since $\frac{\lambda_{0,1}}{\lambda_{0,K}} \frac{ K}{n} =o(1)$ by the assumption.

Next, we show that \eqref{eq:corrcond4} converges to $0$. 
We have 
\bea 
&&\Big| ||\bm\Omega_{21}\bm\xi_k||^2 - ||[\bm\Omega_{21}]_k||^2 \Big| \\
&=&
\Big| || \sum_{j=1}^K \xi_{k,j}[\bm\Omega_{21}]_j||^2 - ||[\bm\Omega_{21}]_k||^2 \Big| \\
&= & (1-\xi_{k,k}^2 )  ||[\bm\Omega_{21}]_k||^2 + ||  \sum_{j\neq k} \xi_{k,j}[\bm\Omega_{21}]_j ||^2 + 2 ( \sum_{j\neq k} \xi_{k,j}[\bm\Omega_{21}]_j)^T (\xi_{k,k}[\bm\Omega_{21}]_k) \\
&\le & (1-\xi_{k,k}^2 )  \hat\lambda_k||[\tilde{\bm\Omega}_{21}]_k||^2 + 
K(1-\xi_{k,k}^2 ) \hat\lambda_1 \sup_{j\neq k}||[\tilde{\bm\Omega}_{21}]_j ||^2
 + 2 \sqrt{K} (1-\xi_{k,k}^2)^{1/2} \sqrt{\hat\lambda_1}\sqrt{\hat\lambda_k}\sup_{j}||[\tilde{\bm\Omega}_{21}]_j ||^2 \\
 &\le& (K+1)(1-\xi_{k,k}^2 ) \hat\lambda_1 \sup_{j}||[\tilde{\bm\Omega}_{21}]_j ||^2
 + 2 \sqrt{K} (1-\xi_{k,k}^2)^{1/2} \sqrt{\hat\lambda_1\hat\lambda_k}\sup_{j}||[\tilde{\bm\Omega}_{21}]_j ||^2,
\eea
and
\bea 
\eqref{eq:corrcond4} &\le&  E \Big[ \pi\Big(  \frac{(K+1)(1-\xi_{k,k}^2 ) \hat\lambda_1 
 + 2 \sqrt{K} (1-\xi_{k,k}^2)^{1/2} \sqrt{\hat\lambda_1\hat\lambda_k}}{\lambda_k(\bm\Omega_{11})^2}\sup_{j}||[\tilde{\bm\Omega}_{21}]_j ||^2 >  a_n \frac{p}{n\lambda_{0,k}}\Big)  \Big]\\
&\le& E \Big[ \pi\Big(  \frac{(K+1)(1-\xi_{k,k}^2 ) \hat\lambda_1 
 + 2 \sqrt{K} (1-\xi_{k,k}^2)^{1/2} \sqrt{\hat\lambda_1\hat\lambda_k}}{\lambda_k(\bm\Omega_{11})^2} C_1\log n >  a_n \frac{1}{\lambda_{0,k}}\Big)  \Big] + o(1)\\
 &\le& E \Big[ \pi\Big( 
 \Big\{ \frac{K}{n} \frac{\lambda_{0,1}}{\lambda_{0,k}} + 
 (\frac{K}{n}\frac{\lambda_{0,1}}{\lambda_{0,k}} )^{1/2} \Big\} C_6\log n >  a_n \Big)  \Big] + o(1),
\eea 
for some positive constant $C_6$, where the second inequality is satisfied by \eqref{eq:tildeOmega21}, and the third inequality is satisfied by Theorems \ref{corr:IWeigenvalue} and \ref{corr:IWeigenvector}.
The upper bound of \eqref{eq:corrcond4} converges to $0$ by setting $a_n^2 =  (\frac{K}{n}\frac{\lambda_{0,1}}{\lambda_{0,k}})^{1/2}\log n  $ that converges to $0$ by assumption.

We give the upper bound of \eqref{eq:corrcond3}. 
Note that $\bm\Omega_{22} \sim IW_{p-K} ( (n+\nu-2p-2) \hat{\bm\Lambda}_2 , n+\nu-2K)$.
Let 
\bea 
\bm\Omega_{22} =\begin{pmatrix}
    \bm\Omega_{22,(1)} & * \\
    * & \bm\Omega_{22,(2)} 
\end{pmatrix},~ \hat{\bm\Lambda}_2 = \begin{pmatrix}
    \hat{\bm\Lambda}_{2,(1)} & \bm{O}\\
    \bm{O} & \hat{\bm\Lambda}_{2,(2)}
\end{pmatrix}, 
\eea 
with $\bm\Omega_{22,(1)},\hat{\bm\Lambda}_{2,(1)} \in \calC_{n-K}$ and $\bm\Omega_{22,(2)} ,\hat{\bm\Lambda}_{2,(2)}\in \calC_{p-n}$.
By Lemma \ref{lemma:nonspikeeigen}, there exists a positive constant $C_7$ such that 
\bea 
P\left(||\hat{\bm\Lambda}_2 || >  C_7\Big(\frac{p}{n}\vee 1\Big) \right)
\eea 
converges to $0$.
We have
    \bea
    \eqref{eq:corrcond3}
    &\le& P\left(||\hat{\bm\Lambda}_2 || >  C_7\Big(\frac{p}{n}\vee 1\Big) \right) \\
    &&+E[ \pi(  || \bm\Omega_{22,(1)}||_2 >C_1(\frac{p}{n}\vee 1)/4 \mid\bbX_n ) I (  ||\hat{\bm\Lambda}_2 || \le   C_7\Big(\frac{p}{n}\vee 1\Big) ) ] \\
    &&+ E[ \pi(  || \bm\Omega_{22,(2)}||_2 >C_1(\frac{p}{n}\vee 1)/4 \mid\bbX_n ) ],
    \eea
and 
    \bea 
    \bm\Omega_{22,(1)}\mid\bbX_n\sim IW_{n-K} \left( (n+\nu_n -2p-2) \bm{\hat\Lambda}_{2,(1)} , n+\nu_n - 2p+2(n-K) \right).
    \eea 
    Let $\bm\Omega_1$ be a random matrix with $\bm\Omega_1 \sim W_{n-K}(n+\nu_n-2p+(n-K)-1,(n+\nu_n-2p+(n-K)-1)^{-1}\bm{I}_{n-K})$. 
    Then, 
    \bea 
      \bm\Omega_{22,(1)}  \equiv 
    \frac{n+\nu_n -2p-2}{n+\nu_n-2p+(n-K)-1}\bm{\hat\Lambda}_{2,(1)} ^{1/2}\bm\Omega_1^{-1}\bm{\hat\Lambda}_{2,(1)} ^{1/2}.
    \eea 
    When $\bm{\hat\Lambda}_{2} \le C_7\Big(\frac{p}{n}\vee 1\Big) $, 
        \bea 
    &&\pi\left(||\bm\Omega_{22,(1)}||_2 >  C_1\left(\frac{p}{n}\vee 1 \right)/4\Bigm|\bbX_n \right)  \\
    &\le& 
    \pi\left( C_7\left(\frac{p}{n}\vee 1  \right) \frac{n+\nu_n -2p-2}{n+\nu_n-2p+(n-K)-1}||\bm\Omega_1^{-1} ||_2 > C_1\left(\frac{p}{n}\vee 1 \right)/4 \Bigm|\bbX_n \right) \\
    &\le& \pi\left( ||\bm\Omega_1^{-1} ||_2 >C_1/(4C_7) \mid\bbX_n \right) \\
    &\le& \pi \left(\lambda_{\min} (\bm\Omega_1) < 4C_7/C_1 \mid\bbX_n \right) \\
    &\le& 2\exp( -C_8n ) ,
    \eea 
    for all sufficiently large $n$, where $C_8$ is a positive constant and the last inequality is satisfied by Lemma B.7 in \cite{lee2018optimal}.

Since $\bm\Omega_{22,(2)}\mid\bbX_n\sim IW_{p-n} \left( (n+\nu_n -2p-2) \bm{\hat\Lambda}_{2,(2)} , n+\nu_n - 2n \right)$, we have 
\bea 
    \frac{n+\nu_n-2p-2}{(\nu_n- p-1)}\hat{\bm\Lambda}_{2,(2)}^{1/2}\bm\Omega_2^{-1} \hat{\bm\Lambda}_{2,(2)}^{1/2} \equiv \bm\Omega_{22,(2)}.
    \eea  with 
    \bea 
    \bm\Omega_2 \sim W_{p-n}\left( (\nu_n-p-1)^{-1} \bm{I}_{p- n},\nu_n-p -1  \right).
    \eea 
    
    Thus,
    \bea 
    &&\pi\left(  ||\bm\Omega_{22,(2)}||_2 >C_1\left(\frac{p}{n}\vee 1 \right)/4 \Bigm|\bbX_n\right) \\
    &\le &\pi\left( ||\bm\Omega_2^{-1} ||_2 > C_1\left(\frac{p}{n}\vee 1 \right)\frac{\nu_n - p-1}{ (n+\nu_n-2p-2) ||\hat{\bm\Lambda}_{2,(2)}||_2 }\Bigm|\bbX_n \right) \\
    &\le &\pi \left(\lambda_{\min} (\bm\Omega_2) <C_1^{-1} \left(\frac{p}{n}\vee 1 \right)^{-1}\frac{ (n+\nu_n-2p-2)||\bm{A}_n||_2 }{\nu_n - p-1}\Bigm|\bbX_n \right).
    \eea 
Since $C_1^{-1} \left(\frac{p}{n}\vee 1 \right)^{-1}\frac{ (n+\nu_n-2p-2)||\bm{A}_n||_2 }{\nu_n - p-1}=o(1)$, by Lemma B.7 in \cite{lee2018optimal}, we have
    \bea 
    && \pi \left(\lambda_{\min} (\bm\Omega_2) <C_1^{-1} \left(\frac{p}{n}\vee 1 \right)^{-1}\frac{ (n+\nu_n-2p-2)||\bm{A}_n||_2 }{\nu_n - p-1}\Bigm|\bbX_n \right)\\
    &\le& 2\exp( - (\nu_n- p-1) ( 1- \sqrt{(p-n)/(\nu_n- p-1)})^2/8),
    \eea 
    which converges to $0$.
    Collecting the inequalities, we obtain $\eqref{eq:corrcond3}\lra 0$.

We also have $\eqref{eq:corrcond2} \to 0$ by applying Lemma \ref{lemma:loweigenvaluevector} with sufficiently large $C$.

Finally, we have 
\bea 
E^\pi(||[\bm\Omega_{21}]_k||^2 \mid\bbX_n) &=& \sum_{j=K+1}^p \textup{E}(\omega_{jk}^2 \mid\bbX_n)\\
&=& \sum_{j=K+1}^p \textup{Var}(\omega_{jk} \mid\bbX_n) \\
&=&\frac{(n - \nu - 2p - 2)\hat{\lambda}_k \sum_{l=K+1}^p \hat{\lambda}_l}{(n + \nu - 2p-1)(n + \nu - 2p - 4)}.
\eea

\end{proof}

\section{Proof of Theorem \ref{thm:BIC}}\label{sec:pf4.1}
We give the proof of Theorem \ref{thm:BIC}.

\begin{proof}[Proof of Theorem \ref{thm:BIC}]

Without loss of generality, we consider the flat non-spiked eigenvalue as $1$, i.e., $\lambda_{0,K_0+1}=\ldots= \lambda_{0,p}=1$.
First, we consider the case $K< K_0$.
We have 
{\small
    \bea 
    \frac{BIC_K - BIC_{K_0}}{n}  &=& - \sum_{k=K+1}^{K_0} \log \hat\lambda_k 
     + (p-K) \log ( w\hat{c}_{K_0} + \sum_{k=K+1}^{K_0}\hat\lambda_k /(p-K)  ) \\
     &&- (p-K_0) \log (\hat{c}_{K_0})
    +\frac{\{(K-K_0)(p+1) -K(K+1)/2 + K_0(K_0+1)/2 \}\log n }{n},
    \eea 
    }
    $w = (p-K_0)/(p-K)$.
Since
\bea 
\log ( w\hat{c}_{K_0} + \sum_{k=K+1}^{K_0}\hat\lambda_k /(p-K)  )  &=&  
\log  (w\hat{c}_{K_0}) + \log ( 1 + \frac{\sum_{k=K+1}^{K_0}\hat\lambda_k} {(p-K) w\hat{c}_{K_0}}  ) 
\\
&=&\log (\hat{c}_{K_0} ) + \log \Big(1 + \frac{K-K_0}{p-K} \Big)\\
&&+  \log \Big( 1+ \sum_{k=K+1}^{K_0}\frac{\hat\lambda_k}{(p-K)w\hat{c}_{K_0}} \Big),
\eea 
we obtain 
\bea
&&\frac{BIC_K - BIC_{K_0}}{n}  \nonumber\\
&=& 
- \sum_{k=K+1}^{K_0} \log \hat\lambda_k 
+ (p-K) \log \Big( 1+ \frac{K-K_0}{p-K} \Big)  \nonumber\\
&&+ (p-K) \log \Big( 1+ \sum_{k=K+1}^{K_0}\frac{\hat\lambda_k}{(p-K)w\hat{c}_{K_0}}  \Big)  + (K_0 -K) \log (\hat{c}_{K_0})  \nonumber \\
&&+\frac{\{(K-K_0)(p+1) -K(K+1)/2 + K_0(K_0+1)/2 \}\log n }{n}  \nonumber\\
&\ge& - \sum_{k=K+1}^{K_0} \log \hat\lambda_k + \frac{\sum_{k=K+1}^{K_0} \hat\lambda_k}{2w\hat{c}_{K_0}} - (K_0 -K) (2-\log (\hat{c}_{K_0}))
-\frac{(K_0-K) (p+1) \log n}{n} \nonumber\\
&\ge&  C_1 \sum_{k=K+1}^{K_0}\hat\lambda_k  ( 1  - C_2 \frac{p\log n}{n\hat\lambda_k } ) +\textup{const.} \nonumber 
\eea 
for some positive constants $C_1$ and $C_2$ and all sufficiently large $n$ and $p$, where the first inequality is satisfied since $\log (1+x) \ge x/2$ and $\log (1-x)\ge -2x$ when $x\in (0,1/2)$, and the second inequality is satisfied since $\hat{c}_{K_0} - \bar{c}_{K_0} = \mathcal{O}_P\left( \frac{1}{\sqrt{n}} \right)$ by Lemma 7 of \cite{yata2012effective}, where $\bar{c}_{K_0} = \sum_{k=K_0+1}^p \lambda_{0,k}/(p-K_0)$.
Thus, $\frac{BIC_K - BIC_{K_0}}{n}  \gtrsim  \sum_{k=K+1}^{K_0}\hat\lambda_k$ when $\frac{p\log n }{n\lambda_{0,K_0}} =o(1)$.

We suppose the case when $K> K_0 $.
We have 
\bean 
&&    \frac{BIC_K - BIC_{K_0}}{n} \nonumber\\ &=&  \sum_{k=K_0+1}^{K} \log \hat\lambda_k 
     + (p-K) \log ( w\hat{c}_{K_0} - \sum_{k=K_0+1}^{K}\hat\lambda_k /(p-K)  ) \nonumber\\
     &&- (p-K_0) \log (\hat{c}_{K_0})\nonumber\\
     &&+\frac{\{(K-K_0)(p+1) -K(K+1)/2 + K_0(K_0+1)/2 \}\log n }{n}\nonumber\\
    &=& \sum_{k=K_0+1}^{K} \log \hat\lambda_k  + (K-K_0) (\frac{(p+1)\log n}{n}- \log (\hat{c}_{K_0}) ) \label{eq:BICdiff1}\\
    &&+(p-K)\log \Big( 1+ \frac{K-K_0}{p-K} \Big)+ (p-K)\log [1  -\sum_{k=K_0+1}^{K}\hat\lambda_k /\{ \hat{c}_{K_0}(p-K_0)\}] \label{eq:BICdiff2}\\
    &&+\frac{\{ -K(K+1)/2 + K_0(K_0+1)/2 \}\log n }{n}\label{eq:BICdiff3},
    \eean 
    where the second equality is satisfied since
\bea 
&&\log ( w\hat{c}_{K_0} - \sum_{k=K_0+1}^{K}\hat\lambda_k /(p-K)  ) \\&=&
      \log (\hat{c}_{K_0}) +\log \Big( 1+ \frac{K-K_0}{p-K} \Big)+\log [1  -\sum_{k=K_0+1}^{K}\hat\lambda_k /\{ \hat{c}_{K_0}(p-K_0)\}] .
\eea

Lemma \ref{lemma:nonspikeeigen} gives, for $k> K_0$
\bea 
\hat\lambda_k \le \hat\lambda_{K_0+1} \le C_3 p/n ,
\eea
with high probability, for some positive constant $C_3$, which gives
\bea  
\sum_{k=K_0+1}^{K}\hat\lambda_k /\{ \hat{c}_{K_0}(p-K_0)\} \le \frac{ C_2 (K-K_0) p}{\hat{c} n(p-K_0)}\to 0 .
\eea 
Thus, 
\bea 
\eqref{eq:BICdiff2}  &\ge& -2 \frac{p-K}{\hat{c}_{K_0}(p-K_0)} \sum_{k=K_0+1}^K \hat\lambda_k  \ge -2C_2 \frac{p-K}{\hat{c}_{K_0} (K-K_0)(p-K_0)} \frac{p}{n} \gtrsim \frac{p(K-K_0)}{n}
\eea 
with high probability.

We have 
\bea 
\hat\lambda_k &\ge& \frac{\hat{c}_{K_0} (p-K_0) - \hat\lambda_{K_0+1} (K-K_0)}{p-K}\\
&\ge& \frac{\hat{c}_{K_0} (p-K_0) - C_3p(K-K_0)/n}{p-K} \\
&\stackrel{p}{\to }& \textup{const.},
\eea 
when $(K-K_0)/n\to 0$, because $\hat{c}_{K_0} \stackrel{p}{\to} \bar{c}_{K_0}$ (Lemma 7 in \cite{yata2012effective}). 
So, we have a lower bound of $\log \hat\lambda_k$ that converges to a constant.
Thus, we have
\bea 
\eqref{eq:BICdiff1} \gtrsim (K-K_0) \frac{p\log n}{n},
\eea 
which dominates the lower bounds of \eqref{eq:BICdiff2} and \eqref{eq:BICdiff3}.

     \end{proof}
\section{Additional Simulation Results}\label{sec:addsims}
\subsection{Estimation of eigenvalues}

We present the results of eigenvalue estimation for the second setting, which follows the design of \citet{wang2017asymptotics} with a slight modification. Specifically, the mean of the Gamma distribution characterizing the idiosyncratic errors is increased. The idiosyncratic covariance matrix is diagonal, \(\bm{\Sigma}_u = \mathrm{diag}(\sigma_1^2, \ldots, \sigma_p^2)\), where each \(\sigma_i\) is independently drawn from a \(\mathrm{Gamma}(a, b)\) distribution with \(a = 150\) and \(b = 100\). Increasing the idiosyncratic error makes it more challenging to estimate the leading eigenvalues and more likely to observe eigenvalue inflation in high-dimensional settings. 
\autoref{tab:eigval2} shows that the proposed IW-PHC and IW-PC estimators deliver stable performance, with low relative errors and high CP values across most settings, even when $n$ is small. In contrast, several methods, particularly IW and SIW, suffer from substantial error inflation and severe coverage deterioration for smaller eigenvalues, with IW exhibiting CP close to zero for \(\lambda_3\) in high-dimensional, low-sample scenarios. The proposed methods improve the performance of the IW estimator for $\lambda_2$ and $\lambda_3$ by correcting the bias of its eigenvalue estimates in the $n=100, p=1000$ case. Similar to the results from the first setting, SPOET generally maintains robust performance but tends to produce conservative confidence intervals that exceed the nominal coverage level. In moderate-dimensional scenarios with large samples ($n=500$), all methods achieve comparable estimation accuracy and CP values close to the nominal level.

\begin{table}[t!]
\begin{center}
\caption{Average relative errors and coverage probabilities (CP) of the estimated eigenvalues over 100 replications under the second setting. NA indicates that the value is not available.}
{\scriptsize
\begin{tabular}{cccccccccccccccc}
\toprule\noalign{}
 & & & \multicolumn{2}{c}{SC} & \multicolumn{2}{c}{IW} &  \multicolumn{2}{c}{SPOET} & \multicolumn{2}{c}{SIW} & \multicolumn{2}{c}{IW-PHC} & \multicolumn{2}{c}{IW-PC} \\
\cmidrule(lr){4-5} \cmidrule(lr){6-7} \cmidrule(lr){8-9} \cmidrule(lr){10-11} \cmidrule(lr){12-13}  \cmidrule(lr){14-15} 
 & $n$ & $p$ & $\text{Err}_\lambda$ & CP & $\text{Err}_\lambda$ & CP & $\text{Err}_\lambda$ & CP  & $\text{Err}_\lambda$ & CP & $\text{Err}_\lambda$ & CP  & $\text{Err}_\lambda$ & CP\\
\midrule\noalign{}
\multirow{4}{*}{$\lambda_1$} & 100 & 500 &  0.1228 & NA & 0.1218 & 0.95 & 0.1256 & 0.94 & 0.1213 & 0.95 & 0.1212 & 0.93 & 0.1214 & 0.93 \\ 
         & 100 & 1000   &  0.1281 & NA & 0.1248 & 0.95 & 0.1286 & 0.94 & 0.2453 & 0.53 & 0.1266 & 0.93 & 0.1264 & 0.93 \\ 
       
 & 500 & 500 & 0.0512 & NA & 0.0511 & 0.93 & 0.0512 & 0.96 & 0.0560 & 0.93 & 0.0509 & 0.93 & 0.0509 & 0.93 \\ 
         & 500 & 1000   & 0.0478 & NA & 0.0526 & 0.94 & 0.0473 & 0.98 & 0.0507 & 0.95 & 0.0482 & 0.94 & 0.0477 & 0.94 \\ 
          
\addlinespace
\multirow{4}{*}{$\lambda_2$} & 100 & 500 & 0.1089 & NA & 0.1072 & 0.93 & 0.1103 & 0.97 & 0.1745 & 0.93 & 0.1117 & 0.92 & 0.1116 & 0.91\\ 
         & 100 & 1000   & 0.1189 & NA & 0.1223 & 0.96 & 0.1197 & 0.96 & 0.6610 & 0.63 & 0.1206 & 0.92 & 0.1185 & 0.90 \\ 
 & 500 & 500 & 0.0458 & NA & 0.0466 & 0.94 & 0.0458 & 0.96 & 0.0460 & 0.97 & 0.0465 & 0.93 & 0.0456 & 0.97 \\ 
         & 500 & 1000   &0.0519 & NA & 0.0528 & 0.93 & 0.0523 & 0.94 & 0.0588 & 0.92 & 0.0544 & 0.91 & 0.0551 & 0.89 \\ 
\addlinespace
\multirow{4}{*}{$\lambda_3$} & 100 & 500 & 0.1294 & NA & 0.1869 & 0.82 & 0.1084 & 0.94 & 0.1662 & 0.89 & 0.1066 & 0.88 & 0.1063 & 0.86 \\ 
          & 100  & 1000   & 0.2105 & NA & 0.5020 & 0.00 & 0.1259 & 0.95 & 0.9119 & 0.69 & 0.1112 & 0.94 & 0.1187 & 0.87 \\ 
 & 500 & 500 & 0.0470 & NA & 0.0521 & 0.94 & 0.0467 & 0.96 & 0.0450 & 0.96 & 0.0463 & 0.97 & 0.0447 & 0.97 \\ 
         & 500 & 1000   & 0.0585 & NA & 0.0877 & 0.78 & 0.0494 & 0.94 & 0.0477 & 0.94 & 0.0479 & 0.95 & 0.0467 & 0.94 \\ 
\bottomrule\noalign{}
\end{tabular}
}
\label{tab:eigval2}
\end{center}
\end{table}

\subsection{Estimation of eigenvectors}
In this subsection, we assess the accuracy of eigenvector estimation using the metric 
\(\text{err}_{\bm{\xi}} := 1 - \left(\bm{\xi}_k(\bm{\Sigma})^\top \bm{\xi}_k(\bm{\Sigma}_0)\right)^2\), 
where \(\bm{\xi}_k(\bm{\Sigma})\) and \(\bm{\xi}_k(\bm{\Sigma}_0)\) denote the estimated and true \(k\)th eigenvectors, respectively. For Bayesian methods, point estimates are obtained by averaging posterior samples.  Since eigenvectors lie on a manifold, their posterior means are computed using the Grassmann average method \citep{hauberg2014grassmann}. 
Coverage probability for eigenvectors is omitted, as it is defined elementwise and thus not directly comparable across different methods.

We adopt the same spiked covariance structures and experimental settings as those used in the main manuscript. The number of observations and variables, as well as the number of spikes, are set identically. For the proposed method, the hyperparameters of the inverse-Wishart prior are also set in the same way as in the main text: \( \bm{A}_n = 0.1 \times I_p \) and \( \nu_n = 2p + 2 \). We compare the proposed method with four existing approaches: sample covariance (SCOV), inverse-Wishart posterior (IW), SPOET \citep{wang2017asymptotics}, and SIW \citep{berger2020bayesian}.  Since IW-PHC yields eigenvectors identical to those derived from the inverse-Wishart (IW) posterior, it is treated as equivalent to IW in our experiments.

For the estimation of eigenvectors, Tables~\ref{tab:eigvec} indicates that, for \(n = 100\), the SCOV, IW (IW-PHC), SPOET, and IW-PC methods yield comparable estimation errors across all eigenvectors. The SIW method exhibits slightly higher accuracy for all eigenvectors across both settings. However, IW-PC shows substantially larger errors for the second eigenvector, particularly when \(n = 100\) and \(p = 500\). In contrast, when \(n = 500\), the estimation of the leading eigenvector is highly stable across all methods, with even the simple sample covariance estimator (SCOV) achieving accuracy comparable to that of more sophisticated approaches. 
In this paper, we establish asymptotic theoretical results for eigenvector estimation, but do not explicitly address the challenges of eigenvector estimation in high-dimensional settings. This remains an important direction for future research.

\begin{table}[t!]
\begin{center}
\caption{Average estimation errors  of the eigenvectors based on 100 simulation replications for the first and second settings with \(n = 100\) and \(n = 500\).}
{\footnotesize
\begin{tabular}{ccccccccc}
\toprule
 & eigenvector & $n$ & $p$ & SCOV & IW & SPOET & SIW & IW-PC \\
\midrule
\multirow{12}{*}{Setting 1} 
 & \multirow{4}{*}{$\bm\xi_1$} & 100 & 500 & 0.1344 & 0.1365 & 0.1348 & 0.1353 & 0.1365 \\ 
 & & 100  & 1000   & 0.1431 & 0.1503 & 0.1433 & 0.1477 & 0.1472 \\ 
 & & 500  & 500   & 0.0198 & 0.0201 & 0.0197 & 0.0196 & 0.0201 \\ 
 & & 500  & 1000   & 0.0243 & 0.0246 & 0.0244 & 0.0243 & 0.0248 \\ 
\addlinespace
 & \multirow{4}{*}{$\bm\xi_2$} & 100 & 500 & 0.1869 & 0.1860 & 0.1891 & 0.1827 & 0.2152 \\ 
 & & 100  & 1000   & 0.1860 & 0.1997 & 0.1862 & 0.1913 & 0.1960 \\ 
 & & 500  & 500   & 0.0270 & 0.0274 & 0.0271 & 0.0269 & 0.0275 \\ 
 & & 500  & 1000   & 0.0349 & 0.0355 & 0.0352 & 0.0349 & 0.0357 \\ 
\addlinespace
 & \multirow{4}{*}{$\bm\xi_3$} & 100 & 500 & 0.1478 & 0.1470 & 0.1516 & 0.1424 & 0.1454 \\ 
 & & 100  & 1000   & 0.2100 & 0.2202 & 0.2127 & 0.2106 & 0.2126 \\ 
 & & 500  & 500   & 0.0272 & 0.0276 & 0.0280 & 0.0273 & 0.0276 \\ 
 & & 500  & 1000   & 0.0470 & 0.0476 & 0.0484 & 0.0471 & 0.0477 \\ 
\addlinespace
\addlinespace
\multirow{12}{*}{Setting 2} 
 & \multirow{4}{*}{$\bm\xi_1$} & 100 & 500 & 0.0067 & 0.0068 & 0.0067 & 0.0077 & 0.0068 \\ 
 & & 100  & 1000   & 0.0122 & 0.0123 & 0.0121 & 0.2244 & 0.0123 \\ 
 & & 500  & 500   & 0.0013 & 0.0013 & 0.0013 & 0.0013 & 0.0013 \\ 
 & & 500  & 1000   & 0.0026 & 0.0026 & 0.0026 & 0.0026 & 0.0026 \\ 
\addlinespace
 & \multirow{4}{*}{$\bm\xi_2$} & 100 & 500 & 0.0364 & 0.0366 & 0.0361 & 0.0505 & 0.0367 \\ 
 & & 100  & 1000   & 0.0636 & 0.0643 & 0.0632 & 0.4222 & 0.0645 \\ 
 & & 500  & 500   & 0.0070 & 0.0070 & 0.0070 & 0.0070 & 0.0071 \\ 
 & & 500  & 1000   & 0.0131 & 0.0132 & 0.0131 & 0.0131 & 0.0132 \\ 
\addlinespace
 & \multirow{4}{*}{$\bm\xi_3$} & 100 & 500 & 0.1150 & 0.1166 & 0.1144 & 0.1313 & 0.1162 \\ 
 & & 100  & 1000   & 0.2012 & 0.2098 & 0.2000 & 0.4066 & 0.2045 \\ 
 & & 500  & 500   & 0.0237 & 0.0239 & 0.0237 & 0.0238 & 0.0239 \\ 
 & & 500  & 1000   & 0.0458 & 0.0463 & 0.0459 & 0.0460 & 0.0462 \\ 
\bottomrule
\end{tabular}
}
\label{tab:eigvec}
\end{center}
\end{table}

\end{document}